\documentclass[12pt]{article}
\title{Dp-finite fields I: infinitesimals and positive characteristic}
\author{Will Johnson}

\usepackage{amsmath, amssymb, amsthm}    	
\usepackage{fullpage} 	
\usepackage{amscd}
\usepackage{hyperref}
\usepackage[all]{xy}
\usepackage{centernot}
\usepackage{titlefoot}

\DeclareMathOperator*{\forkindep}{\raise0.2ex\hbox{\ooalign{\hidewidth$\vert$\hidewidth\cr\raise-0.9ex\hbox{$\smile$}}}}

\newcommand{\rk}{\operatorname{rk}}

\newcommand{\Frac}{\operatorname{Frac}}

\newcommand{\Aut}{\operatorname{Aut}}

\newcommand{\Gr}{\operatorname{Gr}}

\newcommand{\acl}{\operatorname{acl}}
\newcommand{\dcl}{\operatorname{dcl}}
\newcommand{\tp}{\operatorname{tp}}

\newcommand{\val}{\operatorname{val}}

\newcommand{\dpr}{\operatorname{dp-rk}}
\newcommand{\redrk}{\operatorname{rk}_0}
\newcommand{\botrk}{\operatorname{rk}_\bot}
\newcommand{\mininf}{-_\infty}
\newcommand{\cl}{\operatorname{cl}}

\newcommand{\Pow}{\mathcal{P}ow}

\newcommand{\Pro}{\operatorname{Pro}}

\newtheorem{theorem}{Theorem}[section] 
\newtheorem{lemma}[theorem]{Lemma}

\newtheorem{corollary}[theorem]{Corollary}
\newtheorem{fact}[theorem]{Fact}

\newtheorem{question}[theorem]{Question}
\newtheorem{proposition}[theorem]{Proposition}
\newtheorem{proposition-eh}[theorem]{Proposition(?)}
\newtheorem*{theorem-star}{Theorem}
\newtheorem*{conjecture-star}{Conjecture}
\newtheorem*{lemma-star}{Lemma}

\theoremstyle{definition}
\newtheorem{definition}[theorem]{Definition}

\theoremstyle{remark}
\newtheorem{remark}[theorem]{Remark}
\newtheorem{claim}[theorem]{Claim}

\newtheorem*{acknowledgment}{Acknowledgments}
\newtheorem{speculation}[theorem]{Speculative Remark}

\newcommand{\Qq}{\mathbb{Q}}

\newcommand{\Rr}{\mathbb{R}}

\newcommand{\Zz}{\mathbb{Z}}

\newcommand{\Nn}{\mathbb{N}}
\newcommand{\Cc}{\mathcal{C}}

\newcommand{\Mm}{\mathbb{M}}

\newcommand{\Pp}{\mathbb{P}}

\newenvironment{claimproof}[1][\proofname]
               {
                 \proof[#1]
                 
               }
               {
                 \endproof
               }

\begin{document}
\maketitle\unmarkedfntext{

}

\begin{abstract}
  We prove that NIP valued fields of positive characteristic are
  henselian. Furthermore, we partially generalize the known results on
  dp-minimal fields to dp-finite fields. We prove a dichotomy: if $K$
  is a sufficiently saturated dp-finite expansion of a field, then
  either $K$ has finite Morley rank or $K$ has a non-trivial
  $\Aut(K/A)$-invariant valuation ring for a small set $A$. In the
  positive characteristic case, we can even demand that the valuation
  ring is henselian. Using this, we classify the positive
  characteristic dp-finite pure fields.
\end{abstract}

\section{Introduction}
The two main conjectures for NIP fields are
\begin{itemize}
\item The \emph{henselianity conjecture}: any NIP valued field
  $(K,\mathcal{O})$ is henselian.
\item The \emph{Shelah conjecture}: any NIP field $K$ is algebraically
  closed, real closed, finite, or admits a non-trivial henselian
  valuation.
\end{itemize}
By generalizing the arguments used for dp-minimal fields (for example
in Chapter 9 of \cite{myself}), we prove the henselianity conjecture
in positive characteristic, and the Shelah conjecture for positive
characteristic dp-finite fields.  This yields the
positive-characteristic part of the expected classification of
dp-finite fields.

We also make partial progress on dp-finite fields of characteristic
zero.  Let $(K,+,\cdot,\ldots)$ be a sufficiently saturated dp-finite
field, possibly with extra structure.  Then either
\begin{itemize}
\item $K$ has finite Morley rank, or
\item There is an $\Aut(K/A)$-invariant non-trivial valuation ring on
  $K$ for some small set $A$.
\end{itemize}
Unfortunately, we can only prove henselianity of this valuation ring
in positive characteristic.

Following the approach used for dp-minimal fields, there are three
main steps to the proof:
\begin{enumerate}
\item Construct a type-definable group of infinitesimals.
\item Construct a valuation ring from the infinitesimals.
\item Prove henselianity.
\end{enumerate}
We discuss each of these steps, explaining the difficulties that arise
when generalizing from rank 1 to rank $n$.
\subsection{Constructing the infinitesimals}
Mimicking the case of dp-minimal fields, we would like to define the
group $I_M$ of $M$-infinitesimals as
\begin{equation*}
  \bigcap_{X \text{ ``big'' and $M$-definable}} \{\delta \in K ~|~ X
  \cap (X + \delta) \text{ is ``big''}\}
\end{equation*}
for some notion of ``big.''  In the dp-minimal case, ``big'' was
``infinite.''  By analyzing the proof for dp-minimal fields, one can
enumerate a list of desiderata for bigness:
\begin{enumerate}
\item Non-big sets should form an ideal.
\item Bigness should be preserved by affine transformations.
\item \label{definability-condition} Bigness should vary definably in
  families.
\item The universe $K$ should be big.
\item \label{mininf-condition} If $X, Y$ are big, the set $\{\delta ~|~
  X \cap (Y + \delta) \text{ is big}\}$ should be big.
\item Bigness should be coherent on externally definable sets, to the
  extent that:
  \begin{itemize}
  \item If $X$ is $M$-definable and big for a small model $M \preceq
    K$, if $Y$ is $K$-definable, and if $X(M) \subseteq Y$, then $Y$
    is big.
  \item If $X$ is big and $X \subseteq Y_1 \cup \cdots \cup Y_n$ for
    externally definable sets $Y_1, \ldots, Y_n$, then there is a big
    definable subset $X' \subseteq X$ such that $X' \subseteq Y_i$ for
    some $i$.
  \end{itemize}
\end{enumerate}
The intuitive guess is that for a field of dp-rank $n$, ``big'' should
mean ``rank $n$.''  But there is no obvious proof of the definability
condition (\ref{definability-condition}), as noted in \S3.2 of
Sinclair's thesis \cite{sinclair}.  An alternative, silly guess is
that ``big'' should mean ``infinite.'' This fails to work in some of
the simplest examples, such as $(\Cc,+,\cdot,\Rr)$.  However, the
silly guess \emph{nearly} works; the only requirement that can fail is
(\ref{mininf-condition}).

The key insight that led to the present paper was the realization that
in rank 2, any failure of (\ref{mininf-condition}) for the silly
option (``big''=``infinite'') fixes (\ref{definability-condition}) for
the intuitive option (``big''=``rank 2'').  Indeed, if $X, Y$ are
infinite sets but $X \cap (Y + \delta)$ is finite for almost all
$\delta$, then
\begin{itemize}
\item By counting ranks, $X, Y$ must be dp-minimal.
\item The map $X \times Y \to X - Y$ is almost finite-to-one.
\item By a theorem of Pierre Simon \cite{surprise}, ``rank 2'' is
  definable on $X \times Y$.\footnote{He proves this under the
    assumption that the entire theory is dp-minimal, but the proof
    generalizes to products of definable sets of dp-rank 1.}  This
  ensures that ``rank 2'' is definable on $X - Y$.
\item A definable set $D \subseteq K$ has rank 2 if and only if some
  translate of $D$ has rank 2 intersection with $X - Y$.
\end{itemize}
So, in the rank-2 setting, one can first try ``big''=``infinite,'' and
if that fails, take ``big''=``rank 2.''

\begin{remark}
  A prototype of this idea appears in Peter Sinclair's thesis
  \cite{sinclair}.  In his \S3.3, he observes that the machinery of
  infinitesimals goes through when ``rank $n$''=``infinite,'' and
  conjectures that this always holds in the pure field reduct.
\end{remark}

By extending this line of thinking to higher ranks, we obtain a notion
of \emph{heavy} sets satisfying the desired properties.\footnote{In
  later work, it has been shown that the heavy sets are exactly the
  sets of full rank (\cite{prdf2}, Theorem~5.9.2).  However, the proof
  relies on the construction of the infinitesimals, and only works in
  hindsight.} See \S\ref{sec:heavy-light} for details; the technique
is reminiscent of Zilber indecomposability in groups of finite Morley
rank.  Once heavy and light sets are defined, the construction of
infinitesimals is carried out in \S\ref{sec:infinitesimals} via a
direct generalization of the argument for dp-minimal fields.  For
certain non-triviality properties, we need to assume that $K$ does not
have finite Morley rank.  The relevant dichotomy is proven in
\S\ref{sec:likeTT}; it is closely related to Sinclair's Large Sets
Property (Definition 3.0.3 in \cite{sinclair}), but with heaviness
replacing full dp-rank.

For ranks greater than 2, we need to slightly upgrade Simon's results
in \cite{surprise}.  We do this in \S\ref{sec:broad-narrow}.  Say that
an infinite definable set $Q$ is \emph{quasi-minimal} if $\dpr(D) \in
\{0, \dpr(Q)\}$ for every definable subset $D \subseteq Q$.  The main
result is the following:
\begin{theorem-star}
  Let $M$ be an NIP structure eliminating $\exists^\infty$.  Let $Q_1,
  \ldots, Q_n$ be quasi-minimal sets, and $m = \dpr(Q_1 \times \cdots
  \times Q_n)$.  Then ``rank $m$'' is definable in families of
  definable subsets of $Q_1 \times \cdots \times Q_n$.
\end{theorem-star}
This is a variant of Corollary 3.12 in \cite{surprise}.  Note that
dp-minimal sets are quasi-minimal, and quasi-minimal sets are
guaranteed to exist in dp-finite structures.

\subsection{Getting a valuation}
Say that a subring $R \subseteq K$ is a \emph{good Bezout domain} if
$R$ is a Bezout domain with finitely many maximal ideals, and
$\Frac(R) = K$.  This implies that $R$ is a finite intersection of
valuation rings on $K$.  Ideally, we could prove the following
\begin{conjecture-star}
  The $M$-infinitesimals $I_M$ are an ideal in a good Bezout domain $R
  \subseteq K$.
\end{conjecture-star}
Assuming the Conjecture, one can tweak $R$ and arrange for $I_M$ to be
the Jacobson radical of $R$.  This probably implies that
\begin{itemize}
\item The ring $R$ is $\vee$-definable, and so are the associated
  valuation rings.
\item The canonical topology is a field topology, and has a definable
  basis of opens.
\end{itemize}
This would put us in a setting where we could generalize the
infinitesimal-based henselianity proofs from the dp-minimal case,
modulo some technical difficulties in characteristic zero.  The
strategy would be to prove that $R$ is the intersection of just one
valuation ring; this rules out the possibility of $K$ carrying two
definable valuations, leading easily to proofs of the henselianity and
Shelah conjectures.\footnote{The details of this strategy have been
  verified in \cite{prdf2}, which proves the classification of
  dp-finite fields contingent on the above Conjecture.  The
  ``technical difficulties in characteristic zero'' involve showing
  that $1 + I_M$ is exactly the group of ``multiplicative
  infinitesimals.''  This is \cite{prdf2}, Proposition~5.12.}

As it is, the Conjecture is still unknown, even in positive
characteristic.  Nevertheless, we \emph{can} prove that a non-trivial
$\Aut(K/A)$-invariant good Bezout domain $R$ exists for some small set
$A$.  Unfortunately, the connection between $R$ and $I_M$ is too weak
for the henselianity arguments, and we end up using a different
approach, detailed in the next section.

How does one obtain a good Bezout domain?  It is here that we diverge
most drastically from the dp-minimal case.  Say that a type-definable
group $G$ is \emph{00-connected} if $G = G^{00}$.  The group $I_M$ of
$M$-infinitesimals is 00-connected because it is the minimal
$M$-invariant heavy subgroup of $(K,+)$.  In the dp-minimal case (rank
1), the 00-connected subgroups of $(K,+)$ are totally ordered by
inclusion.  Once one has \emph{any} non-trivial 00-connected proper
subgroup $J < (K,+)$, the set $\{x \in K ~|~ x \cdot J \subseteq J\}$
is trivially a valuation ring.

In higher rank, the lattice $\mathcal{P}$ of 00-connected subgroups is
no longer totally ordered.  Nevertheless, dp-rank bounds the
complexity of the lattice.  Specifically, it induces a rank function
$\rk_{dp}(A/B)$ for $A \ge B \in \mathcal{P}$
probably\footnote{One must verify that the basic properties of
  dp-rank, including subadditivity, go through for hyperimaginaries.}
satisfying the following axioms:
\begin{itemize}
\item $\rk_{dp}(A/B)$ is a nonnegative integer, equal to zero iff $A =
  B$.
\item If $A \ge B \ge C$, then $\rk_{dp}(A/C) \le \rk_{dp}(A/B) +
  \rk_{dp}(B/C)$.
\item If $A \ge A' \ge B$, then $\rk_{dp}(A/B) \ge \rk_{dp}(A'/B)$.
\item If $A \ge B' \ge B$, then $\rk_{dp}(A/B) \ge \rk_{dp}(A/B')$.
\item For any $A, B$
  \begin{align*}
    \rk_{dp}(A \vee B/B) & = \rk_{dp}(A/A \wedge B) \\
    \rk_{dp}(A \vee B/A \wedge B) &= \rk_{dp}(A/A \wedge B) + \rk_{dp}(B/A \wedge B).
  \end{align*}
\end{itemize}
This rank constrains the lattice substantially.  For example, if $n =
\dpr(K)$ there can be no ``$(n+1)$-dimensional cubes'' in
$\mathcal{P}$, i.e., no injective homomorphisms of unbounded lattices
from $\Pow([n])$ to $\mathcal{P}$.

Somehow, the rank needs to be leveraged to recover a valuation.  First we need some lattice-theoretic way to recover valuations.
Recall that in a lower-bounded modular lattice $(P,\wedge,\vee,\bot)$,
there is a modular pregeometry on the set of atoms.  More generally,
if $x$ is an element in an unbounded modular lattice
$(P,\wedge,\vee)$, there is a modular pregeometry on the set of
``relative atoms'' over $x$, i.e., elements $y \ge x$ such that the
closed interval $[x,y]$ has size two.  If
$(K,\mathcal{O},\mathfrak{m})$ is a model of ACVF and $P_n$ is the
lattice of definable additive subgroups of $K^n$, then the pregeometry
of relative atoms over $\mathfrak{m}^n \in P_n$ is projective
$(n-1)$-space over the residue field $k$, because the relative atoms
all lie in the interval $[\mathfrak{m}^n,\mathcal{O}^n]$ whose quotient
$\mathcal{O}^n/\mathfrak{m}^n$ is isomorphic to $k^n$.  Furthermore,
the specialization map on Grassmannians
\begin{equation*}
  \Gr(\ell, K^n) \to \Gr(\ell, k^n)
\end{equation*}
is essentially the map sending $V$ to the set of relative atoms below
$V + \mathfrak{m}^n$.

For dp-finite fields, we do something formally similar: we consider
the lattice $\mathcal{P}_n$ of 00-connected type-definable subgroups
of $K^n$, we consider the pregeometry $\mathcal{G}_n$ of relative
atoms over $(I_M)^n$, and we extract a map $f_n$ from the lattice of
$K$-linear subspaces of $K^n$ to the lattice of closed sets in
$\mathcal{G}_n$.  Morally, this leads to a system of maps
\begin{equation}
  \label{n-ell-eq}
  f_{n, \ell} : \Gr(\ell, K^n) \to \Gr(r \cdot \ell, k^n)
\end{equation}
where $r = \dpr(K)$ and $k$ is some skew field.  The maps $f_{n,
  \ell}$ satisfy some compatibility across $n$ and $\ell$.  Call this
sort of configuration an \emph{$r$-fold specialization}.  The hope was
that $r$-fold specializations would be classified by valuations (or
good Bezout domains).  Unfortunately, this turns out to be far from
true.  All we can say is that there is a way of ``mutating'' $r$-fold
specializations such that, in the limit, they give rise to Bezout
domains.  See Remark~\ref{r-fold-specs} and Section \ref{sec:mutate}
for further details.\footnote{Luckily, we can avoid the machinery of
  $r$-fold specializations in the present paper.  We defer a proper
  treatment of $r$-fold specializations until a later paper,
  \cite{prdf3}.}

The above picture is complicated by three technical issues.  First and
most glaringly, the modular lattices we are considering don't really
have enough atoms.  For example, there are usually no minimal non-zero
type-definable subgroups of $K$.  So instead of using atoms, we use
equivalence classes of quasi-atoms.  In a lower-bounded modular
lattice $(P,\wedge,\vee,\bot)$, say that an element $x > \bot$ is a
\emph{quasi-atom} if the interval $(\bot,x]$ is a sublattice, and two
quasi-atoms $x, y$ are \emph{equivalent} if $x \wedge y > \bot$.
There is always a modular geometry on equivalence classes of
quasi-atoms.  In the presence of a subadditive rank such as
$\rk_{dp}(-/-)$, there are always ``enough'' quasi-atoms, and the
modular geometry has bounded rank.  We verify these facts in
\S\ref{sec:geometry}, in case they are not yet known.

Second of all, equation (\ref{n-ell-eq}) says that an
$\ell$-dimensional subspace $V \subseteq K^n$ must map to a closed set
in $\mathcal{G}_n$ of rank exactly $r \cdot \ell$, where $r =
\dpr(K)$.  This only works if the pregeometry of relative atoms over
$I_M$ has the same rank as the dp-rank of the field.  This turns out
to be false in general, necessitating two modifications:
\begin{itemize}
\item Rather than using dp-rank, we use the minimum subadditive rank
  on the lattice.  We call this rank ``reduced rank,'' and verify its
  properties in \S\ref{sec:reduced-rank}.  In equation
  (\ref{n-ell-eq}), $r$ should be the reduced rank rather than the
  dp-rank.
\item Rather than using $I_M$, we use some other group $J$ which
  maximizes the rank of the associated geometry.  We call such $J$
  \emph{special}, and consider their properties in
  \S\ref{sec:special}.
\end{itemize}
With these changes, equation (\ref{n-ell-eq}) can be recovered; see
Lemma~\ref{special-lemma-1}.

Lastly, there is an issue with the lattice meet operation.  Note that
the lattice operations on the lattice of 00-connected type-definable
subgroups are
\begin{align*}
  A \vee B & := A + B \\
  A \wedge B & := (A \cap B)^{00}.
\end{align*}
The 00 in the definition of $A \wedge B$ is a major
annoyance\footnote{Specifically, it would derail the proof of
  Proposition~\ref{special-proposition}.\ref{guard-application}, among other
  things.}; it would be nicer if $A \wedge B$ were simply $A \cap B$.
The following trick clears up this headache:
\begin{theorem-star}
  There is a small submodel $M_0 \preceq K$ such that $J = J^{00}$ for
  every type-definable $M_0$-linear subspace $J \le K$.  Consequently,
  if $\mathcal{P}' \subseteq \mathcal{P}$ is the sublattice of
  00-connected $M_0$-linear subspaces of $K$, then the lattice
  operations on $\mathcal{P}'$ are given by
  \begin{align*}
    A \vee B & = A + B \\
    A \wedge B & = A \cap B.
  \end{align*}
\end{theorem-star}
This is a corollary of the following uniform bounding principle for
dp-finite abelian groups:
\begin{theorem-star}
  If $H$ is a type-definable subgroup of a dp-finite abelian group
  $G$, then $|H/H^{00}|$ is bounded by a cardinal $\kappa(G)$
  depending only on $G$.
\end{theorem-star}
We prove both facts in \S\ref{sec:bounds}.

\subsection{Henselianity}
There is comparatively little to say about henselianity.  In the
dp-minimal case, henselianity of definable valuations was first proven
by Jahnke, Simon, and Walsberg \cite{JSW}.  Independently around the same time,
the author proved henselianity using the machinery of infinitesimals.
The outline of the proof is
\begin{enumerate}
\item Assume non-henselianity
\item Pass to a finite extension and get multiple incomparable
  valuations.
\item Pass to a coarsening and get multiple independent valuations.
\item Consider the multiplicative homomorphism $f(x) = x^2$ in
  characteristic $\ne 2$ or the additive homomorphism $f(x) = x^2 -
  x$.
\item Use strong approximation to produce an element $x$ such that $x
  \notin I'_M$ but $f(x) \in I'_M$.
\item Argue that $f(I_M')$ is strictly smaller than $I_M'$,
  contradicting the minimality of the group of infinitesimals among
  $M$-definable infinite type-definable groups.
\end{enumerate}
In the last two steps, $I_M'$ denotes the multiplicative
infinitesimals $1 + I_M$ or the additive infinitesimals $I_M$ as
appropriate.

The ``strong approximation'' step requires the infinitesimals to be an
intersection of valuation ideals.  As noted in the previous section,
this property is unknown for higher rank, so we are not able to prove
much.

But then a miracle occurs.  In the additive Artin-Schreier case, the
only real properties of $I_M$ being used are that $I_M$ is the
Jacobson radical of a good Bezout domain, and $I_M = I_M^{00}$.  If
$\mathcal{O}_1$ and $\mathcal{O}_2$ are two incomparable valuation
rings on an NIP field, then $R = \mathcal{O}_1 \cap \mathcal{O}_2$ is
a good Bezout domain whose Jacobson radical is easily seen to be
00-connected, provided that the residue fields are infinite.  This
yields an extremely short proof of the henselianity conjecture for NIP
fields in positive characteristic.  See \S\ref{sec:henselianity},
specifically Lemma~\ref{hensel-key}.

This argument does not directly apply to the $A$-invariant valuation
rings constructed using modular lattices.  Nevertheless, a variant can
be made to work leveraging the infinitesimals, utilizing the existence
of $G^{000}$ for type-definable $G$.  See \S\ref{sec:further-hensel}
for details.  Putting everything together, we find that if $K$ is a
sufficiently saturated dp-finite field of positive characteristic,
then either $K$ has finite Morley rank, or $K$ admits a non-trivial
henselian valuation.

This in turn yields a proof of the Shelah conjecture for dp-finite
positive characteristic fields, as well as the expected
classification.  See \S\ref{sec:the-end}.

\subsection{Outline}
In \S\ref{sec:henselianity} we prove the henselianity conjecture in
positive characteristic.  In \S\ref{sec:broad-narrow} we generalize
Pierre Simon's results \cite{surprise} about definability of dp-rank
in dp-minimal theories eliminating $\exists^\infty$.  In
\S\ref{sec:heavy-light}, we apply this to define a notion of ``heavy''
sets which take the place of infinite sets in the construction of the
group of infinitesimals in \S\ref{sec:infinitesimals}.  Section
\ref{sec:likeTT} verifies the additional property of heavy sets that
holds when the field is not of finite Morley rank.  Section
\ref{sec:further-hensel} applies the infinitesimals to the problem of
proving henselianity of $A$-invariant valuation rings in positive
characteristic.  In \S\ref{sec:bounds} we prove the technical fact
that $|H/H^{00}|$ is uniformly bounded as $H$ ranges over
type-definable subgroups of a dp-finite abelian group.  Section
\ref{sec:modular-lats} is a collection of abstract facts about modular
lattices, probably already known to experts.  Specifically, we verify
that ``independence'' and ``cubes'' make sense
(\S\ref{sec:independence}-\ref{sec:cubes}), that there is a minimum
subadditive rank (\S\ref{sec:reduced-rank}), and that there is a
modular pregeometry on quasi-atoms (\S\ref{sec:geometry}).  Then, in
\S\ref{sec:valuations}, we apply the abstract theory to construct a
valuation ring.  Finally, in \S\ref{sec:the-end} we verify the Shelah
conjecture for positive characteristic dp-finite fields, and enumerate
the consequences.

\tableofcontents

\section{Henselianity in positive characteristic} \label{sec:henselianity}
In this section, we prove that definable NIP valuations in positive
characteristic must be henselian.  We also consider the situation of
$\Aut(\Mm/A)$-invariant valuation rings.
\begin{lemma}
  \label{incomparables}
  Let $\mathcal{O}_1$ and $\mathcal{O}_2$ be two incomparable
  valuation rings on a field $K$.  For $i = 1, 2$ let $\mathfrak{m}_i$
  be the maximal ideal of $\mathcal{O}_i$.  Consider the sets
  \begin{align*}
    R &= \mathcal{O}_1 \cap \mathcal{O}_2 \\
    I_1 &= \mathfrak{m}_1 \cap \mathcal{O}_2 \\
    I_2 &= \mathcal{O}_1 \cap \mathfrak{m}_2 \\
    J &= \mathfrak{m}_1 \cap \mathfrak{m}_2.
  \end{align*}
  Then $R$ is a Bezout domain with exactly two maximal ideals $I_1,
  I_2$ and Jacobson radical $J$.  The quotient $R/I_i$ is isomorphic
  to the residue field $\mathcal{O}_i/\mathfrak{m}_i$.  Moreover,
  \begin{equation*}
    (a + \mathfrak{m}_1) \cap (b + \mathfrak{m}_2) \ne \emptyset
  \end{equation*}
  for any $a \in \mathcal{O}_1$ and $b \in \mathcal{O}_2$.
\end{lemma}
(This is mostly or entirely well-known.)
\begin{proof}
  Let $\val_i : K^\times \to \Gamma_i$ denote the valuation associated
  to $\mathcal{O}_i$.  Note that $x|y$ in $R$ if and only if
  $\val_1(x) \le \val_1(y)$ and $\val_2(x) \le \val_2(y)$.
  \begin{claim} \label{3-bezout}
    For any $x, y \in R$, the ideal $(x,y)$ is generated by $x$ or $y$
    or $x - y$.
  \end{claim}
  \begin{claimproof}
    The following cases are exhaustive:
    \begin{itemize}
    \item If $\val_1(x) \le \val_1(y)$ and $\val_2(x) \le \val_2(y)$
      then $(x,y)$ is generated by $x$.
    \item If $\val_1(x) \ge \val_1(y)$ and $\val_2(x) \ge \val_2(y)$,
      then $(x,y)$ is generated by $y$.
    \item If $\val_1(x) < \val_1(y)$ and $\val_2(x) > \val_2(y)$, then
      \begin{align*}
        \val_1(x - y ) = &\val_1(x) < \val_1(y) \\
        \val_2(x - y) = &\val_2(y) < \val_2(x)
      \end{align*}
      so $(x,y)$ is generated by $x- y$.
    \item If $\val_1(x) > \val_1(y)$ and $\val_2(x) < \val_2(y)$, then
      similarly $(x,y)$ is generated by $x - y$. \qedhere
    \end{itemize}
  \end{claimproof}
  By the claim, $R$ is a Bezout domain.
  \begin{claim}
    \label{u-claim}
    There is an element $u \in \mathfrak{m}_1 \cap (1 +
    \mathfrak{m}_2)$.
  \end{claim}
  \begin{claimproof}
    By incomparability, we can find $a \in \mathcal{O}_1 \setminus
    \mathcal{O}_2$ and $b \in \mathcal{O}_2 \setminus \mathcal{O}_1$.
    Then $\val_1(a) \ge 0 > \val_1(b)$ and $\val_2(b) \ge 0 >
    \val_2(a)$, so $a/b \in \mathfrak{m}_1$ and $b/a \in
    \mathfrak{m}_2$.  Let $u = a/(a + b)$.  Then
    \begin{align*}
      u & = \frac{a/b}{a/b + 1} \in \mathfrak{m}_1 \\
      1 - u & = \frac{b/a}{b/a + 1} \in \mathfrak{m}_2 \qedhere
    \end{align*}
  \end{claimproof}
  If $u$ is as in Claim~\ref{u-claim}, then $u \in I_1$ but $u \notin
  I_2$, so $I_1 \ne I_2$.
  \begin{claim} \label{m-i-surj}
    For any $a \in \mathcal{O}_1$, $(a + \mathfrak{m}_1) \cap
    \mathcal{O}_2 \ne \emptyset$.
  \end{claim}
  \begin{claimproof}
    If $a \in \mathfrak{m}_1$ then $0$ is in the intersection.  If $a
    \in \mathcal{O}_2$ then $a$ is in the intersection.  So we may
    assume $a \in \mathcal{O}_1^\times \setminus \mathcal{O}_2$.  Then
    for $u$ as in Claim~\ref{u-claim},
    \begin{equation*}
      (a^{-1} + u)^{-1} \in (a + \mathfrak{m}_1) \cap \mathcal{O}_2.
    \end{equation*}
    Indeed, $a^{-1}$ and $a^{-1} + u$ have the same residue class
    modulo $\mathfrak{m}_1$, so their inverses have the same residue
    class.  And
    \begin{equation*}
      a \notin \mathcal{O}_2 \implies a^{-1} \in \mathfrak{m}_2
      \implies a^{-1} + u \in 1 + \mathfrak{m}_2 \implies (a^{-1} +
      u)^{-1} \in 1 + \mathfrak{m}_2 \subseteq \mathcal{O}_2. \qedhere
    \end{equation*}
  \end{claimproof}
  Claim~\ref{m-i-surj} says that the natural inclusion
  \begin{equation*}
    R/I_1 = (\mathcal{O}_1 \cap \mathcal{O}_2)/(\mathfrak{m}_1 \cap
    \mathcal{O}_2) \hookrightarrow \mathcal{O}_1/\mathfrak{m}_1
  \end{equation*}
  is onto, hence an isomorphism.  Therefore $R/I_1$ is a field and
  $I_1$ is a maximal ideal.  Similarly $I_2$ is a maximal ideal and
  $R/I_2 \cong \mathcal{O}_2/\mathfrak{m}_2$.
  \begin{claim}
    There cannot exist a third maximal ideal $I_3$.
  \end{claim}
  \begin{claimproof}
    Otherwise use the Chinese remainder theorem to find $x \in (1 +
    I_1) \cap (1 + I_2) \cap I_3$ and $y \in I_1 \cap (1 + I_2) \cap
    (1 + I_3)$.  Then
    \begin{align*}
      (x) \subseteq & I_3 \not \supseteq (x,y) \\
      (y) \subseteq & I_1 \not \supseteq (x,y) \\
      (x - y) \subseteq & I_2 \not \supseteq (x,y)
    \end{align*}
    contradicting Claim~\ref{3-bezout}.
  \end{claimproof}
  Consider the composition
  \begin{equation*}
    R/(I_1 \cap I_2) \twoheadrightarrow (R/I_1) \times (R/I_2)
    \stackrel{\sim}{\to} (\mathcal{O}_1/\mathfrak{m}_1) \times
    (\mathcal{O}_2/\mathfrak{m}_2).
  \end{equation*}
  The first map is surjective by the Chinese remainder theorem (as
  $I_1$ and $I_2$ are distinct maximal ideals).  The second map was
  shown to be an isomorphism earlier.  Therefore the composition is
  surjective, which exactly means that for any $a \in \mathcal{O}_1$
  and $b \in \mathcal{O}_2$, the intersection $(a + \mathfrak{m}_1)
  \cap (b + \mathfrak{m}_2)$ is non-empty.
\end{proof}

\begin{lemma}
  \label{hensel-key}
  Let $K$ be a field and $\mathcal{O}_i$ be a valuation ring on $K$
  for $i = 1,2$.  If the structure $(K,\mathcal{O}_1,\mathcal{O}_2)$
  is NIP and $K$ has characteristic $p > 0$, then $\mathcal{O}_1$ and
  $\mathcal{O}_2$ are comparable.
\end{lemma}
\begin{proof}
  Otherwise, let $\mathfrak{m}_1, \mathfrak{m}_2, R, I_1, I_2, J$ be
  as in Lemma~\ref{incomparables}.  The incomparability of the
  $\mathcal{O}_i$ implies that $K$ is infinite, hence Artin-Schreier
  closed (by the Kaplan-Scanlon-Wagner theorem \cite{NIPfields}).  Consider the
  homomorphism
  \begin{equation*}
    f : (J,+) \to (\Zz/p,+)
  \end{equation*}
  defined as follows: given $x \in J = \mathfrak{m}_1 \cap
  \mathfrak{m}_2$, take an Artin-Schreier root $y \in K$ such that
  $y^p - y = x$.  Then $y \in (a + \mathfrak{m}_1) \cap (b +
  \mathfrak{m}_2)$ for unique $a, b \in \Zz/p$.  Define $f(x) = a -
  b$.  This is independent of the choice of the root $y$, and the
  homomorphism $f$ is definable.  By Lemma~\ref{incomparables} there
  is some $y$ such that $y \in (1 + \mathfrak{m}_1) \cap
  \mathfrak{m}_2$; then $f(y^p - y) = 1$.  Therefore $f$ is onto and
  $\ker f$ is a definable subgroup of $J$ of index $p$.  Recall that
  $J^{00}$ exists in NIP theories.  It follows that $J^{00} \subsetneq
  J$.  Now for any $a \in R$, we have
  \begin{equation*}
    a \cdot J^{00} = (a \cdot J)^{00} \subseteq J^{00},
  \end{equation*}
  and so $J^{00}$ is an ideal in $R$.  Choose $b \in J \setminus
  J^{00}$ and let $I = \{x \in R | xb \in J^{00}\}$.  Then $I$ is a
  proper ideal of $R$, so $I \le I_i$ for $i = 1$ or $i = 2$.  The
  maps
  \begin{align*}
    (R \cdot b)/(R \cdot b \cap J^{00}) &\cong R/I \twoheadrightarrow R/I_i \cong \mathcal{O}_i/\mathfrak{m}_i \\
    (R \cdot b)/(R \cdot b \cap J^{00}) &\cong (R \cdot b +
    J^{00})/J^{00} \hookrightarrow J/J^{00}
  \end{align*}
  together show that $|J/J^{00}| \ge |\mathcal{O}_i/\mathfrak{m}_i|$.  But
  the fact that $K$ is Artin-Schreier closed forces the residue field
  $\mathcal{O}_i/\mathfrak{m}_i$ to be Artin-Schreier closed, hence
  infinite, hence unbounded in elementary extensions.  This
  contradicts the definition of $J^{00}$.
\end{proof}

\begin{remark}
  Let $(K,\mathcal{O})$ be a valued field and $L/K$ be a finite normal
  extension.  Every extension of $\mathcal{O}$ to $L$ is definable
  (identifying $L$ with $K^d$ for $d = [L : K]$).
\end{remark}
\begin{proof}
  This essentially follows from Beth implicit definability.  Naming parameters, we may assume that $\mathcal{O}$ is 0-definable
  and $L$ is 0-interpretable.  Recall that $\Aut(L/K)$ acts
  transitively on the set of extensions (essentially because valued
  fields can be amalgamated, i.e., ACVF has quantifier elimination).
  So there are only finitely many extensions to $L$, and it suffices
  to show that at least one extension is definable.  For any formula
  $\phi(x;y)$, the condition ``$\phi(L;a)$ is a valuation ring on $L$
  extending $\mathcal{O}$'' is expressed by a formula $\psi(a)$, so we
  may replace the original $(K,\mathcal{O})$ with an elementary
  extension.  Fix some $\mathcal{O}_L$ on $L$ extending $K$, and pass
  to an elementary extension if necessary to ensure that
  \begin{itemize}
  \item $M^+ := (L,K,\mathcal{O}_L,\mathcal{O})$ is
    $\aleph_0$-saturated.
  \item $M^- := (L,K,\mathcal{O})$ is $\aleph_0$-homogeneous and
    $\aleph_0$-saturated.
  \end{itemize}
  Because there are only finitely many extensions of $\mathcal{O}$ to
  $\mathcal{O}_L$ we can find $a_1, \ldots, a_n, b_1, \ldots, b_m \in
  L$ such that $\mathcal{O}_L$ is the unique extension of
  $\mathcal{O}$ to $L$ containing $A = \{a_1,\ldots,a_n\}$ and
  disjoint from $B = \{b_1, \ldots, b_m\}$.  Now if $\sigma \in
  \Aut(M^-/AB)$, then $\sigma$ preserves $\mathcal{O}$ setwise, and
  therefore moves $\mathcal{O}_L$ to \emph{some} extension
  $\mathcal{O}_L'$ of $\mathcal{O}$.  But $\sigma$ fixes $A$ and $B$
  so $\mathcal{O}_L'$ must still contain $A$ and be disjoint from $B$.
  Thus $\mathcal{O}_L' = \mathcal{O}_L$, so $\Aut(M^-/AB)$ fixes
  $\mathcal{O}_L$ setwise.  By $\aleph_0$-homogeneity, we see that
  \begin{equation*}
    \tp_{M^-}(c/AB) = \tp_{M^-}(c'/AB) \implies (c \in \mathcal{O}_L
    \iff c \in \mathcal{O}_L')
  \end{equation*}
  for $c, c' \in L$.  Then by considering the map of type spaces
  \begin{equation*}
    S_{M^+}^L(AB) \to S_{M^-}^L(AB)
  \end{equation*}
  the clopen set in $S_{M^+}^L(AB)$ corresponding to $\mathcal{O}_L$
  must be the preimage of some subset (necessarily clopen) in
  $S_{M^-}^L(AB)$, implying that $\mathcal{O}_L$ is $AB$-definable in
  $M^-$.  Here we are using $\aleph_0$-saturation to ensure that every
  $M^+$ and $M^-$ type is realized in $L$.
\end{proof}

\begin{theorem} \label{henselianity-conjecture}
  Let $(K,\mathcal{O})$ be an NIP valued field of positive
  characteristic.  Then $(K,\mathcal{O})$ is henselian.
\end{theorem}
\begin{proof}
  Otherwise there is some finite normal extension $L/K$ such that
  $\mathcal{O}$ has multiple extensions to $L$.  If $\mathcal{O}_1,
  \mathcal{O}_2$ are two distinct extensions, then
  $(L,\mathcal{O}_1,\mathcal{O}_2)$ is an NIP 2-valued field.  By
  Lemma~\ref{hensel-key}, $\mathcal{O}_1$ and $\mathcal{O}_2$ must be
  comparable.  But $\Aut(L/K)$ acts transitively on the set of
  extensions, and a finite group cannot act transitively on a poset
  unless all elements are incomparable.
\end{proof}

\begin{speculation}
  The same proof should work when $\mathcal{O}$ is $\vee$-definable,
  or equivalently, $\mathfrak{m}$ is type-definable.
\end{speculation}
We will also need a variant of the above results for invariant
valuation rings.  Here and in what follows, $A$-invariant means
$\Aut(\Mm/A)$-invariant, and we will often use ``invariant'' to mean
``$A$-invariant for some small $A \subseteq \Mm$''.

\begin{definition}
  \label{span-a-valuation}
  A valuation ring $\mathcal{O} \subseteq K$ is \emph{spanned} by a
  set $S$ if
  \begin{itemize}
  \item Every element of $S$ has positive valuation.
  \item For every $x \in K$ of positive valuation, there is $y \in S$
    such that $\val(y) \le \val(x)$.
  \end{itemize}
  In other words, $\val(S)$ is downwards-cofinal in the interval
  $(0,+\infty) \subseteq \Gamma$.  Equivalently, $S$ is a set of
  generators for the maximal ideal of $\mathcal{O}$.
\end{definition}

Recall that $G^{000}$ exists for type-definable abelian $G$ in NIP
theories, by a theorem of Shelah (\cite{shelah-g000}, Theorem 1.12).
\begin{lemma}
  \label{invariant-hensel-key}
  Let $(K,\ldots)$ be a monster NIP field of positive characteristic,
  and let $\mathcal{O}_1, \mathcal{O}_2$ be two invariant valuation
  rings.  Suppose that $\mathcal{O}_1$ and $\mathcal{O}_2$ are both
  spanned by some type-definable subgroup $(I,+) \le (K,+)$ with $I =
  I^{000}$.  Then $\mathcal{O}_1$ and $\mathcal{O}_2$ are comparable.
\end{lemma}
\begin{proof}
  Assume not.  Take a small set $A$ such that $\mathcal{O}_1,
  \mathcal{O}_2$ are $A$-invariant and $I$ is type-definable over $A$.
  Let $\mathfrak{m}_1, \mathfrak{m}_2, R, I_1, I_2, J$ be as in
  Lemma~\ref{incomparables}.  Define $f : (J,+) \to (\Zz/p,+)$ as in
  the proof of Lemma~\ref{hensel-key}, and let $J'$ be the kernel of
  $f$.  As in the proof of Lemma~\ref{hensel-key} $f$ is a surjective
  homomorphism, and so $J'$ is an index p subgroup of $J$.  Moreover,
  $J'$ is $A$-invariant (though not necessarily definable).  We claim
  that $R \cdot I = J$.  Indeed, if $x \in J$, then $\val_1(x) > 0$
  and $\val_2(x) > 0$, so we may find $y, z \in I$ such that
  $\val_1(y) \le \val_1(x)$ and $\val_2(z) \le \val_2(x)$.  As $R$ is
  a Bezout domain, there is $w \in R$ such that $(w) = (y,z)$.  Then
  \begin{align*}
    w|y &\implies \val_1(w) \le \val_1(y) \implies \val_1(w) \le \val_1(x) \\
    w|z &\implies \val_2(w) \le \val_2(z) \implies \val_2(w) \le \val_2(x)
  \end{align*}
  so $w|x$.  This in turn implies that $(x) \subseteq (w) = (y,z)
  \subseteq R \cdot I$.  Thus $J \subseteq R \cdot I$, and the
  converse holds because $I \subseteq \mathfrak{m}_i$ for $i = 1, 2$.

  For any $r \in R$, the group homomorphism
  \begin{equation*}
    I \stackrel{x \mapsto r \cdot x}{\longrightarrow} r \cdot I
    \hookrightarrow J \twoheadrightarrow J/J' \cong \Zz/p
  \end{equation*}
  is $rA$-invariant.  The kernel of this map is $rA$-invariant of
  finite index in $I$, so the kernel must be all of $I$ because $I =
  I^{000}$.  Therefore $r \cdot I \subseteq J'$ for any $r \in R$.  It
  follows that $R \cdot I \subseteq J'$, contradicting the fact that
  $R \cdot I = J \supsetneq J'$.
\end{proof}

\begin{proposition}
  \label{invariant-henselian}
  Let $(K,\ldots)$ be a monster NIP field of positive characteristic,
  and $\mathcal{O}$ be an invariant valuation ring, spanned by some
  type-definable group $I \le (K,+)$ with $I = I^{000}$.  Then
  $\mathcal{O}$ is henselian.
\end{proposition}
\begin{proof}
  Let $\Gamma$ be the value group of $\mathcal{O}$.  Then $\Gamma$ is
  $p$-divisible (because $K$ is Artin-Schreier closed or finite).  For
  any positive $\gamma \in \Gamma$ and positive integer $n$, we can
  find $\varepsilon \in I$ such that
  \begin{equation}
    \label{division-cofinal}
    \val(\varepsilon) \le \frac{\gamma}{p^n} < \frac{\gamma}{n},
  \end{equation}
  because $\gamma/p^n \in \Gamma_{> 0}$.  If henselianity fails, there
  is a finite normal extension $L/K$ such that $\mathcal{O}$ has at
  least two (and at most $[L : K]$) extensions $\mathcal{O}_1$ and
  $\mathcal{O}_2$ to $L$.  Let $A$ be a set such that $\mathcal{O}$
  and $I$ are $A$-invariant and $L$ is interpretable over $A$.  Any
  $\sigma \in \Aut(K/A)$ fixes $\mathcal{O}$ setwise, hence permutes
  the finitely many extensions of $\mathcal{O}$ to $L$.  After adding
  finitely many parameters to $A$, we may assume that $\Aut(K/A)$
  fixes $\mathcal{O}_1$ and $\mathcal{O}_2$.  Thus $\mathcal{O}_1$ and
  $\mathcal{O}_2$ are $A$-invariant.  If $\Gamma, \Gamma_1, \Gamma_2$
  are the value groups of $\mathcal{O}, \mathcal{O}_1, \mathcal{O}_2$
  respectively, then $|\Gamma_i/\Gamma| \le [L : K] < \aleph_0$ for $i
  = 1, 2$.  Thus $\Gamma_i \otimes \Qq = \Gamma \otimes \Qq$.  By
  (\ref{division-cofinal}) it follows that $\mathcal{O}_1$ and
  $\mathcal{O}_2$ are both spanned by $I$.  By
  Lemma~\ref{invariant-hensel-key} the valuation rings $\mathcal{O}_1$
  and $\mathcal{O}_2$ must be comparable, which is absurd.
\end{proof}

\section{Broad and narrow sets} \label{sec:broad-narrow}
Almost everything in this section is a straightforward generalization
of the second half of \cite{surprise}.
\subsection{The general setting}
Let $T$ be a theory, not assumed to eliminate imaginaries.  Work in a
monster model $\Mm$.
\begin{definition}
  Let $X_1, \ldots, X_n$ be definable sets and $Y \subseteq X_1 \times
  \cdots \times X_n$ be type-definable.  Then $Y$ is \emph{broad} if
  there exist $a_{i,j} \in X_i$ for $1 \le i \le n$ and $j \in \Nn$
  such that the following conditions hold:
  \begin{itemize}
  \item For fixed $i$, the $a_{i,j}$ are pairwise distinct.
  \item For any function $\eta : [n] \to \Nn$,
    \begin{equation*}
      (a_{1,\eta(1)},\ldots,a_{n,\eta(n)}) \in Y.
    \end{equation*}
  \end{itemize}
  Otherwise, we say that $Y$ is \emph{narrow}.
\end{definition}
\begin{remark} \label{n-is-1}
  If $n = 1$, then a type-definable set $Y \subseteq X_1$ is broad if
  and only if it is infinite.
\end{remark}
\begin{remark} \label{n-form}
  By compactness and saturation, $Y$ is broad if and only if the
  following holds: for every $m \in \Nn$ there exist $a_{i,j} \in X_i$
  for $1 \le i \le n$ and $1 \le j \le m$ such that
  \begin{itemize}
  \item For fixed $i$, the $a_{i,j}$ are pairwise distinct.
  \item For any function $\eta : [n] \to [m]$,
    \begin{equation*}
      (a_{1,\eta(1)},\ldots,a_{n,\eta(n)}) \in Y.
    \end{equation*}
  \end{itemize}
  An equivalent condition is that for every $m \in \Nn$ there exist
  subsets $S_1, \ldots, S_n$ with $S_i \subseteq X_i$, $|S_i| = m$,
  and $S_1 \times \cdots \times S_n \subseteq Y$.
\end{remark}
\begin{remark}
  \label{type-definability}
  Fix a product of definable sets $\prod_{i = 1}^n X_i$.
  \begin{enumerate}
  \item A partial type $\Sigma(\vec{x})$ on $\prod_{i = 1}^n X_i$ is
    broad if and only if every finite subtype is broad. \label{continuity}
  \item If $\{D_b\}_{b \in Y}$ is a definable family of definable
    subsets of $\prod_{i = 1}^n X_i$, then the set of $b$ such that
    $D_b$ is broad is type-definable. \label{true-type-definability}
  \item Let $A$ be some small set of parameters and $D$ be a definable
    set.  The set of tuples $(a_1, \ldots, a_n, b) \in X_1 \times
    \cdots \times X_n \times D$ such that $\tp(\vec{a}/bA)$ is broad is
    type-definable. \label{pairs}
  \end{enumerate}
\end{remark}
\begin{proof}
  \begin{enumerate}
  \item This follows immediately by compactness.
  \item The statement that $D_b$ is broad is equivalent to the small
    conjunction
    \begin{equation*}
      \bigwedge_{m \in \Nn} \exists a_{1,1}, \ldots, a_{n,m} \left(
      \bigwedge_{i \in [n], 1 \le j < j' \le m} a_{i,j} \ne a_{i,j'}
      \wedge \bigwedge_{\eta : [n] \to [m]}
      (a_{1,\eta(1)},\ldots,a_{n,\eta(m)}) \in D_b \right).
    \end{equation*}
  \item The statement that $\tp(\vec{a}/bA)$ is broad is equivalent to
    the type-definable condition
    \begin{align*}
      \bigwedge_{m \in \Nn} \exists a_{1,1}, \ldots, a_{n,m} \Biggl( &
      \bigwedge_{i \in [n], 1 \le j < j' \le m} a_{i,j} \ne a_{i,j'}    
      \\  & \wedge \bigwedge_{\eta : [n] \to [m]}
      (a_{1,\eta(1)},\ldots,a_{n,\eta(m)}) \equiv_{bA}
      (a_1,\ldots,a_n) \Biggr).
    \end{align*}
    Note that $\vec{a} \equiv_{A\vec{c}} \vec{b}$ is a type-definable
    condition on $(\vec{a},\vec{b},\vec{c})$ for a fixed small set
    $A$, and that in a monster model the type-definable conditions are
    closed under quantification and small conjunctions. \qedhere
  \end{enumerate}
\end{proof}
\begin{lemma}
  Let $A$ be a small set of parameters.  Suppose $X_1, \ldots, X_n$
  are $A$-definable sets and $Y \subseteq X_1 \times \cdots \times
  X_n$ is type-definable over $A$.  Then $Y$ is broad if and only if
  there exists a mutually $A$-indiscernible array $\langle a_{ij}
  \rangle_{i \in [n], j \in \Nn}$ such that for fixed $i$ the $a_{ij}$
  are pairwise distinct elements of $X_i$, and such that for any $\eta
  : [n] \to \Nn$, the tuple
  $(a_{1,\eta(1)},a_{2,\eta(2)},\ldots,a_{n,\eta(n)})$ is an element
  of $Y$.  In other words, the witnesses of broadness can be chosen to
  be mutually $A$-indiscernible.
\end{lemma}
\begin{proof}
  This follows from the fact that we can extract mutually
  indiscernible arrays.
\end{proof}
\begin{proposition} \label{ideal}
  Let $X_1, \ldots, X_n$ be infinite definable sets, and let $Y, Z$ be
  type-definable subsets.
  \begin{enumerate}
  \item If $Y \subseteq Z$ and $Y$ is broad, then $Z$ is broad.
  \item \label{ideal-unions} If $Y \cup Z$ is broad, then $Y$ is broad
    or $Z$ is broad.
  \end{enumerate}
  Equivalently, narrow sets form an ideal:
  \begin{enumerate}
  \item If $Y \subseteq Z$ and $Z$ is narrow, then $Y$ is narrow.
  \item If $Y$ and $Z$ are narrow, then $Y \cup Z$ is narrow.
  \end{enumerate}
\end{proposition}
\begin{proof}
  If $Y$ is broad then $Z$ is broad because the witnesses of broadness of
  $Y$ show broadness of $Z$.  If $Y \cup Z$ is broad, we can take a small
  set $A$ over which the $X_i, Y, Z$ are defined and then find a
  mutually $A$-indiscernible array $\langle a_{i,j} \rangle_{i \in
    [n], j \in \Nn}$ such that the $i$th row is a sequence of distinct
  elements, and such that for every $\eta : [n] \to \Nn$,
  \begin{equation*}
    (a_{1,\eta(1)},\ldots,a_{n,\eta(n)}) \in Y \cup Z.
  \end{equation*}
  Choose some fixed $\eta_0 : [n] \to \Nn$.  Without loss of
  generality (interchanging $Y$ and $Z$),
  \begin{equation*}
    (a_{1,\eta_0(1)},\ldots,a_{n,\eta_0(n)}) \in Y.
  \end{equation*}
  By mutual $A$-indiscernibility, it follows that
  \begin{equation*}
    (a_{1,\eta(1)},\ldots,a_{n,\eta(n)}) \in Y
  \end{equation*}
  for all $\eta$, and so $Y$ is broad.
\end{proof}
\begin{proposition} \label{completion}
  Let $A$ be a small set of parameters, let $X_1,\ldots,X_n$ be
  $A$-definable, and let $Y \subseteq X_1 \times \cdots \times X_n$ be
  type-definable over $A$.  Then $Y$ is broad if and only if
  $\tp(\vec{a}/A)$ is broad for some $\vec{a} \in Y$.
\end{proposition}
\begin{proof}
  If $\tp(\vec{a}/A)$ is broad then $Y$ is broad as it contains (the set
  of realizations of) $\tp(\vec{a}/A)$.  Conversely, suppose $Y$ is
  broad.  Choose a mutually $A$-indiscernible array $\{\alpha_{i,j}\}_{i
    \in [n], j \in \Nn}$ witnessing broadness.  Let $\vec{a} =
  (\alpha_{1,1},\alpha_{2,1},\ldots,\alpha_{n,1})$.  By mutual
  indiscernibility, for any $\eta : [n] \to \Nn$ the $n$-tuple
  $(\alpha_{1,\eta(1)},\ldots,\alpha_{n,\eta(n)})$ realizes
  $\tp(\vec{a}/A)$, so $\tp(\vec{a}/A)$ is broad.
\end{proof}
The trick in the following proof is taken from Theorem 3.10 in
\cite{surprise}.
\begin{lemma} \label{simon-lemma}
  Assume NIP.  Let $X_1, \ldots, X_n$ be infinite definable sets, and
  let $Y \subseteq X_1 \times \cdots \times X_n$ be \emph{definable},
  not just type-definable.  Assume $Y$ is broad and $n \ge 2$.
  \begin{enumerate}
  \item \label{sl-1} There exists some $b \in X_n$ such that the slice
    \begin{equation*}
      \{(a_1,\ldots,a_{n-1}) \in X_1 \times \cdots \times X_{n-1} ~|~
      (a_1,\ldots,a_{n-1},b) \in Y\}
    \end{equation*}
    is broad as a subset of $X_1 \times \cdots \times X_{n-1}$.
  \item \label{sl-2} There exists a broad definable subset $D_{< n}
    \subseteq X_1 \times \cdots \times X_{n-1}$ and an infinite
    definable subset $D_n \subseteq X_n$ such that
    \begin{equation*}
      (D_{< n} \times D_n) \setminus Y
    \end{equation*}
    is a ``hyperplane,'' in the sense that for every $b \in D_n$, the
    definable set
    \begin{equation*}
      \{(a_1,\ldots,a_{n-1}) \in D_{< n} ~|~ (a_1,\ldots,a_{n-1},b)
      \notin Y\}
    \end{equation*}
    is narrow in $X_1 \times \cdots \times X_{n-1}$.
  \end{enumerate}
\end{lemma}
\begin{proof}
  Choose a small set $A$ over which the $X_i$ and $Y$ are definable.
  Choose a mutually indiscernible array $\{a_{i,j}\}_{i \in [n], j \in
    \Nn}$ witnessing broadness of $Y$.  Let $b_j = a_{n,j}$ be the
  elements of the bottom row.  Note that the top $n-1$ rows
  $\{a_{i,j}\}_{i \in [n-1], j \in \Nn}$ form a mutually indiscernbile
  array over $A\vec{b}$.  In particular, if $\vec{e}$ denotes the
  first column
  \begin{equation*}
    \vec{e} = (e_1,\ldots,e_{n-1}) := (a_{1,1},a_{2,1},\ldots,a_{n-1,1}),
  \end{equation*}
  then $\tp(\vec{e}/A\vec{b})$ is broad (within $X_1 \times \cdots
  \times X_{n-1}$).  Note that the $n$-tuple $\vec{e}b_j$ lies in $Y$
  for any $j \in \Nn$.  In particular, $\vec{e}$ lies in the slice
  over $b_1$, so the slice over $b_1$ is broad, proving the first point.

  For the second point, consider sequences $c_1, c_2, \ldots, c_m$
  satisfying the following constraints:
  \begin{enumerate}
  \item The sequence $I = b_1 c_1 b_2 c_2 \cdots b_m c_m b_{m+1} b_{m+2}
    b_{m+3} \cdots$ is $A$-indiscernible.
  \item The type $\tp(\vec{e}/AI)$ is broad.
  \item The tuple $\vec{e}c_j$ is \emph{not} in $Y$ for $j = 1,
    \ldots, m$.
  \end{enumerate}
  There is at least one such sequence, namely, the empty sequence.
  Because of NIP, the first and third conditions imply some absolute
  bound on $m$, so we can find such a sequence with $m$ maximal.  Fix
  such a sequence
  \[ I = b_1 c_1 b_2 c_2 \cdots b_m c_m b_{m+1} b_{m+2} b_{m+3} \cdots \]
  and choose $c_{m+1}$ such that $b_1 c_1 \cdots b_m
  c_m b_{m+1} c_{m+1} b_{m+2} b_{m+3} \cdots$ is $A$-indiscernible.
  \begin{claim}
    If $\vec{\alpha} \in X_1 \times \cdots \times X_{n-1}$ and $\beta
    \in X_n$ satisfy
    \begin{equation*}
      \vec{\alpha} \equiv_{AI} \vec{e},~ \beta \equiv_{AI} c_{m+1},\textrm{ and } 
      (\vec{\alpha},\beta) \notin Y
    \end{equation*}
    then $\tp(\vec{\alpha}/AI\beta)$ is narrow.
  \end{claim}
  \begin{claimproof}
    Applying an automorphism over $AI$, we may assume $\vec{\alpha} =
    \vec{e}$.  If $\tp(\vec{e}/AI\beta)$ is broad, then
    \begin{enumerate}
    \item The sequence $I' = b_1 c_1 b_2 c_2 \cdots b_m c_m b_{m+1}
      \beta b_{m+2} b_{m+3} \cdots$ is $A$-indiscernible, because
      $\beta \equiv_{AI} c_{m+1}$.
    \item The type $\tp(\vec{e}/AI')$ is broad.
    \item The tuple $\vec{e}\beta$ is not in $Y$, and neither are
      $\vec{e}c_j$ for $j = 1,\ldots,m$.
    \end{enumerate}
    This contradicts the maximality of $m$.
  \end{claimproof}
  By Remark~\ref{type-definability}.\ref{pairs} and compactness, there
  must be formulas $\varphi(\vec{x}) \in \tp(\vec{e}/AI)$ and $\psi(y)
  \in \tp(c_{m+1}/AI)$ such that
  \begin{equation*}
    \varphi(\vec{\alpha}) \wedge \psi(\beta) \wedge \left((\vec{\alpha},\beta) \notin Y\right)
  \end{equation*}
  implies narrowness of $\tp(\vec{\alpha}/AI\beta)$.  Let $D_{<n}$ be
  the subset of $X_1 \times \cdots \times X_{n-1}$ cut out by
  $\varphi(\vec{x})$, and let $D_n$ be the subset of $X_n$ cut out by
  $\psi(y)$.  Because the sequence
  \begin{equation*}
    b_1 c_1 b_2 c_2 \cdots b_m c_m b_{m+1} c_{m+1} b_{m+2} b_{m+3} b_{m+4} \cdots
  \end{equation*}
  is $A$-indiscernible and non-constant (as the $b_i$ are distinct),
  it follows that no term is in the algebraic closure of the other
  terms.  In particular
  \begin{equation*}
    c_{m+1} \notin \acl(Ab_1c_1 \ldots b_mc_m b_{m+1} b_{m+2} \cdots) = \acl(AI).
  \end{equation*}
  Therefore $D_n$ is infinite.  And $D_{<n}$ is a broad subset of $X_1
  \times \cdots \times X_{n-1}$ because $\tp(\vec{e}/AI)$ is broad (by
  choice of $c_1,\ldots,c_m$).  It remains to show that if $\beta \in D_n$, then the set
  \begin{equation*}
    \{ \vec{\alpha} \in D_{< n} ~|~ (\vec{\alpha},\beta) \notin Y\}
  \end{equation*}
  is narrow as a subset of $X_1 \times \cdots \times X_{n-1}$.  Indeed,
  the set in question is $AI\beta$ definable, and if $\vec{\alpha}$
  belongs to the set then $\tp(\vec{\alpha}/AI\beta)$ is narrow by
  choice of $\varphi(\vec{x})$ and $\psi(y)$.  By
  Proposition~\ref{completion} the set in question is narrow.
\end{proof}

\begin{theorem} \label{main-characterization}
  Assume NIP.  Let $X_1, \ldots, X_n$ be definable and $Y$ be a
  definable subset of $X_1 \times \cdots \times X_n$.  Then $Y$ is broad
  if and only if there exist infinite subsets $D_i \subseteq X_i$ such
  that $(D_1 \times \cdots \times D_n) \setminus Y$ is a
  ``hyperplane,'' in the sense that for every $b \in D_n$ the set
  \begin{equation*}
    \{(a_1,\ldots,a_{n-1}) \in D_1 \times \cdots \times D_{n-1} ~|~
    (a_1,\ldots,a_{n-1},b) \notin Y\}
  \end{equation*}
  is narrow as a subset of $X_1 \times \cdots \times X_{n-1}$.
\end{theorem}
\begin{proof}
  For the ``if'' direction, first note that $D_1 \times \cdots \times
  D_n$ is broad: for each $i \in [n]$ we can choose $a_{i,1}, a_{i,2},
  \ldots$ to be an arbitrary infinite sequence of distinct elements in
  $D_i$; the array $\{a_{i,j}\}_{i \in [n], j \in \Nn}$ then witnesses
  broadness of $D_1 \times \cdots \times D_n$.  Consider the sets
  \begin{align*}
    Y' &:= (D_1 \times \cdots \times D_n) \cap Y \\
    H &:= (D_1 \times \cdots \times D_n) \setminus Y.
  \end{align*}
  Then $D_1 \times \cdots \times D_n$ is the (disjoint) union of $Y'$
  and $H$, so at least one of $Y'$ and $H$ must be broad.  If $H$ were
  broad, by Lemma \ref{simon-lemma}.\ref{sl-1} it could not be a
  hyperplane.  Therefore $H$ is narrow and $Y'$ is broad, and so $Y$ is
  broad, proving the ``if'' direction.

  We prove the ``only if'' direction by induction on $n$.  For the
  base case $n = 1$, we can take $D_1 = Y$ by Remark~\ref{n-is-1}.
  Next suppose $n > 1$.  By Lemma~\ref{simon-lemma}.\ref{sl-2} there
  exist definable sets $D_{<n} \subseteq X_1 \times \cdots \times
  X_{n-1}$ and $D_n \subseteq X_n$ with $D_{< n}$ broad, $D_n$ infinite,
  and $(D_{<n} \times D_n) \setminus Y$ a hyperplane.  By induction
  there exist infinite definable sets $D_i \subseteq X_i$ for $1 \le i
  < n$ such that the set $H' = (D_1 \times \cdots \times D_{n-1})
  \setminus D_{<n}$ is a hyperplane.  By
  Lemma~\ref{simon-lemma}.\ref{sl-1} the set $H'$ is narrow.  For any
  $b \in D_n$, the two sets
  \begin{align*}
    & \{\vec{a} \in D_1 \times \cdots \times D_{n-1} ~|~ \vec{a} \notin D_{< n} \} = H' \\
    & \{\vec{a} \in D_{< n} ~|~ (\vec{a},b) \notin Y\}
  \end{align*}
  are both narrow (the latter because $(D_{< n} \times D_n) \setminus
  Y$ is a hyperplane).  The union of these two sets contains
  \begin{equation*}
    \{\vec{a} \in D_1 \times \cdots \times D_{n-1} ~|~ (\vec{a},b) \notin
    Y\}
  \end{equation*}
  which must therefore be narrow.  Therefore $(D_1 \times \cdots \times
  D_n) \setminus Y$ is a hyperplane, completing the proof.
\end{proof}

\begin{theorem}
  Assume that $T$ is NIP and eliminates $\exists^\infty$.  Then
  ``broadness is definable in families'' on the product $X_1 \times
  \cdots \times X_n$.  In other words, if $\{Y_b\}_{b \in Z}$ is a
  definable family of definable subsets of $X_1 \times \cdots \times
  X_n$, then the set $\{b \in Z ~|~ Y_b \textrm{ is broad}\}$ is
  definable.
\end{theorem}
In fact, we only need $T$ to eliminate $\exists^\infty$ on the sets
$X_1,\ldots,X_n$.
\begin{proof}
  We proceed by induction on $n$.  The base case $n = 1$ is equivalent
  to elimination of $\exists^\infty$ by Remark~\ref{n-is-1}.  Suppose
  $n > 1$ and $\{Y_b\}_{b \in Z}$ is a definable family of definable
  subsets of $X_1 \times \cdots \times X_n$.  Let $A$ be a set of
  parameters over which everything is defined.  The set $\{b \in Z ~|~
  Y_b \textrm{ is broad}\}$ is type-definable by
  Remark~\ref{type-definability}.\ref{true-type-definability}.  It
  remains to show that the set is also $\vee$-definable: if $Y_{b_0}$
  is broad, then there is some $A$-definable neighborhood $N$ of $b_0$
  such that $Y_b$ is broad for $b \in N$.  Indeed, by
  Theorem~\ref{main-characterization} there exist formulas
  $\phi_i(x,z)$ and elements $c_1,\ldots,c_n$ such that
  \begin{equation*}
    \bigwedge_{i = 1}^n \exists^\infty x \in X_i : \phi_i(x,c_i)
  \end{equation*}
  and
  \begin{align*}
    \forall y \in X_n : \phi_n(x,c_n) \implies \neg \exists^{broad} &
    (x_1,\ldots,x_{n-1}) \in X_1 \times \cdots \times X_{n-1} : \\ &
    \left((x_1,\ldots,x_{n-1},y) \notin Y_{b_0}\right) \wedge \bigwedge_{i = 1}^{n-1}
    \phi_i(x_i,c_i)
  \end{align*}
  where $\exists^{broad} \vec{x} : P(\vec{x})$ means that the set of
  $\vec{x}$ such that $P(\vec{x})$ holds is broad.  (This quantifier is
  eliminated, by induction.)  We take $N$ to be the $A$-definable set of
  $b$ such that
  \begin{align*}
    \exists c_1, \ldots, c_n : & \left( \bigwedge_{i = 1}^n
    \exists^\infty x \in X_i : \phi_i(x,c_i) \right)  \\ 
    &\wedge \Biggl( \forall y \in X_n : \phi_n(x,c_n) \implies \neg
    \exists^{broad} (x_1,\ldots,x_{n-1}) \in X_1 \times \cdots \times
    X_{n-1} : \\ & \left((x_1,\ldots,x_{n-1},y) \notin Y_b\right) \wedge \bigwedge_{i =
      1}^{n-1} \phi_i(x_i,c_i)\Biggr).
  \end{align*}
  If $b \in N$, then the sets $\phi_i(\Mm;c_i)$ show that $Y_b$ is broad
  by Theorem~\ref{main-characterization}.  This proves
  $\vee$-definability of the set of $b$ such that $Y_b$ is broad,
  completing the inductive step and the proof.
\end{proof}
\begin{corollary}
  \label{n-bound}
  Assume $T$ is NIP and eliminates $\exists^\infty$.  Let $X_1,
  \ldots, X_n$ be definable sets and $\{D_b\}_{b \in Y}$ be a
  definable family of subsets of $X_1 \times \cdots \times X_n$.  Then
  there is some constant $m$ depending on the family such that for any
  $b \in Y$, the set $D_b$ is broad if and only if there exist
  $\{a_{i,j}\}_{i \in [n],~j \in [m]}$ such that
  \begin{itemize}
  \item For fixed $i$, the $a_{i,j}$ are pairwise distinct elements of
    $X_i$.
  \item For any $\eta : [n] \to [m]$, the tuple
    $(a_{1,\eta(1)},\ldots,a_{n,\eta(n)})$ belongs to $D_b$.
  \end{itemize}
  Equivalently, there is an $m$ such that for every $b$, $D_b$ is
  broad if and only if there exist finite subsets $S_i \subseteq X_i$
  of cardinality $m$ such that $S_1 \times \cdots \times S_n \subseteq
  D_b$.
\end{corollary}
\begin{proof}
  This follows by compactness and Remark~\ref{n-form}, once we know
  that ``broad'' is a definable condition.
\end{proof}

\subsection{Externally definable sets} \label{sec:extern1}
It is a theorem of Chernikov and Simon that in NIP theories
eliminating $\exists^\infty$, any infinite externally definable set
contains an infinite internally definable set (Corollary~1.12 in
\cite{edsdp}).  As a consequence, if an infinite (internally)
definable set $Y$ is covered by finitely many externally definable
sets $D_1, \ldots, D_\ell$, then one of the $D_i$ contains an infinite
internally definable set.

The following lemma is an analogue for broad sets.
\begin{lemma}\label{chernikov-1}
  Assume $T$ is NIP and eliminates $\exists^\infty$.  Let $M \preceq
  \Mm$ be a small model.  Let $X_1, \ldots, X_n$ be $M$-definable
  infinite sets and $Y \subseteq X_1 \times \cdots \times X_n$ be an
  $M$-definable broad set.  Let $D_1, \ldots, D_\ell$ be
  $\Mm$-definable subsets of $X_1 \times \cdots \times X_n$ such that
  \begin{equation*}
    Y(M) \subseteq \bigcup_{k = 1}^\ell D_k.
  \end{equation*}
  Then there exists some $k$ and some $M$-definable broad set $Y'
  \subseteq Y$ such that $Y'(M) \subseteq D_k$.
\end{lemma}
\begin{proof}
  For every $m \in \Nn$ we can find $\{a_{i,j}\}_{i \in [n], j \in
    [m]}$ satisfying the following conditions:
  \begin{itemize}
  \item $a_{i,j} \in X_i$
  \item $a_{i,j} \ne a_{i,j'}$ for $j \ne j'$.
  \item For any $\eta : [n] \to [m]$ the tuple
    $(a_{1,\eta(1)},\ldots,a_{n,\eta(n)})$ lies in $Y$.
  \end{itemize}
  Because $X_i$ and $Y$ are $M$-definable, we can even choose the
  $a_{i,j}$ to lie in $M$.  By the Ramsey-theoretic statement
  underlying Proposition~\ref{ideal}.\ref{ideal-unions}, there is some
  fixed $k \in [\ell]$ such that for every $m \in \Nn$ we can find
  $\{a_{i,j}\}_{i \in [n], j \in [m]}$ satisfying the following
  conditions:
  \begin{itemize}
  \item $a_{i,j} \in X_i(M)$
  \item $a_{i,j} \ne a_{i,j'}$ for $j \ne j'$.
  \item For any $\eta : [n] \to [m]$ the tuple
    $(a_{1,\eta(1)},\ldots,a_{n,\eta(n)})$ lies in $Y \cap D_k$.
  \end{itemize}
  By honest definitions (\cite{NIPguide}, Remark~3.14), the externally
  definable set $Y(M) \cap D_k$ can be approximated by internally
  definable sets in the following sense: there is an $M$-definable
  family $\{F_b\}$ such that for every finite subset $S \subseteq Y(M)
  \cap D_k$ there is a $b \in \dcl(M)$ such that
  \begin{equation*}
    S \subseteq F_b(M) \subseteq Y(M) \cap D_k.
  \end{equation*}
  Take $m$ as in Corollary~\ref{n-bound} for the family $\{F_b\}$.
  Take $a_{i,j}$ for $i \in [n]$ and $j \in [m]$ such that
  \begin{itemize}
  \item $a_{i,j} \in X_i(M)$
  \item $a_{i,j} \ne a_{i,j'}$ for $j \ne j'$.
  \item For any $\eta : [n] \to [m]$ the tuple
    $(a_{1,\eta(1)},\ldots,a_{n,\eta(n)})$ lies in $Y \cap D_k$ (hence
    in $Y(M) \cap D_k$).
  \end{itemize}
  Let $S$ be the finite subset of tuples of the form
  $(a_{1,\eta(1)},\ldots,a_{n,\eta(n)})$.  Take $b \in \dcl(M)$ such
  that $S \subseteq F_b(M) \subseteq Y(M) \cap D_k$.  The $a_{i,j}$
  show that $F_b$ is broad, by choice of $m$.  Let $Y' = F_b$.  The
  fact that $Y'(M) \subseteq Y(M)$ implies that $Y' \subseteq Y$, as
  $Y$ and $Y'$ are both $M$-definable and $M \preceq \Mm$.
\end{proof}

\begin{lemma}
  \label{something-similar-1}
  Let $M \preceq \Mm$ be a small model and $X_1, \ldots, X_n$ be infinite
  $M$-definable sets.  Suppose that $Y \subseteq X_1 \times \cdots
  \times X_n$ is broad and $M$-definable, and that $W \subseteq X_1
  \times \cdots \times X_n$ is $\Mm$-definable.  If $Y(M) \subseteq
  W$, then $W$ is broad.
\end{lemma}
\begin{proof}
  For every $m$, we can find $\{a_{i,j}\}_{i \in [n], j \in [m]}$
  satisfying the following conditions:
  \begin{enumerate}
  \item $a_{i,j} \in X_i$.
  \item $a_{i,j} \ne a_{i,j'}$ for $j \ne j'$.
  \item For any $\eta : [n] \to [m]$ the tuple
    $(a_{1,\eta(1)},\ldots,a_{n,\eta(n)})$ lies in $Y$.
  \end{enumerate}
  The requirements on the $a_{i,j}$ are $M$-definable, and $[n] \times
  [m]$ is finite.  Therefore, for any $m$ we can choose the $a_{i,j}$
  to lie in $M$.  Having done so,
  \begin{equation*}
    (a_{1,\eta(1)},\ldots,a_{n,\eta(n)}) \in Y(M) \subseteq W
  \end{equation*}
  for any $\eta$.  The existence of the $a_{i,j}$ for all $m$ imply
  that $W$ is broad.
\end{proof}

\subsection{The finite rank setting}
We now turn to proving Theorem~\ref{finite-rank-case}, which relates
broadness and narrowness to dp-rank, under certain assumptions.
\begin{lemma}
  \label{narrow-confusion}
  Let $X_1, \ldots, X_n$ be $A$-definable sets, and let
  $(a_1,\ldots,a_n)$ be a tuple in $X_1 \times \cdots \times X_n$.
  Suppose that there exists a sequence $b_1, b_2, \ldots \in X_n$ of
  pairwise distinct elements such that
  \begin{equation*}
    (a_1,\ldots,a_{n-1},b_i) \equiv_A (a_1,\ldots,a_{n-1},a_n)
  \end{equation*}
  for every $i$, and such that $\tp(a_1,\ldots,a_{n-1}/A\vec{b})$ is
  broad.  Then $\tp(a_1,\ldots,a_n/A)$ is broad.
\end{lemma}
\begin{proof}
  Let $\{e_{i,j}\}_{i \in [n-1], j \in \Nn}$ witness broadness of
  $\tp(a_1,\ldots,a_{n-1}/A\vec{b})$.  For any $\eta : [n-1] \to \Nn$,
  \begin{equation*}
    (e_{1,\eta(1)},\ldots,e_{n-1,\eta(n-1)}) \equiv_{A \vec{b}} (a_1,\ldots,a_{n-1}).
  \end{equation*}
  Thus, for any $\eta : [n-1] \to \Nn$ and any $j$,
  \begin{align*}
    (e_{1,\eta(1)},\ldots,e_{n-1,\eta(n-1)},b_j) & \equiv_A
    (a_1,\ldots,a_{n-1},b_j) \\ & \equiv_A (a_1,\ldots,a_{n-1},a_n).
  \end{align*}
  If we set $e_{n,j} := b_j$, then the $e_{i,j}$ show that
  $\tp(\vec{a}/A)$ is broad.
\end{proof}

\begin{remark} \label{faux-independence}
  Let $X_1, X_2, \ldots, X_n$ be $A$-definable sets of finite dp-ranks
  $r_1, \ldots, r_n$.  Let $(a_1,a_2,\ldots,a_n)$ be a tuple in $X_1
  \times \cdots \times X_n$.  Then
  \begin{equation*}
    \dpr(a_1,\ldots,a_n/A) \le r_1 + \cdots + r_n
  \end{equation*}
  by subadditivity of dp-rank, and if equality holds then
  \begin{equation*}
    \dpr(a_i/Aa_1a_2 \cdots a_{i-1} a_{i+1} \cdots a_n) = r_i
  \end{equation*}
  for every $i$.  Indeed, otherwise
  \begin{equation*}
    \dpr(a_i/Aa_1a_2 \cdots a_{i-1} a_{i+1} \cdots a_n) < r_i
  \end{equation*}
  and so
 \begin{align*}
   r_1 + \cdots + r_n & \le \dpr(a_i/Aa_1\cdots a_{i-1}a_{i+1}\cdots
   a_n) + \dpr(a_1,\ldots,a_{i-1}, a_{i+1}, \ldots, a_n/A) \\ & < r_i + (r_1
   + \cdots + r_{i-1} + r_{i+1} + \cdots + r_n)
 \end{align*}
  which is absurd.
\end{remark}

The technique for the next proof comes from Proposition 3.4 in
\cite{surprise}.
\begin{lemma}
  \label{mass-confusion}
  Let $X, Y$ be infinite $A$-definable sets of finite dp-rank $n$ and
  $m$, respectively.  Let $(a,b)$ be a tuple in $X \times Y$ with
  $\dpr(a,b/A) = n + m$.  Then there exist pairwise distinct $b_1,
  b_2, \ldots$ such that $ab_i \equiv_A ab$ and such that $\dpr(a/Ab_1
  \ldots, b_\ell) = n$ for every $\ell$.
\end{lemma}
\begin{proof}
  Take an ict-pattern of depth $n + m$ in $\tp(ab/A)$, extract a
  mutually $A$-indiscernible ict-pattern, and extend the pattern to
  have columns indexed by $\Nn \times \Nn$ (with lexicographic order).
  This yields formula $\phi_i(x,y;z)$ for $1 \le i \le n + m$ and
  elements $c_{i,j,k}$ for $1 \le i \le n + m$ and $j, k \in \Nn$ such
  that $c_{i;j,k}$ is a mutually $A$-indiscernible array, and such
  that
  \begin{equation*}
    \bigwedge_{i = 1}^{n + m} \left( \phi_i(x,y;c_{i;0,0}) \wedge
    \bigwedge_{(j,k) \ne (0,0)} \neg \phi_i(x,y;c_{i;j,k}) \right)
  \end{equation*}
  is consistent with $\tp(a,b/A)$.  Moving the $c$'s by an
  automorphism over $A$, we may assume that $(a,b)$ realizes this
  type, so that $\phi_i(a,b;c_{i;j,k})$ holds if and only if $(j,k) =
  (0,0)$.  By the proof of subadditivity of dp-rank, there exist $m$
  rows which form a mutually $Aa$-indiscernible array.  Without loss
  of generality, these are the rows $i = n + 1, \ldots, n + m$.  For
  $j \in \Nn$ let $d_j$ be an enumeration of $\{ c_{i;j,k} \}_{n < i
    \le n+m,~ k \in \Nn}$.  The $d_j$ form an indiscernible sequence
  over $Aa$, so we may choose $b_j$ such that
  \begin{equation*}
    b_jd_j \equiv_{Aa} bd_0
  \end{equation*}
  and such that $b_0 = b$.  In particular, $ab_j \equiv_A ab$ for any
  $j$.
  \begin{claim}
    For any $\ell > 0$ the tuple $(a,b_0,b_1,\ldots,b_{\ell - 1})$ has
    dp-rank $n + \ell m$ over $A$.
  \end{claim}
  \begin{claimproof}
    The dp-rank is at most $n + \ell m$ by subadditivity of dp-rank
    (as $b_j \equiv_A b \implies b_j \in Y$).  It thus suffices to
    exhibit $(n + \ell m)$-many mutually $A$-indiscernible sequences,
    none of which are indiscernible over $A a b_0 b_1 \ldots b_{\ell -
      1}$:
    \begin{itemize}
    \item For $1 \le i \le n$, the sequence
      \begin{equation*}
        c_{i,0,0}, c_{i,0,1}, c_{i,0,2}, \ldots,
      \end{equation*}
      which fails to be indiscernible over $ab_0 = ab$ on account of
      the fact that
      \begin{equation*}
        \phi_i(a,b,c_{i,0,k}) \iff k = 0.
      \end{equation*}
    \item For $n < i \le n + m$ and for $0 \le j < \ell$, the sequence
      \begin{equation*}
        c_{i,j,0}, c_{i,j,1}, c_{i,j,2}, \ldots
      \end{equation*}
      which fails to be indiscernible over $ab_j$ on account of the
      fact that
      \begin{equation*}
        \phi_i(a,b_j,c_{i,j,k}) \iff \phi_i(a,b,c_{i,0,k}) \iff k = 0.
      \end{equation*}
    \end{itemize}
    These sequences are indeed mutually $A$-indiscernible, because we
    can split the mutually indiscernible array $c_{i;j,k}$ into a
    mutually indiscernible array $c_{i,j;k}$.
  \end{claimproof}
  Now by Remark~\ref{faux-independence}, it follows that
  \begin{equation*}
    \dpr(a/Ab_0 b_1 \cdots b_{\ell - 1}) = n
  \end{equation*}
  for each $\ell$.  Moreover,
  \begin{equation*}
    \dpr(b_\ell/Aa b_0 b_1 \cdots b_{\ell - 1}) = m > 0,
  \end{equation*}
  implying that $b_\ell \notin \acl(A a b_0 \cdots b_{\ell - 1})$.  In
  particular, $b_\ell \ne b_j$ for $j < \ell$, and the $b_j$ are
  pairwise distinct.
\end{proof}

\begin{proposition} \label{halfway-there}
  For $i = 1, \ldots, n$ let $X_i$ be a definable set of finite
  dp-rank $r_i > 0$.  Let $Y \subseteq X_1 \times \cdots \times X_n$
  be type-definable.  If $\dpr(Y) = r_1 + \cdots + r_n$ then $Y$ is
  broad.
\end{proposition}
\begin{proof}
  We proceed by induction on $n$.  Assume $\dpr(Y) = r_1 + \cdots +
  r_n$.  For the base case $n = 1$, we see that $\dpr(Y) = r_1 > 0$,
  so $Y$ is infinite.  By Remark~\ref{n-is-1}, $Y$ is broad.  Suppose $n
  > 1$.  Let $A$ be a small set of parameters such that the $X_i$ are
  definable over $A$ and $Y$ is type-definable over $A$.  Take
  $(a_1,\ldots,a_n) \in Y$ such that $\dpr(a_1,\ldots,a_n/A) = r_1 +
  \cdots + r_n$.  By Lemma~\ref{mass-confusion} applied to
  $(a_1,\ldots,a_{n-1};a_n)$, there exist pairwise distinct $b_1, b_2,
  \ldots$ such that
  \begin{equation*}
    (a_1,\ldots,a_{n-1},a_n) \equiv_A (a_1,\ldots,a_{n-1},b_j)
  \end{equation*}
  for every $j$ and such that
  \begin{equation*}
    \dpr(a_1,\ldots,a_{n-1}/Ab_1\ldots b_j) = r_1 + \cdots + r_{n-1}
  \end{equation*}
  for every $j$.  By induction, $\tp(a_1,\ldots,a_{n-1}/Ab_1 \ldots
  b_j)$ is broad for every $j$.  By
  Remark~\ref{type-definability}.\ref{continuity}, it follows that
  $\tp(a_1,\ldots,a_{n-1}/Ab_1 b_2 \cdots)$ is broad.  By
  Lemma~\ref{narrow-confusion}, it follows that $\tp(a_1,\ldots,a_n/A)$
  is broad, and so $Y$ is broad.
\end{proof}

\begin{definition}
  A definable set $D$ is \emph{quasi-minimal} if $D$ has finite
  dp-rank $n > 0$, and every definable subset $D' \subseteq D$ has
  dp-rank $0$ or $n$.  Equivalently, every infinite definable subset
  of $D$ has the same (finite) dp-rank as $D$.
\end{definition}
\begin{remark}
  If $D$ is a definable set of dp-rank 1, then $D$ is quasi-minimal.
\end{remark}
\begin{remark} \label{they-exist}
  If $D$ is an infinite definable set of finite dp-rank, then $D$ contains a
  quasi-minimal definable subset.
\end{remark}
\begin{theorem} \label{finite-rank-case}
  Assume NIP.  Let $X_1, \ldots, X_n$ be quasi-minimal definable sets
  of rank $r_1, \ldots, r_n$, and let $Y \subseteq X_1 \times \cdots
  \times X_n$ be definable.  Then $Y$ is broad if and only if $\dpr(Y) =
  r_1 + \cdots + r_n$.
\end{theorem}
\begin{proof}
  Note that the $r_i > 0$ by definition of quasi-minimal.  The ``if''
  direction then follows by Proposition~\ref{halfway-there}.  We prove
  the ``only if'' direction by induction on $n$.  For $n = 1$, $Y$ is
  broad if and only if $Y$ is infinite (by Remark~\ref{n-is-1}) if and
  only if $\dpr(Y) = r_1$ (by definition of quasi-minimal).  Assume $n
  > 1$.  Let $Y$ be broad.  By Theorem~\ref{main-characterization},
  there are infinite definable subsets $D_1 \times \cdots \times D_n$
  such that for every $b \in D_n$, the set
  \begin{equation*}
    H_b := \{(a_1,\ldots,a_{n-1}) \in D_1 \times \cdots \times D_{n-1} ~|~
    (a_1,\ldots,a_{n-1},b) \notin Y\}
  \end{equation*}
  is narrow (as a subset of $X_1 \times \cdots \times X_{n-1}$).  Note
  that $\dpr(D_i) = r_i$ by quasi-minimality.  Let
  \begin{align*}
    H &= (D_1 \times \cdots \times D_n) \setminus Y = \coprod_{b \in D_n} H_b \\
    Y' &= (D_1 \times \cdots \times D_n) \cap Y.
  \end{align*}
  By induction, $\dpr(H_b) < r_1 + \cdots + r_{n-1}$ for every $b \in
  D_n$.  By subadditivity of the dp-rank, it follows that
  \begin{equation*}
    \dpr(H) < r_1 + \cdots + r_{n-1} + \dpr(D_n) = r_1 + \cdots + r_n.
  \end{equation*}
  Now $D_1 \times \cdots \times D_n = H \cup Y'$, so
  \begin{equation*}
    r_1 + \cdots + r_n = \dpr(H \cup Y') = \max(\dpr(H),\dpr(Y')).
  \end{equation*}
  Given the low rank of $H$, this forces $Y'$ to have rank $r_1 +
  \cdots + r_n$.  Then
  \begin{equation*}
    r_1 + \cdots + r_n = \dpr(Y') \le \dpr(Y) \le \dpr(X_1 \times
    \cdots \times X_n) = r_1 + \cdots + r_n
  \end{equation*}
  completing the inductive step.
\end{proof}
Combining Theorem~\ref{finite-rank-case} with Corollary~\ref{n-bound}
yields the following:
\begin{corollary} \label{finite-rank-final}
  Assume $T$ is NIP and eliminates $\exists^\infty$, and let $X_1,
  \ldots, X_n$ be quasi-minimal sets of finite dp-rank.  Let $r =
  \dpr(X_1 \times \cdots \times X_n)$.  Given a definable family
  $\{D_b\}_{b \in Y}$ of subsets of $X_1 \times \cdots \times X_n$,
  the set of $b$ such that $\dpr(D_b) = r$ is definable.  In fact,
  there is some $m$ depending on the family such that $\dpr(D_b) = r$
  if and only if there exist finite subsets $S_i \subseteq X_i$ of
  cardinality $m$ such that $S_1 \times \cdots \times S_n \subseteq
  D_b$.
\end{corollary}

\section{Heavy and light sets} \label{sec:heavy-light}
In this section, we assume that $\Mm$ is a dp-finite field, possibly
with additional structure.  The goal is to define a notion of
``heavy'' and ``light'' sets, and prove Theorem~\ref{heavy-light},
which verifies that heaviness satisfies the properties needed for the
construction of infinitesimals.

The heavy sets turn out to be exactly the definable sets of full rank,
but this will not be proven until a later paper (\cite{prdf2},
Theorem~5.9.2).

By the assumption of finite dp-rank, $T$ is NIP.  Moreover, $T$
eliminates $\exists^\infty$ by a theorem of Dolich and Goodrick
(\cite{dolich-goodrick-strong-oags}, Corollary 2.2).  Therefore, the
results of the previous section apply.

\subsection{Coordinate configurations}
\begin{definition}
  A \emph{coordinate configuration} is a tuple $(X_1,\ldots,X_n,P)$
  where $X_i \subseteq \Mm$ are quasi-minimal definable sets, where $P
  \subseteq X_1 \times \cdots \times X_n$ is a broad definable set,
  and where the map
  \begin{align*}
    P & \to \Mm \\
    (x_1,\ldots,x_n) & \mapsto x_1 + \cdots + x_n
  \end{align*}
  has finite fibers.  The \emph{target} of the coordinate
  configuration is the image of this map.  The \emph{rank} of the
  coordinate configuration is the sum $\sum_{i = 1}^n \dpr(X_i)$.
\end{definition}
\begin{remark}
  If $(X_1,\ldots,X_n,P)$ is a coordinate configuration with target
  $Y$ and rank $r$, then
  \begin{equation*}
    \dpr(Y) = \dpr(P) = \dpr(X_1 \times \cdots \times X_n) = r
  \end{equation*}
  by Theorem~\ref{finite-rank-case} and the subadditivity of dp-rank.
\end{remark}

\begin{proposition}
  \label{basic-properties}
  Let $(X_1,\ldots,X_n,P)$ be a coordinate configuration with target
  $Y$.
  \begin{enumerate}
  \item \label{bp-1} Full dp-rank is definable on $Y$: if $\{D_b\}$ is
    a definable family of subsets of $Y$, the set of $b$ such that
    $\dpr(D_b) = \dpr(Y)$ is definable. \label{full-rank-def}
  \item If $Y'$ is a definable subset of $Y$ and $\dpr(Y') = \dpr(Y)$,
    then $Y'$ is the target of some coordinate configuration.
  \end{enumerate}
\end{proposition}
\begin{proof}
  Let $r$ be the rank of the coordinate configuration, or
  equivalently, $r = \dpr(Y)$.  Let $\pi : P \to Y$ be the map
  $\pi(x_1,\ldots,x_n) = \sum_{i = 1}^n x_i$.  Because $\pi$ is
  surjective with finite fibers, a subset $D_b \subseteq Y$ has rank
  $r$ if and only if $\pi^{-1}(D_b)$ has rank $r$.  Therefore, the
  definability of rank $r$ on $Y$ follows from the definability of
  rank $r$ on $X_1 \times \cdots \times X_n$, which is
  Corollary~\ref{finite-rank-final}.  If $Y'$ is a full-rank subset of
  $Y$, then $\pi^{-1}(Y)$ is a broad subset of $X_1 \times \cdots
  \times X_n$, and so $(X_1,\ldots,X_n,\pi^{-1}(Y))$ is a coordinate
  configuration with target $Y'$.
\end{proof}
\begin{lemma} \label{upgrades}
  Let $X_1, \ldots, X_n$ be quasi-minimal and $A$-definable.  Let
  $(a_1,\ldots,a_n)$ be an element of $X_1 \times \cdots \times X_n$
  such that $\tp(a_1,\ldots,a_n/A)$ is broad.  Let $s = a_1 + \cdots +
  a_n$.  If $\vec{a} \in \acl(sA)$, then there is a broad set $P
  \subseteq X_1 \times \cdots \times X_n$ containing $\vec{a}$ such
  that $(X_1,X_2,\ldots,X_n,P)$ is a coordinate configuration with
  target containing $s$.
\end{lemma}
\begin{proof}
  We can find some formula $\phi(x_1,\ldots,x_n,y)$ (with suppressed
  parameters from $A$) such that $\phi(a_1,\ldots,a_n,s)$ holds, and
  $\phi(\Mm,s)$ is finite of size $k$.  Strengthening $\phi$, we may
  assume that $\phi(\Mm,s')$ has size at most $k$ for any $s'$.  We
  may also assume that $\phi(x_1,\ldots,x_n,y) \implies y = x_1 +
  \cdots + x_n$.  Let $P$ be the definable set
  \begin{equation*}
    P = \{(x_1,\ldots,x_n) \in X_1 \times \cdots \times X_n ~|~
    \phi(x_1,\ldots,x_n,x_1 + \cdots + x_n)\}.
  \end{equation*}
  Then $P$ is $A$-definable, and broad by virtue of containing
  $\vec{a}$.  The map
  \begin{align*}
    P &\to \Mm \\
    (x_1,\ldots,x_n) & \mapsto x_1 + \cdots + x_n
  \end{align*}
  has fibers of size at most $k$.  Therefore $(X_1,\ldots,X_n,P)$ is a
  coordinate configuration.
\end{proof}

\begin{lemma} \label{coheir-magic-1}
  Let $X_1,\ldots,X_n$ be quasi-minimal $A$-definable sets.  For each
  $i$ let $p_i$ be a global $A$-invariant type in $X_i$.  Assume $p_i$
  is not realized (i.e., not an algebraic type).  Then $(p_1 \otimes
  \cdots \otimes p_n) | A$ is broad.
\end{lemma}
\begin{proof}
  Because $p_i$ is not algebraic, any realization of $p_i^{\otimes
    \omega}$ is an indiscernible sequence of pairwise distinct
  elements.  Let $\{a_{i,j}\}_{i \in [n], j \in \Nn}$ be a realization
  of $p_1^{\otimes \omega} \otimes \cdots \otimes p_n^{\otimes
    \omega}$ over $A$.  Then for every $\eta : [n] \to \Nn$, the tuple
  \begin{equation*}
    (a_{1,\eta(1)},\ldots,a_{n,\eta(n)})
  \end{equation*}
  is a realization of $p_1 \otimes \cdots \otimes p_n$ over $A$, and
  the rows in this array are sequences of distinct elements.
\end{proof}
\begin{lemma}
  \label{coheir-magic-2}
  Let $(X_1,\ldots,X_n,P)$ be a coordinate configuration.  There
  exists a small set $A$ and non-algebraic global $A$-invariant types
  $p_i$ on $X_i$ such that if $(a_1,\ldots,a_n) \models p_1 \otimes
  \cdots \otimes p_n|A$ then $(a_1,\ldots,a_n) \in P$.
\end{lemma}
\begin{proof}
  By Lemma~\ref{simon-lemma}.\ref{sl-1} and
  Theorem~\ref{main-characterization}, we can find infinite definable
  sets $D_1, D_2, \ldots, D_n$ such that $(D_1 \times \cdots \times
  D_n) \setminus P$ is a hyperplane, hence narrow.  Choose a model $M$
  defining the $X_i$'s, $D_i$'s, and $P$, and let $p_i$ be any
  $M$-invariant non-algebraic global type in $D_i$, such as a coheir
  of a non-algebraic type in $D_i$ over $M$.  If $(a_1,\ldots,a_n)
  \models \bigotimes_{i = 1}^n p_i|M$, then $\tp(a_1,\ldots,a_n/M)$ is
  broad, hence cannot live in the hyperplane.  But $a_i \in D_i$, so
  $\vec{a} \in P$.
\end{proof}

\subsection{Critical coordinate configurations}
In this section, we carry out something similar to the Zilber
indecomposability technique in groups of finite Morley rank.  We
consider ways to expand coordinate configurations to higher rank, and
deduce consequences of being a maximal rank coordinate configuration.
\begin{definition}
  The \emph{critical rank} $\rho$ is the maximum rank of any
  coordinate configuration.  A coordinate configuration is
  \emph{critical} if its rank is the critical rank $\rho$.  A set $Y$
  is a \emph{critical set} if $Y$ is the target of some critical
  coordinate configuration.
\end{definition}
\begin{remark}
  \label{subs}
  If $Y$ is a critical set, then $\dpr(Y) = \rho$.  If $Y' \subseteq
  Y$ is definable of rank $\rho$, then $Y'$ is critical.
\end{remark}
\begin{proof}
  Indeed, $Y'$ is the target of some coordinate configuration by
  Proposition~\ref{basic-properties}.\ref{bp-1}. This coordinate
  configuration has rank $\rho$ and is thus critical.
\end{proof}
\begin{remark}
  \label{translate}
  If $Y$ is a critical set and $\alpha \in \Mm$, the translate $\alpha
  + Y$ is critical.
\end{remark}
\begin{proof}
  If $(X_1,\ldots,X_n,P)$ is a coordinate configuration with target
  $Y$, then $(X_1 + \alpha, X_2, \ldots, X_n, P')$ is a coordinate
  configuration with target $Y'$, where
  \begin{equation*}
    P' := \{(a_1 + \alpha, a_2, a_3, \ldots, a_n) ~|~ (a_1,\ldots,a_n) \in P\}.
  \end{equation*}
  Because $\dpr(Y) = \dpr(Y + \alpha)$, $Y + \alpha$ is critical.
\end{proof}

\begin{lemma}
  \label{expander}
  Let $(X_1,\ldots,X_n,P)$ be a critical coordinate configuration and
  $Q$ be a quasi-minimal set.  Let $A$ be a small set of parameters
  over which the $X_i,P,$ and $Q$ are defined.  Let
  $(a_1,\ldots,a_n,b)$ be a tuple in $X_1 \times \cdots \times X_n
  \times Q$ such that $\tp(a_1,\ldots,a_n,b/A)$ is broad and
  $(a_1,\ldots,a_n) \in P$.  Let $s = a_1 + \cdots + a_n + b$.  Then
  $b \notin \acl(sA)$.
\end{lemma}
\begin{proof}
  Suppose otherwise.  Let $s' = a_1 + \cdots + a_n = s - b$.  If $b
  \in \acl(sA)$, then $s' \in \acl(sA)$.  Because $\vec{a} \in P$, we
  have $\vec{a} \in \acl(s'A)$ by definition of coordinate
  configuration.  It follows that $(\vec{a},b) \in \acl(sA)$.  By
  Lemma~\ref{upgrades}, there is a coordinate configuration
  $(X_1,\ldots,X_n,Q,P')$.  This coordinate configuration has higher
  rank than $(X_1,\ldots,X_n,P)$, contradicting criticality.
\end{proof}

\begin{lemma}
  \label{shuffle-chaos}
  Let $A$ be a small set.  Suppose $b \notin \acl(A)$, and suppose $a
  \models p|Ab$ for some global $A$-invariant type $p$.  Then there is
  a small model $M$ containing $A$ and a global $M$-invariant
  \emph{non-constant} type $r$ such that $(a,b) \models p \otimes r |
  M$, i.e., $b \models r | M$ and $a \models p|Mb$.
\end{lemma}
\begin{proof}
  Let $b_1, b_2, \ldots$ be an $A$-indiscernible sequence of pairwise
  distinct realizations of $\tp(b/A)$.  Let $M_0$ be a small model
  containing $A$.  Let $b_1', b_2', \ldots$ be an $M_0$-indiscernible
  sequence extracted from $b_1, b_2, \ldots$.  Thus each $b_i'$
  realizes $\tp(b_i/A)$ and the $b'_i$ are pairwise distinct.  Let
  $\sigma \in \Aut(\Mm/A)$ move $b'_1$ to $b$, and let $M_1 =
  \sigma(M_0)$.  Then $\sigma(b'_1), \sigma(b'_2), \ldots$ is an
  $M_1$-indiscernible sequence whose first entry is $b$ and whose
  entries are pairwise distinct.  Therefore $M_1$ contains $A$ but not
  $b$.  Let $r_0$ be a global coheir of $tp(b/M_1)$; note that $r_0$
  is non-constant as $b \notin \acl(M_1)$.  Let $a'$ realize $p|M_1b$;
  thus $(a',b)$ realizes $p \otimes r_0 | M_1$.  Now $a'$ and $a$ both
  realize $p|Ab$, because $A \subseteq M_1$.  Therefore, there is an
  automorphism $\tau \in \Aut(\Mm/Ab)$ such that $\tau(a') = a$.  Let
  $M = \tau(M_1)$ and $r = \tau(r_0)$.  Then $(a,b) \models p \otimes
  r|M$ because $(a',b) \models p \otimes r_0|M_1$.
\end{proof}

\begin{lemma}
  \label{key-argument}
  Let $Y$ be a critical set, let $Q$ be quasi-minimal, and let $t \ge
  1$ be an integer.  There exist pairwise distinct $q_1, q_2, \ldots,
  q_t \in Q$ such that
  \begin{equation*}
    \dpr\left(\bigcap_{i = 1}^t (Y + q_i)\right) = \dpr(Y).
  \end{equation*}
  In particular, by Remarks~\ref{subs} and \ref{translate}, the
  intersection $\bigcap_{i = 1}^t (Y + q_i)$ is a critical set.
\end{lemma}
\begin{proof}
  Suppose otherwise.  Then for any distinct $q_1, \ldots, q_t$ in $Q$,
  the intersection $\bigcap_{i = 1}^t (Y + q_i)$ has rank at most
  $\rho - 1$.  Choose a critical coordinate configuration
  $(X_1,\ldots,X_n,P)$ with target $Y$.  Note that $\sum_i \dpr(X_i) =
  \rho$.  By Lemma~\ref{coheir-magic-2}, we may find a small set $A$
  and global $A$-invariant non-algebraic types $p_i$ on $X_i$ and
  $p_0$ on $Q$ such that
  \begin{itemize}
  \item The sets $X_i, P, Q$ are $A$-definable.
  \item The type $p_1 \otimes \cdots \otimes p_n$ lives on $P$.
  \end{itemize}
  \begin{claim}
    For any $k$, let $\Omega_k$ be the set of
    $(q_0,a_{1,1},\ldots,a_{1,n},\ldots,a_{k,1},\ldots,a_{k,n}) \in Q
    \times (X_1 \times \cdots \times X_n)^k$ such that
    \begin{itemize}
    \item For each $i \in [k]$, the tuple $(a_{i,1},\ldots,a_{i,n})$
      lies on $P$.
    \item There are infinitely many $q \in Q$ such that
      \begin{equation*}
        \bigwedge_{i = 1}^k \left( q_0 + \sum_{j = 1}^n a_{i,j} \in Y
        + q \right).
      \end{equation*}
    \end{itemize}
    Then $\Omega_k$ is narrow as a subset of $Q \times (X_1 \times
    \cdots \times X_n)^k$ for sufficiently large $k$.
  \end{claim}
  \begin{claimproof}
    Let $h = \dpr(Q) > 0$ and take $k$ so large that $th + k(\rho - 1)
    < h + k\rho$.  The set $\Omega_k$ is definable (because
    $\exists^\infty$ is eliminated).  If $\Omega_k$ were broad, it
    would have dp-rank $h + k\rho$ by Theorem~\ref{finite-rank-case}.
    Therefore $\Omega_k$ would contain a specific tuple
    $(q_0,a_{1,1},\ldots,a_{k,n})$ of dp-rank $h + k\rho$ over $A$.
    Let $s_i = a_{i,1} + \cdots + a_{i,n}$ for $i \in [k]$.  Because
    $(a_{i,1},\ldots,a_{i,n}) \in P$, we have $s_i \in Y$.  Moreover,
    by definition of coordinate configuration,
    \begin{equation*}
      (a_{i,1},\ldots,a_{i,n}) \in \acl(s_iA).
    \end{equation*}
    By definition of $\Omega_k$ there are infinitely many $q$ such
    that
    \begin{equation*}
      \{q_0 + s_1, q_0 + s_2, \ldots, q_0 + s_k\} \subseteq Y + q.
    \end{equation*}
    Choose $q_1, \ldots, q_{t-1}$ distinct such $q$, not equal to
    $q_0$.  Then
    \begin{equation*}
      q_0 + s_i \in (Y + q_0) \cap (Y + q_1) \cap \cdots \cap (Y +
      q_{t-1})
    \end{equation*}
    for every $i$.  Therefore
    \begin{equation*}
      \dpr(a_{i,1},\ldots,a_{i,n}/Aq_0q_1 \cdots q_{t-1}) =
      \dpr(s_i/Aq_0q_1 \cdots q_{t-1}) \le \dpr\left(\bigcap_{i =
        0}^{t-1}(Y + q_i)\right) < \rho.
    \end{equation*}
    By subadditivity of dp-rank,
    \begin{equation*}
      \dpr(a_{1,1},\ldots,a_{k,n},q_0,q_1,\ldots,q_{t-1}/A) \le k(\rho - 1) + th < k\rho + h
    \end{equation*}
    contradicting the fact that
    \begin{equation*}
      \dpr(a_{1,1},\ldots,a_{i,n},q_0/A) = k\rho + h. \qedhere
    \end{equation*}
  \end{claimproof}
  Fix $k$ as in the Claim.  Choose
  \begin{equation*}
    (a_{k,1},\ldots,a_{k,n},\ldots,a_{1,1},\ldots,a_{1,n},q_0)
  \end{equation*}
  realizing $(p_1 \otimes \cdots \otimes p_n)^{\otimes k} \otimes p_0$
  over $A$.  Let $s_i = a_{i,1} + \cdots + a_{i,n}$.  By choice of
  $p_1 \otimes \cdots \otimes p_n$, each $(a_{i,1},\ldots,a_{i,n})$
  lives on $P$, and therefore $s_i \in Y$.  By
  Lemma~\ref{coheir-magic-1}, the type of $(\vec{a},q_0)$ over $A$ is
  broad in $(X_1 \times \cdots \times X_n)^k \times Q$.  Therefore the
  tuple $(q_0,\vec{a})$ cannot lie in the narrow set $\Omega_k$.
  Consequently, there are only finitely many $q \in Q$ such that
  \begin{equation*}
    \{s_1 + q_0, \ldots, s_k + q_0\} \subseteq Y + q.
  \end{equation*}
  As $s_i \in Y$, one such $q$ is $q_0$.  Therefore,
  \begin{equation*}
    q_0 \in \acl(A,s_1 + q_0,\ldots, s_k + q_0).
  \end{equation*}
  Choose $j$ minimal such that
  \begin{equation*}
    q_0 \in \acl(A,s_1 + q_0, \ldots, s_j + q_0).
  \end{equation*}
  Note that $j \ge 1$ because $\tp(q_0/A)$ is the non-algebraic type
  $p_0$.  Let $A' = A \cup \{s_1 + q_0, \ldots, s_{j-1} + q_0\}$.  By
  choice of $j$, $q_0 \notin \acl(A')$.  Note that
  $(a_{j,1},\ldots,a_{j,n})$ realizes the $A$-invariant type $p_1
  \otimes \cdots \otimes p_n$ over $A'q_0 \subseteq
  \dcl(q_0,a_{1,1},\ldots,a_{j-1,n}A)$.  Applying
  Lemma~\ref{shuffle-chaos}, we find a small model $M \supseteq A'$
  and a non-constant global $M$-invariant type $r$ such that
  \begin{equation*}
    (a_{j,1},\ldots,a_{j,n},q_0) \models p_1 \otimes \cdots \otimes
    p_n \otimes r|M.
  \end{equation*}
  The type $r$ extends $\tp(q_0/M)$ and therefore lives in the $M$-definable set $Q$.
  By Lemma~\ref{coheir-magic-1}, it follows that
  \begin{equation*}
    \tp(a_{j,1},\ldots,a_{j,n},q_0/M)
  \end{equation*}
  is broad, as a type in $X_1 \times \cdots \times X_n \times Q$.
  This contradicts Lemma~\ref{expander}: the $X_i, P, Q$ are all
  $M$-definable, the type of $(a_{j,1},\ldots,a_{j,n},q_0)$ over $M$
  is broad, the tuple $(a_{j,1},\ldots,a_{j,n})$ lives on $P$, and
  for
  \begin{equation*}
    s = a_{j,1} + \cdots + a_{j,n} + q_0 = s_j + q_0,
  \end{equation*}
  we have
  \begin{equation*}
    q_0 \in \acl(A',s_j + q_0) \subseteq \acl(sM). \qedhere
  \end{equation*}
\end{proof}

\begin{proposition}\label{stacking}
  Let $Y$ be a critical set and $Q_1, \ldots, Q_n$ be quasi-minimal.
  Then for every $m$ there exist $\{q_{i,j}\}_{i \in [n], j \in [m]}$
  such that
  \begin{itemize}
  \item For fixed $i \in [n]$, the sequence $q_{i,1},\ldots,q_{i,m}$
    consists of $m$ distinct elements of $Q_i$.
  \item The intersection
    \begin{equation*}
      \bigcap_{\eta : [n] \to [m]} \left(Y + \sum_{i = 1}^n q_{i,\eta(i)}\right)
    \end{equation*}
    is critical.
  \end{itemize}
\end{proposition}
\begin{proof}
  We prove this by induction on $n$; the base case $n = 1$ is
  Lemma~\ref{key-argument}.  If $n > 1$, by induction there are
  $q_{1,1},\ldots,q_{n-1,m}$ in the correct places, such that
  \begin{equation*}
    Y' = \bigcap_{\eta : [n-1] \to [m]} \left(Y + \sum_{i = 1}^n
    q_{i,\eta(i)}\right)
  \end{equation*}
  is critical.  By Lemma~\ref{key-argument} there are
  $q_{n,1},\ldots,q_{n,m} \in Q_n$, pairwise distinct, such that
  \begin{equation*}
    Y'' = (Y' + q_{n,1}) \cap \cdots \cap (Y' + q_{n,m})
  \end{equation*}
  is critical.  But $x \in Y''$ if and only if for every $j \in [m]$
  and every $\eta : [n-1] \to [m]$ the element $x - q_{n,j} -
  q_{n-1,\eta(n-1)} - \cdots - q_{1,\eta(1)}$ lies in $Y$.  Thus
  \begin{equation*}
    Y'' = \bigcap_{\eta : [n] \to [m]} \left(Y + \sum_{i = 1}^n
    q_{i,\eta(i)} \right). \qedhere
  \end{equation*}
\end{proof}

\begin{corollary}
  \label{stacking-2}
  Let $Y$ be a critical set and $Q_1,\ldots,Q_n$ be quasi-minimal.
  There exists a $\delta \in \Mm$ such that
  \begin{equation*}
    \{ (x_1,\ldots,x_n) \in Q_1 \times \cdots \times Q_n ~|~ x_1 +
    \cdots + x_n \in Y + \delta\}
  \end{equation*}
  is broad, as a subset of $Q_1 \times \cdots \times Q_n$.
\end{corollary}
\begin{proof}
  For any $m \ge 1$, we can find $\{q_{i,j}\}_{i \in [n], j \in [m]}$
    such that
  \begin{itemize}
  \item $q_{i,j} \in Q_i$
  \item $q_{i,j} \ne q_{i,j'}$ for $j \ne j'$.
  \item The intersection
    \begin{equation*}
      \bigcap_{\eta : [n] \to [m]} \left(Y - \sum_{i = 1}^n q_{i,\eta(i)} \right)
    \end{equation*}
    is non-empty.
  \end{itemize}
  Indeed, this follows from Proposition~\ref{stacking} applied to the
  quasi-minimal sets $-Q_i$.  If $-\delta$ is an element of the
  intersection, then for any $\eta : [n] \to [m]$ we have
  \begin{equation*}
    -\delta \in Y - \sum_{i = 1}^n q_{i,\eta(i)}
  \end{equation*}
  or equivalently,
  \begin{equation*}
    \sum_{i = 1}^n q_{i, \eta(i)} \in Y + \delta.
  \end{equation*}
  Thus, for every $m$ we can find $\{q_{i,j}\}_{i \in [n], j \in [m]}$
  and $\delta \in \Mm$ such that
  \begin{itemize}
  \item $q_{i,j} \in Q_i$.
  \item $q_{i,j} \ne q_{i,j'}$ for $j \ne j'$.
  \item For any $\eta : [n] \to [m]$ the sum $\sum_{i = 1}^n
    q_{i,\eta(i)}$ lies in $Y + \delta$.
  \end{itemize}
  By compactness, we can find $\{q_{i,j}\}_{i \in [n],~j \in \Nn}$ and
  $\delta$ satisfying these same properties.  This means that
  \begin{equation*}
    \{(q_1,\ldots,q_n) \in Q_1 \times \cdots \times Q_n ~|~ q_1 + \cdots
    + q_n \in Y + \delta\}
  \end{equation*}
  is broad as a subset of $Q_1 \times \cdots \times Q_n$.
\end{proof}

\subsection{Heavy sets}
We are nearly ready to define heavy and light sets.
\begin{definition}
  Let $Y$ be a critical set.  Say that a definable set $X \subseteq
  \Mm$ is \emph{$Y$-heavy} if there is $\delta \in \Mm$ such that
  $\dpr(Y \cap (X + \delta)) = \dpr(Y)$.
\end{definition}
\begin{remark} \label{def-in-fams}
  Let $Y$ be a critical set and $\{D_b\}$ be a definable family of
  subsets of $\Mm$.  Let $S$ be the set of $b$ such that $D_b$ is
  $Y$-heavy.  Then the set $S$ is definable.  Indeed, this follows
  almost immediately from
  Proposition~\ref{basic-properties}.\ref{full-rank-def}.
\end{remark}
\begin{proposition}
  Let $Y, Y'$ be two critical sets and $X \subseteq \Mm$ be definable.
  Then $X$ is $Y$-heavy if and only if $X$ is $Y'$-heavy.
\end{proposition}
\begin{proof}
  Suppose $X$ is $Y$-heavy.  Let $X' := Y \cap (X + \delta_0)$ have
  full rank in $Y$.  Then $X'$ is critical by Remark~\ref{subs}.  We
  claim that $X'$ is $Y'$-heavy.  Let $(A_1,\ldots,A_n,P)$ be a
  coordinate configuration for $Y'$.  The set $P$ is broad in $A_1
  \times \cdots \times A_n$, so there exist infinite definable subsets
  $Q_i \subseteq A_i$ such that $(Q_1 \times \cdots \times Q_n)
  \setminus P$ is a hyperplane, and in particular is narrow.  By
  Corollary~\ref{stacking-2}, there is a $\delta_1 \in \Mm$ such that
  \begin{equation*}
    \{(q_1,\ldots,q_n) \in Q_1 \times \cdots \times Q_n ~|~ q_1 + \cdots
    + q_n \in X' + \delta_1\}
  \end{equation*}
  is broad as a subset of $Q_1 \times \cdots \times Q_n$, hence broad
  as a subset of $A_1 \times \cdots \times A_n$.  Because $(Q_1 \times
  \cdots \times Q_n) \setminus P$ is narrow, it follows that
  \begin{equation*}
    \{(q_1,\ldots,q_n) \in P ~|~ q_1 + \cdots + q_n \in X' + \delta_1\}
  \end{equation*}
  is broad as a subset of $A_1 \times \cdots \times A_n$.  Because the
  map $(q_1,\ldots,q_n) \mapsto q_1 + \cdots + q_n$ has finite fibers
  on $P$, it follows that
  \begin{equation*}
    \{q_1 + \cdots + q_n ~|~ \vec{q} \in P \text{ and } q_1 + \cdots + q_n \in
    X' + \delta_1\} = Y' \cap (X' + \delta_1)
  \end{equation*}
  has full rank.  Therefore, $X'$ is $Y'$-heavy.  Now $X' \subseteq (X +
  \delta_0)$, and so $Y' \cap (X + \delta_0 + \delta_1)$ has full rank
  in $Y'$, implying that $X$ is $Y'$-heavy.  The converse follows by
  symmetry.
\end{proof}
\begin{definition}
  A definable subset $X \subseteq \Mm$ is \emph{heavy} if it is $Y$-heavy
  for any/every critical set $Y \subseteq \Mm$, and \emph{light}
  otherwise.
\end{definition}

\begin{theorem} \label{heavy-light}
  Let $X, Y$ be definable subsets of $\Mm$.
  \begin{enumerate}
  \item (Assuming $\Mm$ is infinite) If $X$ is finite, then $X$ is
    light.
  \item \label{light-union} If $X, Y$ are light, then $X \cup Y$ is
    light.
  \item If $Y$ is light and $X \subseteq Y$, then $X$ is light.
  \item If $\{D_b\}$ is a definable family of subsets of $\Mm$, then
    \begin{equation*}
      \{b ~|~ D_b \text{ is light}\}
    \end{equation*}
    is definable.
  \item $\Mm$ is heavy.
  \item If $X$ is heavy, then $\alpha \cdot X$ is heavy for $\alpha \in
    \Mm^\times$.
  \item \label{light-translate} If $X$ is heavy, then $\alpha + X$ is
    heavy for $\alpha \in \Mm$.
  \item If $X$ and $Y$ are heavy, the set
    \begin{equation*}
      X \mininf Y := \{\delta \in \Mm ~|~ X \cap (Y + \delta) \text{ is heavy}\}
    \end{equation*}
    is heavy.
  \end{enumerate}
\end{theorem}
\begin{proof}
  \begin{enumerate}
  \item Indeed, if $X$ is heavy then some translate of $X$ contains a
    critical set, and so $X$ has rank at least $\rho$.  The critical
    rank $\rho$ cannot be 0 because there is at least one
    quasi-minimal set $Q_i$ (by Remark~\ref{they-exist}), and
    $(Q_i,Q_i)$ is a coordinate configuration with target $Q_i$.
    Therefore heavy sets are infinite.
  \item We show the contrapositive: if $X \cup Y$ is heavy, then $X$ is
    heavy or $Y$ is heavy.  Fix some critical set $Z$.  The fact that $X
    \cup Y$ is $Z$-heavy implies that there is a $\delta$ such that $Z
    \cap ((X \cup Y) + \delta)$ has full rank.  But
    \begin{equation*}
      Z \cap ((X \cup Y) + \delta) = (Z \cap (X + \delta)) \cup (Z
      \cap (Y + \delta))
    \end{equation*}
    so at least one of $Z \cap (X + \delta)$ and $Z \cap (Y + \delta)$
    has full rank $\rho$.
  \item If $X \subseteq Y$ and $X$ is $Z$-heavy because $Z \cap (X +
    \delta)$ has full rank $\rho$, then certainly $Z \cap (Y +
    \delta)$ has full rank $\rho$.
  \item This is Remark~\ref{def-in-fams}.
  \item Critical sets exist, and if $Y$ is a critical set, then
    $\dpr(Y \cap \Mm) = \dpr(Y)$.
  \item Heaviness was defined without any reference to multiplication,
    so it is preserved by definable automorphisms of the additive
    group $(\Mm,+)$.
  \item Heaviness is translation-invariant by definition.
  \item Shrinking $X$ and $Y$ we may assume that $X$ and $Y$ are
    translates of critical sets--or merely critical sets by
    Remark~\ref{translate}.  Let $(A_1,\ldots,A_n,P)$ be a coordinate
    configuration for $Y$.  By Theorem~\ref{main-characterization} we
    can find infinite definable subsets $Q_i \subseteq A_i$ such that
    $(Q_1 \times \cdots \times Q_n) \setminus P$ is narrow.  By
    Corollary~\ref{stacking-2}, there is a $\delta \in \Mm$ such that
    the set
    \begin{equation*}
      \{(q_1,\ldots,q_n,q'_1,\ldots,q'_n) \in (Q_1 \times \cdots
      \times Q_n)^2 ~|~ q_1 + \cdots + q_n + q'_1 + \cdots + q'_n \in X
      + \delta\}
    \end{equation*}
    is broad as a subset of $Q_1 \times \cdots \times Q_n \times Q_1
    \times \cdots \times Q_n$.  In particular, we can find $q_{i,j}$
    and $q'_{i,j}$ for $i \in [n]$ and $j \in \Nn$ such that
    \begin{itemize}
    \item $q_{i,j}$ and $q'_{i,j}$ are in $Q_i$.
    \item $q_{i,j} \ne q_{i,j'}$ and $q'_{i,j} \ne q'_{i,j'}$ for $j
      \ne j'$.
    \item For any $\eta : [n] \to \Nn$ and $\eta' : [n] \to \Nn$, we
      have
      \begin{equation*}
        \sum_{i = 1}^n (q_{i,\eta(i)} + q'_{i,\eta'(i)}) \in X +
        \delta.
      \end{equation*}
    \end{itemize}
    Let $s_\eta$ denote the sum $\sum_{i = 1}^n q_{i,\eta(i)}$.  Note
    that for any $\eta$, the $q'_{i,j}$ show that the set
    \begin{equation*}
      \{(q'_1,\ldots,q'_n) \in Q_1 \times \cdots \times Q_n ~|~ s_\eta +
      q'_1 + \cdots + q'_n \in X + \delta\}
    \end{equation*}
    is broad as a subset of $Q_1 \times \cdots \times Q_n$.  It follows
    that
    \begin{equation*}
      \{(q'_1,\ldots,q'_n) \in P ~|~ s_\eta +
      q'_1 + \cdots + q'_n \in X + \delta\}
    \end{equation*}
    has full rank $\rho$, and so $(Y + s_\eta) \cap (X + \delta)$ has
    full rank $\rho$ for any $\eta$.  The translate $X \cap (Y +
    s_\eta - \delta)$ has full rank and is trivially $X$-heavy.
    Therefore $s_\eta - \delta \in X \mininf Y$.  The fact that this
    holds for all $s_\eta$ implies that
    \begin{equation*}
      \{(q_1,\ldots,q_n) \in Q_1 \times \cdots \times Q_n ~|~ q_1 +
      \cdots + q_n + \delta \in X \mininf Y\}
    \end{equation*}
    is broad as a subset of $Q_1 \times \cdots \times Q_n$.  Again, this
    implies that $(Y + \delta) \cap (X \mininf Y)$ has full rank
    $\rho$, and so $X \mininf Y$ is heavy. \qedhere
  \end{enumerate}
\end{proof}

\subsection{Externally definable sets}
We prove the analogues of \S\ref{sec:extern1} for heavy and light
sets.
\begin{lemma}
  \label{chernikov-2}
  Let $M \preceq \Mm$ be a small model defining a critical coordinate
  configuration.  Let $Z$ be an $M$-definable heavy set.  Let $D_1,
  \ldots, D_m$ be $\Mm$-definable sets such that
  \begin{equation*}
    Z(M) \subseteq D_1 \cup \cdots \cup D_m.
  \end{equation*}
  Then there is a $j \in [m]$ and a $M$-definable heavy set $Z'
  \subseteq Z$ such that $Z'(M) \subseteq D_j$.
\end{lemma}
\begin{proof}
  Let $(X_1,\ldots,X_n,P)$ be an $M$-definable critical coordinate
  configuration, with target $Y$.  Because $Z$ is $Y$-heavy, there is
  some $\delta$ such that $Y \cap (Z + \delta)$ has full rank in $Y$.
  Full rank in $Y$ is a definable condition, clearly $M$-invariant, so
  we may take $\delta \in M$.  Replacing $Z$ and the $D_i$'s with $Z +
  \delta$ and $D_i + \delta$, we may assume $\delta = 0$.  The fact
  that $Y \cap Z$ has full rank implies that the set
  \begin{equation*}
    R := \{(x_1,\ldots,x_n) \in P ~|~ x_1 + \cdots + x_n \in Z\}
  \end{equation*}
  is broad.  Furthermore, $R(M)$ is a union of the
  externally-definable sets
  \begin{equation*}
    \{(x_1,\ldots,x_n) \in P(M) ~|~ x_1 + \cdots + x_n \in Z \cap D_j.\}
  \end{equation*}
  By Lemma~\ref{chernikov-1}, there is an $M$-definable $R' \subseteq
  R$ and a $j \in [m]$ such that $R'$ is broad and
  \begin{equation*}
    (x_1,\ldots,x_n) \in R'(M) \implies x_1 + \cdots + x_n \in D_j.
  \end{equation*}
  Let $Z'$ be the $M$-definable set
  \begin{equation*}
    \{x_1 + \cdots + x_n ~|~ (x_1,\ldots,x_n) \in R' \}.
  \end{equation*}
  Because $R' \subseteq R \subseteq P$, the set $Z' \subseteq Y \cap
  Z$.  The set of $(x_1,\ldots,x_n) \in P$ such that $x_1 + \cdots +
  x_n \in Z'$ contains the broad set $R'$, and is therefore broad.  It
  follows that $Z'$ has full rank in $Y$, and so $Z'$ is $Y$-heavy.
  Lastly, we claim that
  \begin{equation*}
    Z'(M) \subseteq D_j.
  \end{equation*}
  Indeed, suppose $y \in Z'(M)$.  Then there exists an $n$-tuple
  $(x_1, \ldots, x_n) \in R'$ such that $y = x_1 + \cdots + x_n$.  The
  fact that $M \preceq \Mm$ implies that we can take $\vec{x} \in
  \dcl(M)$.  Then by choice of $R'$, it follows that
  \begin{equation*}
    y = x_1 + \cdots + x_n \in D_j.
  \end{equation*}
  Therefore $Z'(M) \subseteq D_j$.
\end{proof}
\begin{lemma}\label{something-similar-2}
  Let $M$ be a small model defining a critical coordinate
  configuration.  Let $Z$ be an $M$-definable heavy set, and let $W$
  be an $\Mm$-definable set.  If $Z(M) \subseteq W$, then $W$ is
  heavy.
\end{lemma}
\begin{proof}
  Take an $M$-definable coordinate configuration $(X_1,\ldots,X_n,P)$
  with target $Y$.  As in the proof of Lemma~\ref{chernikov-2} we may
  assume that $Z \cap Y$ has full rank in $Y$, and so the set
  \begin{equation*}
    \{(x_1,\ldots,x_n) \in P ~|~ x_1 + \cdots + x_n \in Z\}
  \end{equation*}
  is broad in $X_1 \times \cdots \times X_n$.  This set is
  $M$-definable, and its $M$-definable points lie in the set
  \begin{equation*}
    \{(x_1,\ldots,x_n) \in P ~|~ x_1 + \cdots + x_n \in W\},
  \end{equation*}
  which is therefore broad by Lemma~\ref{something-similar-1}.  This
  implies that $W \cap Y$ has full rank, and so $W$ is ($Y$-)heavy.
\end{proof}

\section{The edge case: finite Morley rank} \label{sec:likeTT}
If our end goal is to construct non-trivial valuation rings on
dp-finite fields, we had better quarantine the stable fields at some
point.  In other words, we need to find some positive consequence of
the hypothesis ``not stable.''  We prove the relevant dichotomy in
this section.
\begin{lemma}\label{morley-rank}
  Let $X_1, \ldots, X_n$ be infinite definable sets.  If the number of
  broad global types in $X_1 \times \cdots \times X_n$ is finite, then
  every $X_i$ has Morley rank 1.
\end{lemma}
\begin{proof}
  Suppose, say, $X_1$ has Morley rank greater than 1.  Then we can
  find a sequence $D_1, D_2, D_3, \ldots$ of length $\omega$ where
  each $D_i$ is an infinite definable subset of $X_1$ and the $D_i$
  are pairwise disjoint.  Let $P_i = D_1 \times X_2 \times \cdots
  \times X_n$.  Each $P_i$ is broad and the $P_i$ are pairwise
  disjoint.  By Proposition~\ref{completion}, each $P_i$ is inhabited
  by a broad type over $\Mm$, so there are at least $\aleph_0$ broad
  global types.  If there are a finite number of broad global types,
  it follows that every $X_i$ has Morley rank 1.
\end{proof}

\begin{theorem}\label{tt-dichotomy}
  At least one of the following holds:
  \begin{enumerate}
  \item The field $\Mm$ has finite Morley rank.
  \item There are two disjoint heavy sets.
  \end{enumerate}
\end{theorem}
\begin{proof}
  Let $(X_1,\ldots,X_n,P)$ be a critical coordinate configuration.  We
  may find infinite definable subsets $Q_i \subseteq X_i$ such that for
  \begin{align*}
    H &:= (Q_1 \times \cdots \times Q_n) \setminus P \\
    P' &:= (Q_1 \times \cdots \times Q_n) \cap P
  \end{align*}
  the set $H$ is narrow in $Q_1 \times \cdots \times Q_n$, and
  therefore $P'$ is broad.  Then $(Q_1,\ldots,Q_n,P')$ is a coordinate
  configuration of the same rank as $(X_1,\ldots,X_n,P)$.  In
  particular $(Q_1,\ldots,Q_n,P')$ is critical and its target $Y$ is a
  critical set.

  First suppose that at least one $Q_i$ has Morley rank greater than 1
  (possibly $\infty$).  By Lemma~\ref{morley-rank}, there are
  infinitely many broad global types in $Q_1 \times \cdots \times
  Q_n$.  Let $p_1, p_2, p_3, \ldots$ be pairwise distinct broad global
  types.  Because $H$ is narrow, each $p_i$ must live in $P$.  Let $m$
  bound the fiber size of the map $\Sigma : P' \to Y$.  There is a
  pushforward map $\Sigma_*$ from the space of $\Mm$-types in $P'$ to
  the space of $\Mm$-types in $Y$, and this map has fibers of size no
  greater than $m$.  Consequently, we can find $i, j$ such that
  $\Sigma_* p_i \ne \Sigma_* p_j$.  Without loss of generality
  $\Sigma_* p_1 \ne \Sigma_* p_2$.  We can then write $Y$ as a
  disjoint union $Y = Z_1 \cup Z_2$ where $\Sigma_* p_i$ lives on
  $Z_i$.  For $i = 1, 2$, the definable set
  \begin{equation*}
    \{(x_1,\ldots,x_n) \in P' ~|~ x_1 + \cdots + x_n \in Z_i\}
  \end{equation*}
  is broad as a subset of $Q_1 \times \cdots \times Q_n$ because it
  contains the broad type $p_i$.  Therefore this set has dp-rank
  $\rho$ and so $\dpr(Z_i) = \rho = \dpr(Y)$.  So the two sets $Z_1,
  Z_2$ are full-rank subsets of $Y$, and are therefore $Y$-heavy.  We
  have constructed two disjoint heavy sets.

  Next suppose that every $Q_i$ has Morley rank 1.  Let
  \begin{equation*}
    W = Q_1 + \cdots + Q_n.
  \end{equation*}
  The set $W$ contains the critical set $Y$, so $W$ is heavy.  If $W -
  W = \Mm$, then there is a surjective map
  \begin{align*}
    Q_1 \times \cdots \times Q_n \times Q_1 \times \cdots \times Q_n &\to \Mm \\
    (x_1,\ldots,x_n,x'_1,\ldots,x'_n) &\mapsto x_1 + \cdots + x_n - x'_1 - \cdots - x'_n.
  \end{align*}
  The set $(Q_1 \times \cdots \times Q_n)^2$ has finite Morley
  rank,\footnote{See Proposition 1.5.16 in \cite{PillayGST} or the
    proof of Corollary 2.14 in \cite{stab-groups}.}, and so $\Mm$ must
  have finite Morley rank.  Otherwise, take $x \notin W - W$.  Then
  the sets $W$ and $W + x$ are disjoint, and both are heavy.
\end{proof}

\begin{remark}
  Theorem~\ref{tt-dichotomy} is closely related to Sinclair's
  \emph{Large Sets Property} (LSP).  A field $\Mm$ of dp-rank $n$ has
  the LSP if there do not exist two disjoint subsets of $\Mm$ of
  dp-rank $n$ (\cite{sinclair}, Definition 3.0.3).  Sinclair
  conjectures that fields with LSP are stable, hence algebraically
  closed.  Theorem~\ref{tt-dichotomy} says something similar, with
  ``heavy'' in place of dp-rank $n$.
\end{remark}

\section{Infinitesimals}\label{sec:infinitesimals}
In this section, we carry out the construction of infinitesimals.  The
arguments are copied directly from the dp-minimal setting (\cite{myself}, Chapter~9).
\subsection{Basic neighborhoods}
\begin{definition}
  If $X, Y$ are definable sets, $X \mininf Y$ denotes the set
  \begin{equation*}
    \{ \delta \in \Mm ~|~ X \cap (Y + \delta) \text{ is heavy}\}.
  \end{equation*}
\end{definition}
\begin{remark} \label{mininf-props}
  Let $X, Y$ be definable sets.
  \begin{enumerate}
  \item $X \mininf Y$ is definable.
  \item $X \mininf Y \subseteq X - Y$, where
    \begin{equation*}
      X - Y = \{x - y ~|~ x \in X, ~ y \in Y\}.
    \end{equation*}
  \item If $X$ or $Y$ is light, then $X \mininf Y$ is empty.
  \item If $X$ and $Y$ are heavy, then $X \mininf Y$ is heavy.
  \item For any $\delta_1, \delta_2 \in \Mm$,
    \begin{equation*}
      (X + \delta_1) \mininf (Y + \delta_2) = (X \mininf Y) +
      (\delta_1 - \delta_2).
    \end{equation*}
  \item For any $\alpha \in \Mm^\times$,
    \begin{equation*}
      (\alpha \cdot X) \mininf (\alpha \cdot Y) = \alpha \cdot (X
      \mininf Y).
    \end{equation*}
  \item If $X' \subseteq X$ and $Y' \subseteq Y$ then
    \begin{equation*}
      X' \mininf Y' \subseteq X \mininf Y.
    \end{equation*}
  \end{enumerate}
  Indeed, all of these properties follow immediately from
  Theorem~\ref{heavy-light}.
\end{remark}
\begin{definition}
  A \emph{basic neighborhood} is a set of the form $X \mininf X$,
  where $X$ is heavy.
\end{definition}
\begin{remark}
  If $M$ is a small model, then every $M$-definable basic
  neighborhood is of the form $X \mininf X$ with $X$ heavy and
  $M$-definable.
\end{remark}
\begin{proof}
  Let $U = X \mininf X$ be a basic neighborhood which is
  $M$-definable.  Write $X$ as $\phi(\Mm;b)$.  Because heaviness is
  definable in families (by Theorem~\ref{heavy-light}), the set of
  $b'$ such that $\phi(\Mm;b') \mininf \phi(\Mm;b') = U$ is definable.
  This set is $M$-invariant, and non-empty because of $b$, so it
  intersects $M$.  Therefore, we can find $M$-definable $X' =
  \phi(\Mm;b')$ such that $U = X' \mininf X'$.
\end{proof}
\begin{proposition}\label{basic-nbhds}
  ~
  \begin{enumerate}
  \item \label{heavy-nbhd} If $U$ is a basic neighborhood, then $U$ is
    heavy.
  \item \label{zero-in} If $U$ is a basic neighborhood, then $0 \in
    U$.
  \item \label{scaling} If $U$ is a basic neighborhood and $\alpha \in
    \Mm^\times$, then $\alpha \cdot U$ is a basic neighborhood.
  \item \label{filtered-isect} If $U_1, U_2$ are basic neighborhoods,
    there is a basic neighborhood $U_3 \subseteq U_1 \cap U_2$.  If
    $U_1, U_2$ are $M$-definable, we can choose $U_3$ to be
    $M$-definable.
  \item \label{non-degen} If $\Mm$ is not of finite Morley rank, then
    for every $\alpha \ne 0$ there is a basic neighborhood $U \not\ni
    \alpha$.  If $\alpha$ is $M$-definable, we can choose $U$ to be
    $M$-definable.
  \end{enumerate}
\end{proposition}
\begin{proof}
  \begin{enumerate}
  \item This is part of Remark~\ref{mininf-props}.
  \item If $X$ is heavy, then $0 \in X \mininf X$ because $X \cap (X +
    0)$ is the heavy set $X$.
  \item If $U = X \mininf X$, then $\alpha \cdot U = (\alpha \cdot X)
    \mininf (\alpha \cdot X)$ by Remark~\ref{mininf-props}, and
    $\alpha \cdot X$ is heavy by Theorem~\ref{heavy-light}.
  \item Let $U_i = X_i \mininf X_i$.  By Theorem~\ref{heavy-light} the
    set $X_1 \mininf X_2$ is heavy, hence non-empty.  Take $\delta \in
    X_1 \mininf X_2$.  Then $X_3 := X_1 \cap (X_2 + \delta)$ is heavy.
    By Remark~\ref{mininf-props}
    \begin{align*}
      X_3 \mininf X_3 &\subseteq X_1 \mininf X_1 = U_1 \\
      X_3 \mininf X_3 & \subseteq (X_2 + \delta) \mininf (X_2 +
      \delta) = X_2 \mininf X_2 = U_2.
    \end{align*}
    If $U_1$ and $U_2$ are $M$-definable and $X_3 = \phi(\Mm;b_0)$,
    the set of $b$ such that
    \begin{equation*}
      \emptyset \ne \phi(\Mm;b) \mininf \phi(\Mm;b) \subseteq U_1 \cap U_2
    \end{equation*}
    is $M$-definable and non-empty.  Therefore we can move $X_3$ to be
    $M$-definable.
  \item If $\Mm$ is not of finite morley rank, we can find two
    disjoint heavy sets $X, Y$.  By Theorem~\ref{heavy-light}, $X
    \mininf Y \ne \emptyset$.  Take $\delta \in X \mininf Y$.  Then
    the sets
    \begin{align*}
      X' & := X \cap (Y + \delta) \\
      Y' & := (X - \delta) \cap Y
    \end{align*}
    are heavy---$X'$ is heavy by choice of $\delta$ and $Y'$ is a
    translate of $X'$.  Also,
    \begin{equation*}
      Y' \cap (Y' + \delta) = Y' \cap X' \subseteq Y \cap X = \emptyset,
    \end{equation*}
    so $\delta \notin Y' \mininf Y'$.  Because $Y'$ is heavy, there is
    a basic neighborhood $U_0 = Y' \mininf Y'$ such that $\delta
    \notin U_0$.  By part 2, $\delta \ne 0$.  If $\alpha$ is a
    non-zero element of $\Mm$, then we can find $\beta \in \Mm^\times$
    such that $\alpha = \beta \cdot \delta$.  Then $\alpha = \beta
    \cdot \delta$ is not in $U = \beta \cdot U_0$, which is a basic
    neighborhood by part 3.

    Finally, suppose that $\alpha$ is a non-zero element of a small
    model $M$.  We know that there is some heavy set $\phi(\Mm;b)$
    such that
    \begin{equation*}
      \alpha \notin \phi(\Mm;b) \mininf \phi(\Mm;b).
    \end{equation*}
    As before, the set of $b$ with this property is $M$-definable,
    hence we can pull $b$ into $M$, producing an $M$-definable basic
    neighborhood that avoids $\alpha$. \qedhere
  \end{enumerate}
\end{proof}

\subsection{The set of infinitesimals}
\begin{definition}
  Let $M$ be a small model.  An element $\varepsilon \in \Mm$ is an
  \emph{$M$-infinitesimal} if $\varepsilon$ lies in every $M$-definable
  basic neighborhood.  Equivalently, for every $M$-definable heavy set
  $X$, the intersection $X \cap (X + \varepsilon)$ is heavy.  We let
  $I_M$ denote the set of $M$-infinitesimals.
\end{definition}
\begin{remark}
  Note that $I_M$ is type-definable over $M$, as it is the small
  intersection
  \begin{equation*}
    I_M = \bigcap_{X \text{ heavy and $M$-definable}} X \mininf X.
  \end{equation*}
  Moreover, this intersection is \emph{directed}, by
  Proposition~\ref{basic-nbhds}.
\end{remark}
\begin{remark}
  \label{change-of-M}
  If $M' \succeq M$, then $I_{M'} \subseteq I_M$ because there are
  more neighborhoods in the intersection.
\end{remark}
\begin{remark}\label{basic-infs}
  Let $M$ be a small model.
  \begin{enumerate}
  \item \label{non-trivial} If $X$ is a definable set containing
    $I_M$, then $X$ is heavy.  In particular, if $\Mm$ is infinite,
    then $I_M \ne 0$.
  \item $0 \in I_M$.
  \item \label{infinitesimal-times-standard} $I_M \cdot M \subseteq
    I_M$.
  \item If $\Mm$ is not of finite Morley rank, then $I_M \cap M =
    \{0\}$.
  \end{enumerate}
\end{remark}
\begin{proof}
  \begin{enumerate}
  \item The fact that $X \supseteq I_M$ implies by compactness that
    there exist $M$-definable basic neighborhoods $U_1, \ldots, U_n$
    such that
    \begin{equation*}
      X \supseteq U_1 \cap \cdots \cap U_n.
    \end{equation*}
    However, the fact that the intersection is directed (i.e.,
    Proposition~\ref{basic-nbhds}.\ref{filtered-isect}) implies that
    we can take $n = 1$.  Then $X \supseteq U_1$ and $U_1$ is heavy
    (by Proposition~\ref{basic-nbhds}.\ref{heavy-nbhd}), so $X$ is
    heavy.
  \item Every basic neighborhood contains $0$ by
    Proposition~\ref{basic-nbhds}.\ref{zero-in}.
  \item Suppose $a \in M$.  We claim that $a \cdot I_M \subseteq I_M$.
    If $a = 0$ this is the previous point.  Otherwise, the two
    collections
    \begin{align*}
      \{U & ~|~ U \text{ an $M$-definable basic neighborhood}\} \\
      \{a \cdot U &~|~ U \text{ an $M$-definable basic neighborhood}\}
    \end{align*}
    are equal by Proposition~\ref{basic-nbhds}.\ref{scaling}, so they
    have the same intersection.  But the intersection of the first is
    $I_M$ and the intersection of the second is $a \cdot I_M$.  Thus
    $a \cdot I_M = I_M$.
  \item This follows directly from
    Proposition~\ref{basic-nbhds}.\ref{non-degen}. \qedhere
  \end{enumerate}
\end{proof}

\subsection{The group of infinitesimals}
\begin{definition}
  Let $M$ be a small model and $X \subseteq \Mm$ be $M$-definable.  An
  element $\delta \in \Mm$ is said to \emph{$M$-displace} the set $X$
  if
  \begin{equation*}
    x \in X \cap x \in M \implies x + \delta \notin X.
  \end{equation*}
\end{definition}
\begin{lemma}\label{heirs}
  Let $M \preceq M'$ be an inclusion of small models, and $\varepsilon,
  \varepsilon'$ be elements of $\Mm$.  Suppose that $\tp(\varepsilon'/M')$
  is an heir of $\tp(\varepsilon/M)$.
  \begin{enumerate}
  \item If $\varepsilon$ is an $M$-infinitesimal, then $\varepsilon'$ is an
    $M'$-infinitesimal.
  \item If $X \subseteq \Mm$ is $M$-definable and $M$-displaced by
    $\varepsilon$, then $X$ is $M'$-displaced by $\varepsilon'$.
  \end{enumerate}
\end{lemma}
\begin{proof}
  The assumptions imply that $\varepsilon' \equiv_M \varepsilon$, and the
  statements about $\varepsilon$ are $M$-invariant, so we may assume
  $\varepsilon' = \varepsilon$.
  \begin{enumerate}
  \item Suppose $\varepsilon$ fails to be $M'$-infinitesimal.  Then there
    is a tuple $b \in \dcl(M')$ such that $\phi(\Mm;b)$ is heavy and
    $\varepsilon \notin \phi(\Mm;b) \mininf \phi(\Mm;b)$.  These
    conditions on $b$ are $M\varepsilon$-definable.  Because
    $\tp(b/M\varepsilon)$ is finitely satisfiable in $M$, we can find
    such a $b$ in $M$.  Then $\varepsilon$ fails to be $M$-infinitesimal
    because it lies outside $\phi(\Mm;b) \mininf \phi(\Mm;b)$.
  \item Suppose $\varepsilon$ fails to $M'$-displace $X$.  Then there
    exists $b \in X(M')$ such that $b + \varepsilon \in X = X(\Mm)$.
    These conditions on $b$ are $M\varepsilon$-definable, so we can find
    such a $b$ in $M$.  Then $\varepsilon$ fails to $M$-displace $X$. \qedhere
  \end{enumerate}
\end{proof}
\begin{lemma}
  \label{inf-heart}
  Let $M$ be a small model defining a critical coordinate
  configuration.  Let $\varepsilon$ be an $M$-infinitesimal, and $X
  \subseteq \Mm$ be an $M$-definable set that is $M$-displaced by
  $\varepsilon$.  Then $X$ is light.
\end{lemma}
\begin{proof}
  Build a sequence $\varepsilon_0, \varepsilon_1, \ldots$ and $M = M_0
  \preceq M_1 \preceq M_2 \preceq \cdots$ so that for each $i$,
  \begin{itemize}
  \item $M_{i+1} \ni \varepsilon_i$
  \item $\tp(\varepsilon_i/M_i)$ is an heir of $\tp(\varepsilon/M)$.
  \end{itemize}
  By Lemma~\ref{heirs},
  \begin{itemize}
  \item $\varepsilon_i$ is $M_i$-infinitesimal.
  \item The set $X$ is $M_i$-displaced by $\varepsilon_i$.
  \end{itemize}
  For $\alpha$ a string in $\{0,1\}^{< \omega}$, define $X_\alpha$
  recursively as follows:
  \begin{itemize}
  \item $X_{\{\}} = X$.
  \item If $\alpha$ has length $n$, then $X_{\alpha 0} = \{x \in
    X_\alpha ~|~ x + \varepsilon_n \notin X\}$.
  \item If $\alpha$ has length $n$, then $X_{\alpha 1} = \{x \in
    X_\alpha ~|~ x + \varepsilon_n \in X\}$.
  \end{itemize}
  For example
  \begin{align*}
    X_0 = \{x \in X ~|~ &x + \varepsilon_0 \notin X\} \\
    X_1 = \{x \in X ~|~ &x + \varepsilon_0 \in X\} \\
    X_{011} = \{x \in X ~|~ &x + \varepsilon_0 \notin X,\\ &x + \varepsilon_1
    \in X,\\&x + \varepsilon_2 \in X\}.
  \end{align*}
  Note that some of the $X_\alpha$ must be empty, by NIP.  Note also
  that if $\alpha$ has length $n$, then $X_\alpha$ is $M_n$-definable.
  \begin{claim}
    If $X_\alpha$ is heavy, then $X_{\alpha 1}$ is heavy.
  \end{claim}
  \begin{claimproof}
    Let $\alpha$ have length $n$.  Then $X_\alpha$ is heavy and
    $M_n$-definable, and $\varepsilon_n$ is $M_n$-infinitesimal.
    Consequently $X_\alpha \cap (X_\alpha - \varepsilon_n)$ is heavy.
    But
    \begin{equation*}
      X_{\alpha 1} = X_\alpha \cap (X - \varepsilon_n) \supseteq X_\alpha
      \cap (X_\alpha - \varepsilon_n),
    \end{equation*}
    and so $X_{\alpha 1}$ must be heavy.
  \end{claimproof}
  \begin{claim}
    If $X_\alpha$ is heavy, then $X_{\alpha 0}$ is heavy.
  \end{claim}
  \begin{claimproof}
    Let $\alpha$ have length $n$.  The set $X_\alpha$ is
    $M_n$-definable.  Note that
    \begin{equation*}
      x \in X_\alpha(M_n) \implies x \in X(M_n) \implies x +
      \varepsilon_n \notin X(M_n)
    \end{equation*}
    because $X$ is $M_n$-displaced by $\varepsilon_n$.  Therefore
    \begin{equation*}
      X_\alpha(M_n) \subseteq X_{\alpha 0}.
    \end{equation*}
    By Lemma~\ref{something-similar-2} it follows that $X_{\alpha 0}$
    is heavy.
  \end{claimproof}
  If $X = X_{\{\}}$ is heavy, then the two claims imply that every
  $X_\alpha$ is heavy, hence non-empty, for every $\alpha$.  This
  contradicts NIP.
\end{proof}

\begin{lemma}
  \label{inf-subtr}
  Let $M$ be a model defining a critical coordinate configuration.
  Let $\varepsilon_1, \varepsilon_2$ be two $M$-infinitesimals.  Then
  $\varepsilon_1 - \varepsilon_2$ is an $M$-infinitesimal.
\end{lemma}
\begin{proof}
  Let $X$ be an $M$-definable heavy set; we will show that $X \cap (X
  + \varepsilon_1 - \varepsilon_2)$ is heavy.  Note that $X$ is covered by
  the union of the following three sets:
  \begin{align*}
    D_0 &:= \{x \in X ~|~ x + \varepsilon_1 \in X,~ x + \varepsilon_2 \in X\}
    \\
    D_1 & := \{x \in \Mm ~|~ x + \varepsilon_1 \notin X \} \\
    D_2 & := \{x \in \Mm ~|~ x + \varepsilon_2 \notin X \}.
  \end{align*}
  By Lemma~\ref{chernikov-2}, there is a $j \in \{0,1,2\}$ and an
  $M$-definable heavy set $X' \subseteq X$ such that $X'(M) \subseteq
  D_j$.  If $j > 0$, then
  \begin{equation*}
    x \in X'(M) \implies x \in D_j \implies x + \varepsilon_j \notin X
    \implies x + \varepsilon_j \notin X'.
  \end{equation*}
  In other words, $X'$ is $M$-displaced by $\varepsilon_j$.  But $X'$ is
  heavy and $\varepsilon_j$ is an $M$-infinitesimal, so this would
  contradict Lemma~\ref{inf-heart}.  Therefore $j = 0$.  The fact that
  $X'(M) \subseteq D_0$ implies that $D_0$ is heavy, by
  Lemma~\ref{something-similar-2}.  By definition of $D_0$,
  \begin{align*}
    D_0 + \varepsilon_1 &\subseteq X \\
    D_0 + \varepsilon_2 &\subseteq X \\
    D_0 + \varepsilon_1 &\subseteq X + \varepsilon_1 - \varepsilon_2 \\
    D_0 + \varepsilon_1 & \subseteq X \cap (X + \varepsilon_1 - \varepsilon_2).
  \end{align*}
  Heaviness of $D_0$ then implies heaviness of $X \cap (X + \varepsilon_1
  - \varepsilon_2)$.
\end{proof}

\begin{lemma} \label{pull-down}
  Let $U$ be an $M$-definable basic neighborhood.  Then there is an
  $M$-definable basic neighborhood $V$ such that $V - V \subseteq U$.
\end{lemma}
\begin{proof}
  Let $M' \succeq M$ be a larger model that defines a critical
  coordinate configuration.  By Lemma~\ref{inf-subtr} and Remark~\ref{change-of-M},
  \begin{equation*}
    I_{M'} - I_{M'} \subseteq I_{M'} \subseteq I_M \subseteq U.
  \end{equation*}
  The $M'$-definable partial type asserting that
  \begin{align*}
    x & \in I_{M'} \\
    y & \in I_{M'} \\
    x - y & \notin U
  \end{align*}
  is inconsistent.  Because $I_{M'}$ is the directed intersection of
  $M'$-definable basic neighborhoods, there is an $M'$-definable basic
  neighborhood $V'$ such that 
  \begin{equation*}
    x \in V' \wedge y \in V' \implies x - y \in U,
  \end{equation*}
  i.e., $V' - V' \subseteq U$.  In particular, there is a formula
  $\phi(x;z)$ and $c \in \dcl(M')$ such that $\phi(\Mm;c)$ is heavy
  and
  \begin{equation}
    \label{c-cond}
    (\phi(\Mm;c) \mininf \phi(\Mm;c)) - (\phi(\Mm;c) \mininf
    \phi(\Mm;c)) \subseteq U.
  \end{equation}
  But heaviness of $\phi(\Mm;c)$ and condition (\ref{c-cond}) are both
  $M$-definable constraints on $c$, so we may find $c_0 \in \dcl(M)$
  satisfying the same conditions.  Then
  \begin{equation*}
    V := \phi(\Mm;c_0) \mininf \phi(\Mm;c_0)
  \end{equation*}
  is an $M$-definable basic neighborhood such that $V - V \subseteq
  U$.
\end{proof}

\begin{theorem}
  \label{inf-add}
  If $M$ is any model, the set $I_M$ of $M$-infinitesimals is a
  subgroup of the additive group $(\Mm,+)$.
\end{theorem}
\begin{proof}
  The fact that $0 \in I_M$ follows by \ref{basic-infs}, and the fact
  that $I_M - I_M \subseteq I_M$ is a consequence of
  Lemma~\ref{pull-down}.
\end{proof}

\subsection{Consequences}
We mention a few corollaries of Theorem~\ref{inf-add}.
\begin{remark} \label{exists-topology}
  If $M$ is a field of finite dp-rank but not finite Morley rank, then
  there is a (unique) Hausdorff non-discrete group topology on $(M,+)$
  such that
  \begin{enumerate}
  \item The basic neighborhoods $X \mininf X$ form a basis of
    neighborhoods of 0.
  \item For any $\alpha \in M$, the map $x \mapsto x \cdot \alpha$ is
    continuous.
  \end{enumerate}
  Indeed, this follows formally from the fact that basic neighborhoods
  are filtered with intersection $I_M$, and
  \begin{align*}
    I_M - I_M & \subseteq I_M \\
    0 & \in I_M \\
    \{0\} & \ne I_M \\
    I_M \cap M & = \{0\} \\
    M \cdot I_M & \subseteq I_M.
  \end{align*}
  On the other hand, we have not yet shown that there is a definable
  basis of opens, or that multiplication is continuous in general
  (which would require $I_M \cdot I_M \subseteq I_M$).
\end{remark}
Later (Corollary~\ref{ring-topology}), we will see that multiplication
is continuous.

\begin{corollary} \label{minimal-heavy-subgroup}
  Let $J$ be a subgroup of $(\Mm,+)$, type-definable over a small
  model $M$.  Suppose that every $M$-definable set containing $J$ is
  heavy.  Then $J \supseteq I_M$.
\end{corollary}
\begin{proof}
  Let $D$ be any $M$-definable set containing $J$.  We claim $D$
  contains $I_M$.  Because $J - J \subseteq J \subseteq D$ there is,
  by compactness, an $M$-definable set $D'$ such that $D' \supseteq J$
  and $D' - D' \subseteq D$.  By assumption, $D'$ is heavy.  Then $D'
  \mininf D'$ is an $M$-definable basic neighborhood, and so
  \begin{equation*}
    I_M \subseteq D' \mininf D' \subseteq D' - D' \subseteq D.
  \end{equation*}
  As this holds for all $M$-definable $D \supseteq J$, we see that
  $I_M \subseteq J$.
\end{proof}

Because we are in an NIP setting, $G^{00}$ and $G^{000}$ exist for any
type-definable group $G$, by \cite{shelah-g000}, Theorem 1.12 in the abelian case and \cite{gismatullin-g000} in general.
\begin{corollary} \label{triply-connected}
  For any small model $M$, $I_M = I_M^{00} = I_M^{000}$.
\end{corollary}
\begin{proof}
  Let $M' \supseteq M$ be a model containing a representative from
  each coset of $I_M^{000}$ in $I_M$.  Take $a \in I_M$ such that $a$
  does not lie in any $M'$-definable light set; this is possible by
  Theorem~\ref{heavy-light}.\ref{light-union} and
  Proposition~\ref{basic-nbhds}.\ref{heavy-nbhd}.  If $\delta$ is the
  $M'$-definable coset representative for $a + I_M^{000}$, then $a' :=
  a - \delta$ lies in $I_M^{000}$, but not in any $M'$-definable light
  set (by Theorem~\ref{heavy-light}.\ref{light-translate}).  Let $Y$
  be the set of realizations of $\tp(a'/M)$.  By $M$-invariance of
  $I_M^{000}$ we have $Y \subseteq I_M^{000}$.  On the other hand,
  every $M$-definable neighborhood $D \supseteq Y$ is heavy, so $I_M
  \subseteq D \mininf D \subseteq D - D$.  As $Y$ is the directed
  intersection of $M$-definable $D \supseteq Y$, it follows that $I_M
  \subseteq Y - Y \subseteq I_M^{000}$.
\end{proof}

\section{Further comments on henselianity} \label{sec:further-hensel}
Now that we have a type-definable group of infinitesimals, we can
reduce the Shelah conjecture for dp-finite fields of positive
characteristic to the construction of invariant valuation rings.

\begin{lemma}
  \label{full-heavy}
  Let $X$ be a definable set of full dp-rank.  Then $X$ is heavy.
\end{lemma}
\begin{proof}
  Let $n$ be $\dpr(\Mm)$.  Let $(Q_1,\ldots,Q_m,P)$ be a critical
  coordinate configuration of rank $\rho$.  Shrinking $Q_i$, we may
  assume that $(Q_1 \times \cdots \times Q_m) \setminus P$ is narrow.
  Let $M$ be a small model over which $X, Q_1, \ldots, Q_m, P$ are
  defined.  Take a tuple
  \begin{equation*}
    (a,b_1,\ldots,b_m) \in X \times Q_1 \times \cdots \times Q_m
  \end{equation*}
  such that
  \begin{equation*}
    \dpr(a,b_1,\ldots,b_m/M) = \dpr(X) + \dpr(Q_1) + \cdots +
    \dpr(Q_m) = n + \rho.
  \end{equation*}
  By Remark~\ref{faux-independence}, $\dpr(\vec{b}/M) = \rho$.  Let
  $\delta = a - (b_1 + \cdots + b_m)$.  Then
  \begin{equation*}
    n + \rho = \dpr(a, \vec{b}/M) = \dpr(\delta,\vec{b}/M) \le
    \dpr(\vec{b}/\delta M) + \dpr(\delta/M).
  \end{equation*}
  As $\dpr(\delta/M) \le \dpr(\Mm) = n$, it follows that
  \begin{equation*}
    \rho \le \dpr(\vec{b}/\delta M) \le \dpr(\vec{b}/M) = \rho.
  \end{equation*}
  By Proposition~\ref{halfway-there}, $\tp(\vec{b}/\delta M)$ is
  broad.  It follows that the set
  \begin{equation*}
    \{(b_1,\ldots,b_m) \in Q_1 \times \cdots \times Q_m ~|~ b_1 + \cdots
    + b_m + \delta \in X\}
  \end{equation*}
  is broad.  As we arranged $(Q_1 \times \cdots \times Q_m) \setminus
  P$ to be narrow, the set
  \begin{equation*}
    \{(b_1,\ldots,b_m) \in P ~|~ b_1 + \cdots + b_m + \delta \in X\}
  \end{equation*}
  is also broad, and therefore has dp-rank $\rho$.  If $Y$ is the
  target of $(Q_1,\ldots,Q_n,P)$, then
  \begin{equation*}
    \{y \in Y ~|~ y + \delta \in X\}
  \end{equation*}
  has dp-rank $\rho$, as the map $P \to Y$ has finite fibers.
  Therefore $X$ is $Y$-heavy.
\end{proof}
In fact, the converse to Lemma~\ref{full-heavy} is true.  See Theorem~5.9.2 in \cite{prdf2}.
\begin{lemma}
  \label{primality-trick}
  Let $M$ be a small model and $\mathcal{O}$ be a non-trivial
  $M$-invariant valuation ring.  Then the maximal ideal $\mathfrak{m}$
  contains $I_M$, and the subideal of $\mathcal{O}$ generated by $I_M$
  is prime.
\end{lemma}
\begin{proof}
  Let $n = \dpr(\Mm)$.  Take an element $a$ of positive valuation.
  Take an element $b \in \Mm$ with $\dpr(b/aM) = n$.  Replacing $b$
  with $a/b$ if necessary, we may assume $b \in \mathfrak{m}$.  Let
  $Y$ be the type-definable set of realizations of $\tp(b/M)$.  Note
  $\dpr(Y) \ge \dpr(b/aM) = n$, so $\dpr(Y) = n$.  Also note that $Y
  \subseteq \mathfrak{m}$ because $\mathfrak{m}$ is $M$-invariant.
  Every $M$-definable set $D \supseteq Y$ is heavy, by
  Lemma~\ref{full-heavy}.  Therefore, $I_M \subseteq D - D$.  As $D$
  is arbitrary,
  \begin{equation*}
    I_M \subseteq Y - Y \subseteq \mathfrak{m} - \mathfrak{m}
    \subseteq \mathfrak{m}.
  \end{equation*}
  Now let $J$ be the ideal of $\mathcal{O}$ generated by $I_M$.
  Evidently $J \le \mathfrak{m}$ and $J$ is $M$-invariant.  Suppose
  $J$ is not prime.  By the characterization of ideals in valuation
  rings in terms of cuts in the value group, non-primality of $J$
  implies that the ideal $J' := J \cdot J$ is a strictly smaller
  non-zero ideal in $\mathcal{O}$, and in particular $J' \not
  \supseteq I_M$ (by choice of $J$).  On the other hand, $J'$ is
  $M$-invariant.  Because $J'$ is non-zero, there exists $c \in
  \Mm^\times$ such that $c \cdot \mathfrak{m} \subseteq J'$.  Let $d$
  be an element of $Y$ such that $\dpr(d/cM) = n$.  Then $\dpr(d \cdot
  c / cM) = \dpr(d / cM) = n$, so $\dpr(d \cdot c / M) = n$.  Let $Z$
  be the set of realizations of $\tp(d \cdot c / M)$.  As $d \cdot c
  \in c \cdot Y \subseteq c \cdot \mathfrak{m} \subseteq J'$, the
  usual argument shows that $I_M \subseteq Z - Z \subseteq J' - J' =
  J'$, a contradiction.  Therefore $J$ is prime.
\end{proof}

\begin{proposition}
  \label{upgrade-to-span}
  Let $M$ be a small model and $\mathcal{O}$ be a non-trivial
  $M$-invariant valuation ring.  Then there is an $M$-invariant
  coarsening $\mathcal{O}' \supseteq \mathcal{O}$ such that $I_M$
  spans $\mathcal{O}'$, in the sense of
  Definition~\ref{span-a-valuation}.
\end{proposition}
\begin{proof}
  It is a general fact that if $\mathfrak{p}$ is a prime ideal in a
  valuation ring $\mathcal{O}$, then $\mathfrak{p}$ is the maximal
  ideal of some (unique) coarsening $\mathcal{O}' \supseteq
  \mathcal{O}$.  In our specific case,
  \begin{align*}
    I_M & \subseteq \mathfrak{p} = \mathcal{O} \cdot I_M \\
    \implies I_M & \subseteq \mathfrak{p} \subseteq \mathcal{O}' \cdot
    I_M \subseteq \mathcal{O}' \cdot \mathfrak{p} = \mathfrak{p} \\
    \implies I_M & \subseteq \mathfrak{p} = \mathcal{O}' \cdot I_M
  \end{align*}
  so $I_M$ spans $\mathcal{O}'$.
\end{proof}

\begin{theorem}
  \label{invariant-valuations-are-enough}
  If $\Mm$ has characteristic $p > 0$ and if at least one non-trivial
  invariant valuation ring exists, then $\Mm$ admits a non-trivial
  henselian valuation.
\end{theorem}
\begin{proof}
  By Proposition~\ref{upgrade-to-span} there is a small model $M$ and
  an $M$-invariant valuation ring $\mathcal{O}$ spanned by $I_M$.  The
  fact that $I_M \ne 0$ prevents $\mathcal{O}$ from being the trivial
  valuation ring.  Moreover, $I_M = I_M^{000}$ by
  Corollary~\ref{triply-connected}.  By
  Proposition~\ref{invariant-henselian} it follows that $\mathcal{O}$
  is henselian.
\end{proof}

\section{Bounds on connected components} \label{sec:bounds}
In this section, $(G,+,\ldots)$ is a monster-model abelian group,
possibly with additional structure, of finite dp-rank $n$.

\begin{fact} \label{baldwin-saxl-variant}
  Let $G_0, \ldots, G_n$ be type-definable subgroups of $G$.  There is
  some $0 \le k \le n$ such that
  \begin{equation*}
    \left( \bigcap_{i = 0}^n G_i \right)^{00} = \left( \bigcap_{i \ne
      k} G_i \right)^{00}.
  \end{equation*}
\end{fact}
This follows by the proof of Proposition~4.5.2 in \cite{CKS}.

\begin{lemma}
  \label{up-bound}
  Let $H$ be a type-definable subgroup of $G$.  There is a cardinal
  $\kappa$ depending only on $H$ and $G$ such that if $H < H' < G$ for
  some type-definable subgroup $H'$, and if $H'/H$ is bounded, then
  $H'/H$ has size at most $\kappa$.  This $\kappa$ continues to work
  in arbitrary elementary extensions.
\end{lemma}
\begin{proof}
  Naming parameters, we may assume that $H$ (but not $H')$ is
  type-definable over $\emptyset$.  By Morley-Erd\H{o}s-Rado there is
  some cardinal $\kappa$ with the following property: for any sequence
  $\{a_\alpha\}_{\alpha < \kappa}$ of elements of $G$, there is some
  $\emptyset$-indiscernible sequence $\{b_i\}_{i \in \Nn}$ such that
  for any $i_1 < \cdots < i_n$ there is $\alpha_1 < \cdots < \alpha_n$
  such that
  \begin{equation*}
    a_{\alpha_1} \cdots a_{\alpha_n} \equiv_\emptyset b_{i_1} \cdots
    b_{i_n}.
  \end{equation*}
  Let $H'$ be a subgroup of $G$, containing $H$, type-definable over
  some small set $A$.  Suppose that $|H'/H| \ge \kappa$.  We claim
  that $H'/H$ is unbounded.  Suppose for the sake of contradiction
  that $|H'/H| < \lambda$ in all elementary extensions.  Take a
  sequence $\{a_\alpha\}_{\alpha < \kappa}$ of elements of $H'$ lying
  in pairwise distinct cosets of $H$.  Let $\{b_i\}_{i \in \Nn}$ be an
  $\emptyset$-indiscernible sequence extracted from the $a_\alpha$ by
  Morley-Erd\H{o}s-Rado.  Because the $a_\alpha$ live in pairwise
  distinct cosets of $H$ and $H$ is $\emptyset$-definable, the $b_i$
  live in pairwise distinct cosets of $H$.  By indiscernibility, there
  is a 0-definable set $D \supseteq H$ such that $b_i - b_j \notin D$
  for $i \ne j$.  Consider the $\ast$-type over $A$ in variables
  $\{x_\alpha\}_{\alpha < \lambda}$ asserting that
  \begin{enumerate}
  \item $x_\alpha \in H'$ for every $\alpha < \lambda$
  \item If $\alpha_1 < \cdots < \alpha_n$, then
    \begin{equation*}
      x_{\alpha_1} \cdots x_{\alpha_n} \equiv_\emptyset b_1 \cdots b_n.
    \end{equation*}
  \end{enumerate}
  This type is consistent.  Indeed, if
  $\Sigma_{\alpha_1,\ldots,\alpha_n}(\vec{x})$ is the sub-type
  asserting that
  \begin{align*}
    x_{\alpha_1}, \ldots, x_{\alpha_n} &\in H' \\
      x_{\alpha_1} \cdots x_{\alpha_n} &\equiv_\emptyset b_1 \cdots b_n
  \end{align*}
  then $\Sigma_{\alpha_1,\ldots,\alpha_n}(\vec{x})$ is satisfied by
  $(a_{\beta_1},\ldots,a_{\beta_n})$ for some well chosen $\beta_i$,
  by virtue of how the $b_i$ were extracted.  Moreover, the full type
  is a filtered union of $\Sigma_{\vec{\alpha}}(\vec{x})$'s, so it is
  consistent.  Let $\{c_\alpha\}_{\alpha < \lambda}$ be a set of
  realizations.  Then every $c_\alpha$ lies in $H'$, but
  \begin{equation*}
    c_{\alpha} - c_{\alpha'} \notin D \supseteq H
  \end{equation*}
  for $\alpha \ne \alpha'$.  Therefore, the $c_\alpha$ lie in pairwise
  distinct cosets of $H$, and $|H'/H| \ge \lambda$, a contradiction.
\end{proof}

\begin{lemma}
  \label{extract-of-mer}
  For any cardinal $\kappa$ there is a cardinal $\tau(\kappa)$ with
  the following property: given any family $\{H_\alpha\}_{\alpha <
    \tau(\kappa)}$ of type-definable subgroups of $G$, there exist
  subsets $S_1, S_2 \subseteq \tau(\kappa)$ such that $S_1$ is finite,
  $|S_2| = \kappa$, and
  \begin{equation*}
    \left( \bigcap_{\alpha \in S_1} H_\alpha \right)^{00} \subseteq
    \bigcap_{\alpha \in S_2} H_\alpha.
  \end{equation*}
\end{lemma}
\begin{proof}
  Without loss of generality $\kappa \ge \aleph_0$.  By the
  Erd\H{o}s-Rado theorem (or something weaker), we can choose
  $\tau(\kappa)$ such that any coloring of the $n+1$-element subsets
  of $\tau(\kappa)$ with $n+1$ colors contains a homogeneous subset of
  cardinality $\kappa^+$.  Now suppose we are given $H_\alpha$ for
  $\alpha < \tau(\kappa)$.  Given $\alpha_1 < \cdots < \alpha_{n+1}$,
  color the set $\{\alpha_1, \ldots, \alpha_{n+1}\}$ with the smallest
  $k \in \{1, \ldots, n+1\}$ such that
  \begin{equation*}
    \left( \bigcap_{i = 1}^{n+1} H_{\alpha_i} \right)^{00} = \left(
    \bigcap_{i = 1}^{k-1} H_{\alpha_i} \cap \bigcap_{i = k + 1}^{n+1}
    H_{\alpha_i} \right)^{00}.
  \end{equation*}
  This is possible by Fact~\ref{baldwin-saxl-variant}.  Passing to a
  homogeneous subset and re-indexing, we get $\{H_\alpha\}_{\alpha <
    \kappa^+}$ such that every $(n+1)$-element set has color $k$ for
  some fixed $k$.  In particular, for any $\alpha_1 < \cdots <
  \alpha_{n+1} < \kappa^+$, we have
  \begin{equation*}
    H_{\alpha_k} \supseteq \left( \bigcap_{i = 1}^{n+1} H_{\alpha_i}
      \right)^{00} = \left( \bigcap_{i = 1}^{k - 1} H_{\alpha_i} \cap
      \bigcap_{i = k + 1}^{n+1} H_{\alpha_i} \right)^{00}.
  \end{equation*}
  Thus, for any $\beta_1 < \beta_2 < \cdots < \beta_{2n+1} < \kappa^+$
  we have
  \begin{equation*}
    H_{\beta_{n+1}} \supseteq \left( \bigcap_{i = n - k + 2}^n
    H_{\beta_i} \cap \bigcap_{i = n + 2}^{2n - k + 2} H_{\beta_i}
    \right)^{00} \supseteq \left ( \bigcap_{i = 1}^n H_{\beta_i} \cap
    \bigcap_{i = n+ 2}^{2n + 1} H_{\beta_i}\right)^{00}
  \end{equation*}
  by taking
  \begin{equation*}
    (\alpha_1, \ldots, \alpha_{n+1}) =
    (\beta_{n-k+2},\ldots,\beta_{2n-k+2}).
  \end{equation*}
  Then, for any $\beta \in [n+1,\kappa]$,
  \begin{equation*}
    H_\beta \supseteq \left( H_1 \cap \cdots \cap H_n \cap H_{\kappa +
      1} \cap H_{\kappa + n}\right)^{00},
  \end{equation*}
  so we may take $S_1 = \{1, \ldots, n, \kappa + 1, \ldots, \kappa +
  n\}$ and $S_2 = [n+1,\kappa]$.
\end{proof}

\begin{theorem} \label{bounding-theorem}
  There is a cardinal $\kappa$, depending only on the ambient group
  $G$, such that for any type-definable subgroup $H < G$, the index of
  $H^{00}$ in $H$ is less than $\kappa$.  This $\kappa$ continues to
  work in arbitrary elementary extensions.
\end{theorem}
\begin{proof}
  Say that a subgroup $K \subseteq G$ is \emph{$\omega$-definable} if
  it is type-definable over a countable set.  Note that if $K$ is
  $\omega$-definable, so is $K^{00}$.  Moreover, if $K_1, K_2$ are
  $\omega$-definable, then so are $K_1 \cap K_2$ and $K_1 + K_2$.
  Also note that if $H$ is any type-definable group, then $H$ is a
  small filtered intersection of $\omega$-definable groups.

  Up to automorphism, there are only a bounded number of
  $\omega$-definable subgroups of $G$, so by Lemma~\ref{up-bound}
  there is some cardinal $\kappa_0$ with the following property: if
  $K$ is an $\omega$-definable group and if $K'$ is a bigger
  type-definable group, then either $|K'/K| < \kappa_0$ or $|K'/K|$ is
  unbounded.
  \begin{claim} \label{diamond-slide-claim}
    If $H$ is a type-definable group and $K$ is an $\omega$-definable
    group containing $H^{00}$, then $|H/(H \cap K)| < \kappa_0$.
  \end{claim}
  \begin{claimproof}
    Note that
    \begin{equation*}
      H^{00} \subseteq H \cap K \subseteq H,
    \end{equation*}
    so $H/(H \cap K)$ is bounded.  On the other hand, $H/(H \cap K)$
    is isomorphic to $(H + K)/K$, which must then have cardinality
    less than $\kappa_0$.
  \end{claimproof}
  Let $\kappa_1 = \tau((2^{\kappa_0})^+)$ where $\tau(-)$ is as in
  Lemma~\ref{extract-of-mer}.
  \begin{claim}
    If $H$ is a type-definable subgroup of $G$, then there are fewer
    than $\kappa_1$ subgroups of the form $H \cap K$ where $K$ is
    $\omega$-definable and $K \supseteq H^{00}$.
  \end{claim}
  \begin{claimproof}
    Otherwise, choose $\{K_\alpha\}_{\alpha \in \kappa_1}$ such that
    $K_\alpha$ is $\omega$-definable, $K_\alpha \supseteq H^{00}$, and
    \begin{equation*}
      H \cap K_\alpha \ne H \cap K_{\alpha'}
    \end{equation*}
    for $\alpha < \alpha' < \kappa_1$.  By Lemma~\ref{extract-of-mer},
    there are subsets $S_1, S_2 \subseteq \kappa_1$ such that $|S_1| <
    \aleph_0$, $|S_2| = (2^{\kappa_0})^+$, and
    \begin{equation*}
      \left( \bigcap_{\alpha \in S_1} K_\alpha \right)^{00} \subseteq
      \bigcap_{\alpha \in S_2} K_\alpha.
    \end{equation*}
    Let $J$ be the left-hand side.  Then $J$ is an $\omega$-definable
    group containing $H^{00}$, so $|H/(H \cap J)| < \kappa_0$ by
    Claim~\ref{diamond-slide-claim}.  Now for any $\alpha \in S_2$,
    \begin{equation*}
      J \subseteq K_\alpha \implies H \cap J \subseteq H \cap K_\alpha
      \subseteq H.
    \end{equation*}
    There are at most $2^{|H/(H \cap J)|} \le 2^{\kappa_0}$ groups
    between $H \cap J$ and $J$, so there are at most $2^{\kappa_0}$
    possibilities for $H \cap K_\alpha$, contradicting the fact that
    $|S_2| > 2^{\kappa_0}$ and the $H \cap K_\alpha$ are pairwise
    distinct for distinct $\alpha$.
  \end{claimproof}
  Now given the claim, we see that the index of $H^{00}$ in $H$ can be
  at most $\kappa_0^{\kappa_1}$.  Indeed, let $\mathcal{S}$ be the
  collection of $\omega$-definable groups $K$ such that $K \supseteq
  H^{00}$, and let $\mathcal{S}'$ be a subcollection containing a
  representative $K$ for every possibility of $H \cap K$.  By the
  second claim, $|\mathcal{S}'| < \kappa_1$.  Every type-definable
  group is an intersection of $\omega$-definable groups, so
  \begin{equation*}
    H^{00} = \bigcap_{K \in \mathcal{S}} K = \bigcap_{K \in
      \mathcal{S}} (H \cap K) = \bigcap_{K \in \mathcal{S}'} (H \cap K).
  \end{equation*}
  Then there is an injective map
  \begin{equation*}
    H/H^{00} \hookrightarrow \prod_{K \in \mathcal{S}'} H/(H \cap K),
  \end{equation*}
  and the right hand size has cardinality at most
  $\kappa_0^{\kappa_1}$.  But $\kappa_0^{\kappa_1}$ is independent of
  $H$.
\end{proof}
\begin{corollary}
  \label{big-enough-model}
  Let $\Mm$ be a field of finite dp-rank.  There is a cardinal
  $\kappa$ with the following property: if $M \preceq \Mm$ is any
  small model of cardinality at least $\kappa$, and if $J$ is a
  type-definable $M$-linear subspace of $\Mm$, then $J = J^{00}$.
  More generally, if $J$ is a type-definable $M$-linear subspace of
  $\Mm^k$, then $J = J^{00}$.
\end{corollary}
Note that we are not assuming $J$ is type-definable over $M$.
\begin{proof}
  Take $\kappa$ as in the Theorem, $M$ a small model of size at least
  $\kappa$, and $J$ a type-definable $M$-linear subspace of $\Mm$.
  For any $\alpha \in \Mm^\times$, we have $(\alpha \cdot J)^{00} =
  \alpha \cdot J^{00}$.  Restricting to $\alpha \in M^\times$, we see
  that $\alpha \cdot J^{00} = J^{00}$.  In other words, $J^{00}$ is an
  $M$-linear subspace itself.  The quotient $J/J^{00}$ naturally has
  the structure of a vector space over $M$.  If it is non-trivial, it
  has cardinality at least $\kappa$, contradicting the choice of
  $\kappa$.  Therefore, $J/J^{00}$ is the trivial vector space, and
  $J^{00} = J$.  For the ``more generally'' claim, apply
  Theorem~\ref{bounding-theorem} to the groups $\Mm^k$ and take the
  supremum of the resulting $\kappa$.
\end{proof}

\section{Generalities on modular lattices} \label{sec:modular-lats}
Recall that a a lattice is \emph{modular} if the identity
\[ (x \vee a) \wedge b = (x \wedge b) \vee a\]
holds whenever $a \le b$.
Modularity is equivalent to the statement that for any $a, b$, the
interval $[a \wedge b, a]$ is isomorphic as a poset to $[b, a \vee b]$
via the maps
\begin{align*}
  [a \wedge b, a] & \to [b, a \vee b] \\
  x & \mapsto x \vee b
\end{align*}
and
\begin{align*}
  [b, a \vee b] & \mapsto [a \wedge b, a] \\
  x & \mapsto x \wedge a.
\end{align*}
\subsection{Independence} \label{sec:independence}
Let $(P,<)$ be a modular lattice with bottom element $\bot$.
\begin{definition} \label{def-of-ind}
  A finite sequence $a_1, \ldots, a_n$ of elements of $P$ is
  \emph{independent} if $a_k \wedge \bigvee_{i = 1}^{k-1} a_i = \bot$
  for $2 \le k \le n$.
\end{definition}
Note that if $a$ and $b$ are independent (i.e., $a \wedge b = \bot$),
then there is an isomorphism between the poset $[\bot, a]$ and $[b, a
  \vee b]$ as above.
\begin{lemma}\label{3-perm}
  If the sequence $\{a, b, c\}$ is independent, then the sequence
  $\{a, c, b\}$ is independent.
\end{lemma}
\begin{proof}
  By assumption,
  \begin{align*}
    b \wedge a &= \bot \\
    c \wedge (a \vee b) &= \bot.
  \end{align*}
  First note that
  \begin{equation*}
    c \wedge a \le c \wedge (a \vee b) = \bot,
  \end{equation*}
  so $\{a,c\}$ is certainly independent.  Let $x = b \wedge (a \vee
  c)$.  We must show that $x = \bot$.  Otherwise,
  \begin{equation*}
    \bot < x \le b.
  \end{equation*}
  Applying the isomorphism $[\bot, b] \cong [b, a \vee b]$,
  \begin{equation*}
    a < x \vee a \le a \vee b.
  \end{equation*}
  Applying the isomorphism $[\bot, a \vee b] \cong [c, a \vee b \vee
    c]$,
  \begin{equation*}
    a \vee c < x \vee a \vee c \le a \vee b \vee c.
  \end{equation*}
  But $x \vee (a \vee c) = a \vee c$, as $x \le (a \vee c)$.
\end{proof}
\begin{proposition} \label{perm-invar}
  Independence is permutation invariant: if $\{a_1,\ldots,a_n\}$ is
  independent and $\pi$ is a permutation of $[n]$, then
  $\{a_{\pi(1)},\ldots,a_{\pi(n)}\}$ is independent.
\end{proposition}
\begin{proof}
  It suffices to consider the case where $\pi$ is the transposition of
  $k$ and $k+1$.  Let $b = a_1 \vee \cdots \vee a_{k-1}$.  Then we know
  \begin{align*}
    a_k \wedge b &= \bot \\
    a_{k+1} \wedge (b \vee a_k) &= \bot
  \end{align*}
  and we must show
  \begin{align*}
    a_{k+1} \wedge b & \stackrel{?}{=} \bot \\
    a_k \wedge (b \vee a_{k+1}) &\stackrel{?}{=} \bot.
  \end{align*}
  This is exactly Lemma~\ref{3-perm}.
\end{proof}
\begin{proposition} \label{1-collapse}
  Let $c_1, \ldots, c_n$ and $a_1, \ldots, a_m$ be two sequences.  Set
  $a_0 = c_1 \vee \cdots \vee c_n$.  Then the following are
  equivalent:
  \begin{enumerate}
  \item The sequence $c_1, \ldots, c_n$ is independent and the
    sequence $a_0, a_1, \ldots, a_m$ is independent.
  \item The sequence $c_1, \ldots, c_n, a_1, \ldots, a_m$ is
    independent.
  \end{enumerate}
\end{proposition}
\begin{proof}
  Trivial after unrolling the definition.
\end{proof}
\begin{remark}
  \label{1-grouping}
  Let $a_1, \ldots, a_n$ be an independent sequence and $S$ be a
  subset of $[n]$.  Let $b = \bigvee_{i \in S} a_i$ (understood as
  $\bot$ when $S = \emptyset$).  Then $\{b\} \cup \{a_i ~|~ i \notin
  S\}$ is an independent sequence.
\end{remark}
\begin{proof}
  Modulo Proposition~\ref{perm-invar}, this is the (2) $\implies$ (1)
  direction of Proposition~\ref{1-collapse}.
\end{proof}
\begin{lemma}
  \label{grouping}
  Let $a_1, \ldots, a_n$ be an independent sequence.  Let $S_1,
  \ldots, S_m$ be pairwise disjoint subsets of $[n]$.  For $j \in [m]$
  let $b_j = \bigvee_{i \in S_j} a_i$, or $\bot$ if $S_j$ is empty.
  Then $b_1, \ldots, b_m$ is an independent sequence.
\end{lemma}
\begin{proof}
  Iterate Remark~\ref{1-grouping}.
\end{proof}
\begin{lemma}
  \label{3-odd}
  If $a, b, c$ is an independent sequence, then $(a \vee c) \wedge (b
  \vee c) = c$.
\end{lemma}
\begin{proof}
  As $c \le b \vee c$,
  \begin{equation*}
    (a \vee c) \wedge (b \vee c) = (a \wedge (b \vee c)) \vee c = \bot
    \vee c = c,
  \end{equation*}
  where the first equality is by modularity, and the second equality
  is by independence of $\langle b,c,a \rangle$.
\end{proof}
\begin{proposition}
  \label{proto-cube}
  Let $a_1, \ldots, a_n$ be an independent sequence.  For $S \subseteq
  [n]$, let $a_S = \bigvee_{i \in S} a_i$.  Then the following facts
  hold:
  \begin{enumerate}
  \item $a_{S \cup S'} = a_S \vee a_{S'}$.
  \item $a_\emptyset = \bot$.
  \item $a_{S \cap S'} = a_S \wedge a_{S'}$.
  \end{enumerate}
  In other words, $S \mapsto a_S$ is a homomorphism of lower-bounded
  lattices from $\Pow([n])$ to $P$.
\end{proposition}
\begin{proof}
  The first two points are clear by definition.  For the third point, let
  \begin{align*}
    S_1 &= S \setminus S' \\
    S_2 &= S' \setminus S \\
    S_3 &= S \cap S'.
  \end{align*}
  By Lemma~\ref{grouping}, the sequence $a_{S_1}, a_{S_2}, a_{S_3}$ is
  independent.  By Lemma~\ref{3-odd},
  \begin{equation*}
    a_{S \cap S'} = a_{S_3} \stackrel{!}{=} (a_{S_1} \vee a_{S_3})
    \wedge (a_{S_2} \vee a_{S_3}) = a_{S_1 \cup S_3} \wedge a_{S_2
      \cup S_3} = a_S \wedge a_{S'}. \qedhere
  \end{equation*}
\end{proof}
\begin{proposition}
  \label{proto-sharp-cube}
  In Proposition~\ref{proto-cube}, suppose in addition that $a_i >
  \bot$ for $1 \le i \le n$.  Then the map $S \mapsto a_S$ is an isomorphism onto its image:
  \begin{equation*}
    S \subseteq S' \iff a_S \le a_{S'}.
  \end{equation*}
\end{proposition}
\begin{proof}
  The $\implies$ direction follows formally from the fact that $S
  \mapsto a_S$ is a lattice morphism.  Conversely, suppose $S
  \not\subseteq S'$ but $a_S \le a_{S'}$.  Take $i \in S
  \setminus S'$.  Then $a_i \le a_S \le a_{S'}$, so
  \begin{equation*}
    a_i = a_i \wedge a_{S'} = a_{S' \cap \{i\}} = a_\emptyset = \bot,
  \end{equation*}
  contradicting the assumption.
\end{proof}
\begin{lemma}
  \label{down-smashing}
  Let $(P,\le)$ and $(P',\le)$ be two bounded-below modular lattices,
  and $f : P \to P'$ be a map satisfying the following conditions:
  \begin{align*}
    f(x \wedge y) &= f(x) \wedge f(y) \\
    f(\bot_P) &= \bot_{P'}.
  \end{align*}
  If $a_1, \ldots, a_n$ is an independent sequence in $P$ then
  $f(a_1), \ldots, f(a_n)$ is an independent sequence in $P'$.
\end{lemma}
\begin{proof}
  Note that
  \begin{equation*}
    x \le y \iff x = x \wedge y \implies f(x) = f(x) \wedge f(y) \iff
    f(x) \le f(y),
  \end{equation*}
  so $f$ is weakly order-preserving.  In particular, for any $x, y$
  \begin{align*}
    f(x) & \le f(x \vee y) \\
    f(y) & \le f(x \vee y) \\
    f(x) \vee f(y) & \le f(x \vee y).
  \end{align*}
  Therefore for $2 \le k \le n$ we have
  \begin{equation*}
    f(a_k) \wedge \bigvee_{i < k} f(a_i) \le f(a_k) \wedge f\left(
    \bigvee_{i < k} a_i \right) = f\left(a_k \wedge \bigvee_{i < k}
    a_i \right) = f(\bot_P) = \bot_{P'}. \qedhere
  \end{equation*}
\end{proof}

\subsection{Cubes and relative independence} \label{sec:cubes}
We continue to work in a modular lattice $(P,\le)$, but drop the
assumption that a bottom element exists.
\begin{definition}
  Let $b$ be an element of $P$.
  \begin{enumerate}
  \item A sequence $a_1, \ldots, a_n$ of elements of $P$ is
    \emph{independent over $b$} if it is an independent sequence in
    the bounded-below modular lattice $\{x \in P ~|~ x \ge b\}$.
    Equivalently, $a_1, \ldots, a_n$ is independent over $b$ if $a_i
    \ge b$ for each $i$, and
    \begin{equation*}
      a_k \wedge \bigvee_{i < k} a_i = b
    \end{equation*}
    for $2 \le k \le n$.
  \item A sequence $a_1, \ldots, a_n$ of elements of $P$ is
    \emph{co-independent under $b$} if $a_i \le b$ for each $i$, and
    \begin{equation*}
      a_k \vee \bigwedge_{i < k} a_i = b
    \end{equation*}
    for $2 \le k \le n$.
  \end{enumerate}
\end{definition}
\begin{lemma}
  \label{relative-smashing}
  Let $b$ and $c$ be elements of $P$.
  \begin{enumerate}
  \item If $a_1, \ldots, a_n$ is an independent sequence over $b$,
    then $a_1 \wedge c, a_2 \wedge c, \ldots, a_n \wedge c$ is an
    independent sequence over $b \wedge c$.
  \item If $a_1, \ldots, a_n$ is a co-independent sequence under $b$,
    then $a_1 \vee c, a_2 \vee c, \ldots, a_n \vee c$ is a
    co-independent sequence under $b \vee c$.
  \end{enumerate}
\end{lemma}
\begin{proof}
  Part 1 follows from Lemma~\ref{down-smashing} applied to the
  function
  \begin{align*}
    f : \{x \in P ~|~ x \ge b\} & \to \{x \in P ~|~ x \ge b \wedge c\} \\
    x & \mapsto x \wedge c,
  \end{align*}
  and part 2 follows by duality.
\end{proof}
In what follows, we require ``lattice homomorphisms'' to preserve
$\vee$ and $\wedge$, but not necessarily $\top$ and $\bot$ when they
exist.
\begin{definition}
  An \emph{$n$-cube} in $P$ is a family $\{a_S\}_{S \subseteq [n]}$ of
  elements of $P$ such that $S \mapsto a_S$ is a lattice homomorphism
  from $\Pow([n])$ to $P$.  A \emph{strict $n$-cube} is an $n$-cube
  such that this homomorphism is injective.  The \emph{top} and
  \emph{bottom} of an $n$-cube are the elements $a_\emptyset$ and
  $a_{[n]}$, respectively.
\end{definition}
\begin{remark} \label{lattice-hom}
  Lattice homomorphisms are weakly order-preserving, and injective
  lattice homomorphisms are strictly order-preserving.
\end{remark}
\begin{proposition} \label{cubes-and-independence}
  Let $b$ be an element of $P$.
  \begin{enumerate}
  \item If $\{a_S\}_{S \subseteq [n]}$ is an $n$-cube with bottom $b$,
    then $a_1, \ldots, a_n$ is an independent sequence over $b$.
  \item This establishes a bijection from the collection of $n$-cubes
    with bottom $b$ to the collection of independent sequences over
    $b$ of length $n$.
  \item If $a_1, \ldots, a_n$ is an independent sequence over $b$, the
    corresponding $n$-cube is strict if and only if every $a_i$ is
    strictly greater than $b$.
  \end{enumerate}
\end{proposition}
\begin{proof}
  \begin{enumerate}
  \item By definition of $n$-cube and Remark~\ref{lattice-hom},
    \begin{align*}
      a_i & \ge a_\emptyset = b \\
      a_i \wedge \bigvee_{j < i} a_j & = a_i \wedge a_{[i-1]} = a_{i
        \cap [i-1]} = a_\emptyset = b.
    \end{align*}
  \item Injectivity: if $\{a_S\}_{S \subseteq [n]}$ and $\{a'_S\}_{S
    \subseteq [n]}$ are two $n$-cubes such that $a_S = a'_S$ when $|S|
    \le 1$, then
    \begin{equation*}
      a_S = \bigvee_{i \in S} a_i = \bigvee_{i \in S} a'_i = a'_S
    \end{equation*}
    when $|S| > 1$.

    Surjectivity: let $a_1, \ldots, a_n$ be a sequence independent
    over $b$.  By Proposition~\ref{proto-cube} applied to the
    sublattice $\{x \in P ~|~ x \ge b\}$, there is an $n$-cube with
    bottom $b$.
  \item Let $\{a_S\}_{S \subseteq [n]}$ be the corresponding $n$-cube.
    If $a_i > b$ for $i \in [n]$, then the cube is strict by
    Proposition~\ref{proto-sharp-cube}.  Conversely, if the cube is
    strict, then $a_i > a_\emptyset = b$. \qedhere
  \end{enumerate}
\end{proof}
Dually,
\begin{proposition}
  \label{cubes-and-coindependence}
  Let $b$ be an element of $P$.
  \begin{enumerate}
  \item Let $\{a_S\}_{S \subseteq [n]}$ be an $n$-cube with top $b$.
    For $i \in [n]$ let $c_i = a_{[n] \setminus \{i\}}$.  Then $c_1,
    \ldots, c_n$ is a co-independent sequence under $b$.
  \item This establishes a bijection from the collection of $n$-cubes
    with top $b$ to the collection of co-independent sequences under
    $b$ of length $n$.
  \item If $c_1, \ldots, c_n$ is a co-independent sequence under $b$,
    the corresponding $n$-cube is strict if and only if every $c_i$ is
    strictly less than $b$.
  \end{enumerate}
\end{proposition}

\subsection{Subadditive ranks} \label{sec:reduced-rank}
\begin{definition}
  The \emph{reduced rank} of a modular lattice $(P,\le)$ is the
  supremum of $n \in \Zz$ such that a strict $n$-cube exists in $P$,
  or $\infty$ if there is no supremum.  If $a \ge b$ are two elements
  of $P$, the \emph{reduced rank} $\redrk(a/b)$ is the reduced rank of
  the sublattice $[b,a] := \{x \in P ~|~ b \le x \le a\}$.
\end{definition}
\begin{lemma}
  \label{subs-and-products}
  ~ 
  \begin{enumerate}
  \item Let $P_1$ and $P_2$ be two modular lattices.  If the reduced
    rank of $P_1$ is at least $n$ and the reduced rank of $P_2$ is at
    least $m$, then the reduced rank of $P_1 \times P_2$ is at least
    $n + m$.
  \item \label{subs2} If $P_1$ is a modular lattice and $P_2$ is a
    sublattice, the reduced rank of $P_1$ is at least the reduced rank
    of $P_2$.
  \item If $a, b$ are two elements of a modular lattice $(P,\le)$,
    then there is an injective lattice homomorphism
    \begin{align*}
      [a \wedge b, a] \times [a \wedge b, b] &\to [a \wedge b, a \vee b] \\
      (x,y) & \mapsto x \vee y.
    \end{align*}
  \item \label{product-goal} If $a, b$ are two elements of a modular
    lattice $(P,\le)$ such that $\redrk(a/a \wedge b) \ge n$ and
    $\redrk(b/ a \wedge b) \ge m$, then $\redrk(a \vee b/a \wedge b)
    \ge n + m$.
  \end{enumerate}
\end{lemma}
\begin{proof}
  \begin{enumerate}
  \item Let $\{a_S\}_{S \subseteq[n]}$ and $\{b_S\}_{S \subseteq [m]}$
    be a strict $n$-cube in $P_1$ and a strict $m$-cube in $P_2$.
    Then
    \begin{equation*}
      (S,T) \mapsto (a_S,b_T)
    \end{equation*}
    is an injective lattice homomorphism $\Pow([n]) \times \Pow([m]) \to
    P_1 \times P_2$.  But $\Pow([n]) \times \Pow([m])$ is isomorphic to
    $\Pow([n+m])$.
  \item Any $n$-cube in $P_2$ is already an $n$-cube in $P_1$.
  \item The map is clearly a homomorphism of $\vee$-semilattices.
    \begin{claim}
      If $(x, y) \in [a \wedge b, a] \times [a \wedge b, b]$, then $x
      \vee y = (x \vee b) \wedge (y \vee a)$.
    \end{claim}
    \begin{claimproof}
      By modularity,
      \begin{equation*}
        (b \vee x) \wedge a = (b \wedge a) \vee x.
      \end{equation*}
      Simplifying, we see that $a \wedge (x \vee b) = x$.  Because $y
      \le b$, we have $x \vee y \le x \vee b$.  By modularity,
      \begin{equation*}
        (a \vee (x \vee y)) \wedge (x \vee b) = (a \wedge (x \vee b))
        \vee (x \vee y).
      \end{equation*}
      Simplifying, we see that
      \begin{equation*}
        (y \vee a) \wedge (x \vee b) = x \vee x \vee y = x \vee y. \qedhere
      \end{equation*}
    \end{claimproof}
    In light of the claim, the map $(x,y) \mapsto x \vee y$ is the
    composition
    \begin{equation*}
      [a \wedge b, a] \times [a \wedge b, b] \stackrel{\sim}{\to} [b,
        a \vee b] \times [a, a \vee b] \to [a \wedge b, a \vee b]
    \end{equation*}
    where the first map is the product of the poset isomorphisms
    \begin{align*}
      [a \wedge b, a] & \stackrel{\sim}{\to} [b, a \vee b] \\
      [a \wedge b, b] & \stackrel{\sim}{\to} [a, a \vee b]
    \end{align*}
    and the second map is the $\wedge$-semilattice homomorphism
    $(x',y') \mapsto x' \wedge y'$.  As a composition of a poset
    isomorphism and a $\wedge$-semilattice homomorphism, the original
    map in question is a $\wedge$-semilattice homomorphism.  Finally,
    for injectivity, note that the composition of $(x,y) \mapsto x
    \vee y$ with $z \mapsto (z \vee b, z \vee a)$ is
    \begin{equation*}
      (x,y) \mapsto (x \vee b, y \vee a)
    \end{equation*}
    which is the aforementioned isomorphism
    \begin{equation*}
      [a \wedge b, a] \times [a \wedge b, b] \stackrel{\sim}{\to} [b,
        a \vee b] \times [a, a \vee b].
    \end{equation*}
    Therefore $(x,y) \mapsto x \vee y$ is injective on $[a \wedge b,
      a] \times [a \wedge b, b]$.
  \item This follows by combining the previous three points. \qedhere
  \end{enumerate}
\end{proof}
Work in a modular lattice $(P,\le)$.
\begin{lemma}
  \label{subadditive-lemma}
  Let $x \le y \le z$ be three elements of $P$.  If there is a strict
  $n$-cube $\{a_S\}_{S \subseteq [n]}$ in $[x,z]$, then we can write
  $n = m + \ell$ and find a strict $m$-cube in $[x,y]$ and a strict
  $\ell$-cube in $[y,z]$.
\end{lemma}
\begin{proof}
  Passing to the sublattice $[x,z]$, we may assume that $\bot, \top$
  exist and equal $x, z$ respectively.
  \begin{claim} \label{one-or-the-other}
    For any strict inequality $b < c$ in $P$, at least \emph{one} of
    the following strict inequalities holds:
    \begin{align*}
      b \wedge y &< c \wedge y \\
      b \vee y &< c \vee y.
    \end{align*}
  \end{claim}
  \begin{claimproof}
    Otherwise, $b \wedge y = c \wedge y$ and $b \vee y = c \vee y$.
    Then
    \begin{equation*}
      c = (y \vee c) \wedge c = (y \vee b) \wedge c = (y \wedge c)
      \vee b = (y \wedge b) \vee b = b
    \end{equation*}
    where the middle equality is the modular law.
  \end{claimproof}
  Take $S_0 \subseteq [n]$ maximal such that $a_{S_0} \wedge y =
  a_\emptyset \wedge y$.  Applying a permutation, we may assume $S_0 =
  [\ell]$ for some $\ell \le n$.  Let $m = n - \ell$.  Note that there
  are injective lattice homomorphisms
  \begin{align*}
    \Pow([\ell]) &\to P \\
    S &\mapsto a_S \\
    \Pow([n] \setminus [\ell]) &\to P \\
    S &\mapsto a_{S \cup [\ell]}
  \end{align*}
  obtained by restricting the original lattice homomorphism $S \mapsto
  a_S$.  By Proposition~\ref{cubes-and-coindependence} applied to the
  first of these two cubes, there is a co-independent sequence $b_1,
  \ldots, b_\ell$ under $a_{[\ell]}$ given by
  \begin{equation*}
    b_i = a_{[\ell] \setminus \{i\}},
  \end{equation*}
  and $b_i < a_{[\ell]}$ for $i \in [\ell]$.  Similarly, by
  Proposition~\ref{cubes-and-independence} applied to the second of
  these two cubes, there is an independent sequence $c_1, \ldots, c_m$
  over $a_{[\ell]}$ given by
  \begin{equation*}
    c_i = a_{[\ell] \cup \{\ell + i\}},
  \end{equation*}
  and $c_i > a_{[\ell]}$ for $i \in [m]$.  Note that
  \begin{equation*}
    a_\emptyset \wedge y \le b_i \wedge y \le a_{[\ell]} \wedge y \le
    c_j \wedge y
  \end{equation*}
  for $i \in [\ell]$ and $j \in [m]$.  By choice of $S_0$, we in fact
  have
  \begin{equation*}
    a_\emptyset \wedge y = b_i \wedge y = a_{[\ell]} \wedge y < c_j \wedge y.
  \end{equation*}
  Then $b_i \wedge y = a_{[\ell]} \wedge y$ and $b_i < a_{[\ell]}$
  imply $b_i \vee y < a_{[\ell]} \vee y$ by
  Claim~\ref{one-or-the-other}.  Thus
  \begin{align*}
    b_i \vee y & < a_{[\ell]} \vee y \\
    a_{[\ell]} \wedge y & < c_j \wedge y
  \end{align*}
  for all $i \in [\ell]$ and $j \in [m]$.  By
  Lemma~\ref{relative-smashing}, the sequence $\{b_i \vee y\}$ is
  co-independent under $\{a_{[\ell]} \vee y\}$, yielding a strict
  $\ell$-cube in $[y,\top]$.  Similarly, the sequence $\{c_j \wedge
  y\}$ is independent over $\{a_{[\ell]} \wedge y\}$, yielding a
  strict $m$-cube in $[\bot,y]$.
\end{proof}
\begin{corollary}
  \label{subadditive-corollary}
  Let $a \ge b \ge c$ be three elements of a modular lattice
  $(P,\le)$.  If $\redrk(a/b) < \infty$ and $\redrk(b/c) < \infty$,
  then
  \begin{equation*}
    \redrk(a/c) \le \redrk(a/b) + \redrk(b/c) < \infty.
  \end{equation*}
\end{corollary}
\begin{definition}
  \label{def-of-subadd-rank}
  A \emph{subadditive rank} on a modular lattice is a function
  assigning a non-negative integer $\rk(a/b)$ to every pair $a \ge b$,
  satisfying the following axioms
  \begin{enumerate}
  \item \label{not-in-weak} $\rk(a/b) = 0$ if and only if $a = b$.
  \item If $a \ge b \ge c$, then
    \begin{align*}
      \rk(a/c) & \le \rk(a/b) + \rk(b/c) \\
      \rk(a/c) & \ge \rk(a/b) \\
      \rk(a/c) & \ge \rk(b/c).
    \end{align*}
  \item If $a, b$ are arbitrary, then
    \begin{align*}
      \rk(a/a \wedge b) &= \rk(a \vee b/b) \\
      \rk(b/a \wedge b) &= \rk(a \vee b/a) \\
      \rk(a \vee b / a \wedge b) & = \rk(a/a \wedge b) + \rk(b/a \wedge b).
    \end{align*}
  \end{enumerate}
  A \emph{weak subadditive rank} is a function satisfying the axioms
  other than (\ref{not-in-weak}).
\end{definition}
\begin{speculation} \label{abelian-category-motivation}
  We could also define the notion of a subadditive rank on an abelian
  category.  This should be a function $\rk(-)$ from objects to
  nonnegative integers satisfying the axioms:
  \begin{enumerate}
  \item $\rk(A) = 0$ if and only if $A = 0$.
  \item If $f : A \to B$ is epic (resp. monic) then $\rk(A) \ge
    \rk(B)$ (resp. $\rk(A) \le \rk(B)$).
  \item In a short exactly sequence
    \begin{equation*}
      0 \to A \to B \to C \to 0,
    \end{equation*}
    we have $\rk(B) \le \rk(A) + \rk(C)$ with equality if the sequence
    is split.
  \end{enumerate}
  A weak subadditive rank would satisfy all the axioms except the
  first.  With these definitions, a weak or strong subadditive rank on
  an abelian category $\mathcal{C}$ should induce a weak or strong
  subadditive rank on the subobject lattice of any $A \in
  \mathcal{C}$.\footnote{These ideas are developed further in
    \cite{prdf3}, Appendix E.}
\end{speculation}
\begin{remark}
  \label{subadditive-rk-in-cubes}
  If $\rk(-/-)$ is a subadditive rank and $a_1, a_2, \ldots, a_n$ is
  relatively independent over $b$, then
  \begin{equation*}
    \rk(a_1 \vee \cdots \vee a_n/b) = \sum_{i = 1}^n \rk(a_i/b),
  \end{equation*}
  by induction on $n$.
\end{remark}
\begin{speculation} \label{dp-rank-example}
  Let $K$ be a $\kappa$-saturated field of dp-rank $n < \infty$.  Let
  $P$ be the lattice of subgroups of $(K,+)$ that are type-definable
  over sets of size less than $\kappa$.  There is probably a weak
  subadditive rank on $P$ such that $\rk_{dp}(A/B)$ is the dp-rank of
  the hyper-definable set $A/B$.  To verify this, we would merely need
  to double-check that subadditivity of dp-rank continues to hold on
  hyperimaginaries.  Although this example motivates the notion of
  ``subadditive rank,'' we will never specifically need it.
\end{speculation}
\begin{speculation}
  \begin{enumerate}
  \item Let $P$ be a modular lattice with a weak subadditive rank.
    Define $x \approx y$ if $\rk(x/x \wedge y) = \rk(y/x \wedge y) =
    0$.  Then $\approx$ is probably a lattice congruence, so we can
    form the quotient lattice $P/\approx$.  The weak rank on $P$
    should induce a (non-weak) subadditive rank on $P/\approx$.
  \item In the example of Remark~\ref{dp-rank-example}, $G \approx H$
    if and only if $G^{00} = H^{00}$, and $P/\approx$ is presumably
    isomorphic to the subposet
    \begin{equation*}
      P' := \{G \in P ~|~ G = G^{00}\},
    \end{equation*}
    which is a lattice under the operations
    \begin{align*}
      G \vee H & := G + H \\
      G \wedge H & := (G \cap H)^{00}.
    \end{align*}
    The fact that dp-rank is weak and $P'$ is not a sublattice of $P$
    is an annoyance that motivates \S\ref{sec:bounds}.  By taking a
    large enough small model $K_0 \preceq K$ and restricting to the
    sublattice
    \begin{equation*}
      P'' := \{G \in P ~|~ G \text{ is a $K_0$-linear subspace of } K\}
    \end{equation*}
    the equation $G = G^{00}$ is automatic, and dp-rank presumably
    induces a (non-weak) subadditive rank on $P''$.
  \end{enumerate}
\end{speculation}
\begin{proposition} \label{redrk-subadd}
  Let $(P,\le)$ be a modular lattice.
  \begin{enumerate}
  \item If the reduced rank $\redrk(a/b)$ is finite for all pairs $a
    \ge b$, then $\redrk$ is a (non-weak) subadditive rank.
  \item If there is a subadditive rank $\rk$ on $P$, then $\redrk(a/b)
    \le \rk(a/b) < \infty$ for all $a \ge b$.
  \end{enumerate}
\end{proposition}
\begin{proof}
  \begin{enumerate}
  \item Assume that $\redrk(a/b)$ is finite for all pairs $a \ge b$.
    We verify the definition of subadditive rank.
    \begin{itemize}
    \item Note that a strict 1-cube is just a pair $(x,y)$ with $x <
      y$.  Therefore, $\redrk(a/b) \ge 1$ if and only if there exist
      $x, y \in [b,a]$ such that $x < y$, if and only if $a > b$.
      Conversely, $\redrk(a/b) = 0$ if and only if $a = b$.
    \item Suppose $a \ge b \ge c$.  The inequalities
      \begin{align*}
        \redrk(a/c) & \ge \redrk(a/b) \\
        \redrk(a/c) & \ge \redrk(b/c)
      \end{align*}
      follow by Lemma~\ref{subs-and-products}.\ref{subs2} applied to
      the inclusions $[c,b] \subseteq [c,a]$ and $[b,a] \subseteq
      [c,a]$.  The requirement that $\redrk(a/c) \le \redrk(a/b) +
      \redrk(b/c)$ is Corollary~\ref{subadditive-corollary}.
    \item Let $a, b$ be arbitrary.  By modularity, we have lattice
      isomorphisms
      \begin{align*}
        [a \wedge b, a] & \cong [b, a \vee b] \\
        [a \wedge b, b] & \cong [a, a \vee b],
      \end{align*}
      so certainly
      \begin{align*}
        \redrk(a/a \wedge b) &= \redrk(a \vee b/b) \\
        \redrk(b/a \wedge b) &= \redrk(a \vee b/a).
      \end{align*}
      By subadditivity, it remains to show that
      \begin{align*}
        \redrk(a \vee b / a \wedge b) \ge \redrk(a / a \wedge b) +
        \redrk(b / a \wedge b).
      \end{align*}
      This is Lemma~\ref{subs-and-products}.\ref{product-goal}.
    \end{itemize}
  \item Suppose $\redrk(a/b) \ge n$; we will show $\rk(a/b) \ge n$.
    By definition of reduced rank, there is a strict $n$-cube
    $\{c_S\}_{S \subseteq [n]}$ in the interval $[b, a]$.  By
    definition of strict $n$-cube, we have
    \begin{align*}
      c_i \wedge c_{[i-1]} &= c_\emptyset \\
      c_i \vee c_{[i-1]} &= c_{[i]} \\
      c_i &> c_\emptyset
    \end{align*}
    for $0 < i \le n$.  By the axioms of subadditive rank,
    \begin{equation*}
      \rk(c_{[i]}/c_\emptyset) = \rk(c_{[i-1]}/c_\emptyset) +
      \rk(c_i/c_\emptyset) > \rk(c_{[i-1]}/c_\emptyset).
    \end{equation*}
    By induction on $i$, we see that $\rk(c_{[i]}/c_\emptyset) \ge i$.
    In particular,
    \begin{equation*}
      \rk(a/b) \ge \rk(c_{[n]}/c_\emptyset) \ge n. \qedhere
    \end{equation*}
  \end{enumerate}
\end{proof}
\begin{corollary}
  A modular lattice $(P,\le)$ admits a subadditive rank if and only if
  $\redrk(a/b) < \infty$ for all pairs $a \ge b$, in which case
  $\redrk(a/b)$ is the unique minimum subadditive rank.
\end{corollary}
\begin{speculation}
  \label{expected-properties}
  Let $K$ be a $\kappa$-saturated field of dp-rank $n$, and $P$ be the
  lattice of subgroups of $(K,+)$ that are type-definable over sets of
  size less than $\kappa$.
  \begin{enumerate}
  \item $P$ need not have finite reduced rank.  For example, if $K$
    has positive characteristic, there are arbitrarily big cubes made
    of \emph{finite} subgroups of $(K,+)$.
  \item Let $P'$ be the subposet of $G \in P$ such that $G = G^{00}$.
    Then $P'$ should be a modular lattice with lattice operations
    given by $G \wedge H = (G \cap H)^{00}$ and $G \vee H = G + H$.
    The map $G \mapsto G^{00}$ should be a surjective lattice
    homomorphism from $P$ to $P'$.  The lattice $P'$ should have
    reduced rank at most $n$ by Fact~\ref{baldwin-saxl-variant}.
  \end{enumerate}
  Moreover, the same facts should hold when $K$ is a field of finite
  \emph{burden} $n$.  However, $G^{00}$ can fail to exist in this
  case, and we can only realize $P'$ as the quotient of $P$ modulo the
  00-commensurability equivalence relation where $G \approx H$ iff
  $G/(G \cap H)$ and $H/(G \cap H)$ are bounded.\footnote{These claims
    have now been verified in \cite{prdf3}, \S 6.4.}  In particular,
  we get a weak subadditive rank on the lattice $P$ and a subadditive
  rank on the lattice $P'$ without needing the conjectural
  subadditivity of burden.
\end{speculation}
We will not directly use Remark~\ref{expected-properties}.  Instead,
we will take the following variant approach which avoids all $G^{00}$
issues:
\begin{proposition} \label{reduced-rank-vs-dp-rk}
  Let $\Mm$ be a field, possibly with additional structure.  Assume
  $\Mm$ is a monster model, and $\dpr(\Mm) = n < \infty$.  Let $M_0$
  be a small submodel as in Corollary~\ref{big-enough-model}.  Let $P$
  be the modular lattice of $M_0$-linear subspaces of $\Mm$,
  type-definable over small\footnote{Here and in what follows, we will
    be a bit sloppy with what exactly ``small'' means, but this can
    probably be fixed.} parameter sets.  Then $P$ has reduced rank at
  most $n$.
\end{proposition}
\begin{proof}
  Otherwise, there is an $(n+1)$-cube $\{H_S\}_{S \subseteq [n+1]}$ of
  type-definable subgroups which happen to be $M_0$-linear subspaces
  of $\Mm$.  For $i = 1, \ldots, n+1$, let $G_i = H_{[n+1] \setminus
    \{i\}}$.  Note that
  \begin{align*}
    \bigcap_{i = 1}^{n+1} G_i &= H_\emptyset \\
    \bigcap_{i \ne k} G_i &= H_k,
  \end{align*}
  for any $k \in [n+1]$, by definition of $(n+1)$-cube.  Because the
  cube is strict, we have $H_\emptyset \subsetneq H_k$, so
  \begin{equation*}
    \bigcap_{i = 1}^{n+1} G_i \subsetneq \bigcap_{i \ne k} G_i
  \end{equation*}
  for each $k$.  Both sides are $M_0$-linear type-definable subspaces.
  By choice of $M_0$ (i.e., by Corollary~\ref{big-enough-model}), for
  every $k$ we have
  \begin{equation*}
    \left( \bigcap_{i = 1}^{n+1} G_i\right)^{00} = \bigcap_{i = 1}^{n+1} G_i
    \subsetneq \bigcap_{i \ne k} G_i = \left( \bigcap_{i \ne k} G_i
    \right)^{00}
  \end{equation*}
  contradicting Fact~\ref{baldwin-saxl-variant}.
\end{proof}

\subsection{The modular pregeometry on quasi-atoms} \label{sec:geometry}
Work in a modular lattice $(P,\le)$ with bottom element $\bot$.
\begin{definition}
  An element $a > \bot$ is a \emph{quasi-atom} if the set
  \begin{equation*}
    \{x \in P ~|~ \bot < x \le a\}
  \end{equation*}
  is downwards directed.  Equivalently, if $\bot < x \le a$ and $\bot
  < y \le a$, then $\bot < x \wedge y$.
\end{definition}
\begin{proposition}
  \label{quasi-atoms-basics}
  ~
  \begin{enumerate}
  \item If $a$ is a quasi-atom and $\bot < a' < a$ then $a'$ is a
    quasi-atom.
  \item Every atom is a quasi-atom.
  \item \label{sym-trans} Non-independence is an equivalence relation
    on quasi-atoms.  Two quasi-atoms $a, a'$ are equivalent if $a
    \wedge a' > \bot$.  This equivalence relation is the transitive
    symmetric closure of $\ge$ on the set set of quasi-atoms.
  \item \label{enough-quasiats} Suppose a subadditive rank exists.  For every $a > \bot$,
    there is a quasi-atom $a' \le a$.
  \end{enumerate}
\end{proposition}
\begin{proof}
  \begin{enumerate}
  \item If $\bot < x, y$ and $x,y \le a'$ then $x,y \le a$ so $x
    \wedge y > \bot$.
  \item If $a$ is an atom, then $\{x \in P ~|~ \bot < x \le a\} = \{a\}$
    which is trivially directed.
  \item By definition, two arbitrary elements $a, b$ are
    non-independent exactly if $a \wedge b > \bot$.  Now restrict to
    the case of quasi-atoms.  The relation $a \wedge a' > \bot$ is
    clearly symmetric.  It is reflexive because $a > \bot$ is part of
    the definition of quasi-atom.  For transitivity, suppose that $a,
    a', a''$ are quasi-atoms, and $a \wedge a' > \bot < a' \wedge
    a''$.  Then
    \begin{equation*}
      a \wedge a'' \ge (a \wedge a') \wedge (a'' \wedge a') > \bot
    \end{equation*}
    where the strict inequality holds because $a'$ is a quasi-atom and
    \begin{align*}
      \bot &< a \wedge a' \le a' \\
      \bot &< a'' \wedge a' \le a'.
    \end{align*}
    Thus transitivity holds.  This equivalence relation contains the
    restriction of $\ge$ to quasi-atoms, because if $a$ and $a'$ are
    quasi-atoms such that $a \ge a'$, then $a \wedge a' = a' > \bot$.
    Finally, it is contained in the transitive symmetric closure of
    $\ge$ because if $a \wedge a' > \bot$, then $a'' := a \wedge a'$
    is a quasi-atom by part 1, and $a \ge a'' \le a$.
  \item Among the elements of the set $\{x \in P ~|~ \bot < x \le a\}$,
    choose an element $a'$ such that $\rk(a'/\bot)$ is minimal.  We
    claim that $a'$ is a quasi-atom.  Otherwise, there exist $x, y \le
    a'$ such that $x, y > \bot = x \wedge y$.  Then
    \begin{equation*}
      \rk(a'/\bot) \ge \rk(x \vee y / \bot) = \rk(x/\bot) + \rk(y/\bot).
    \end{equation*}
    As $y > \bot$, it follows that $\rk(y/\bot) > 0$ and so
    $\rk(x/\bot) < \rk(a'/\bot)$, contradicting the choice of $a'$. \qedhere
  \end{enumerate}
\end{proof}
\begin{speculation}
  In particular, if $K$ is a field of finite dp-rank and $P$ is the
  lattice of type-definable subgroups $G \le (K,+)$ such that $G =
  G^{00}$, then every non-zero element $G \in P$ has a quasi-atomic
  subgroup $G' < G$.  The same fact might hold in strongly dependent
  fields, even if a finite subadditive rank is lacking.  Otherwise,
  one can take a counterexample $G$ and split off an infinite
  independent sequence $H_1, H_2, \ldots$ of subgroups of $G$.  This
  sequence might violate strong dependence.
\end{speculation}
Say that two quasi-atoms $a, a'$ are \emph{equivalent} if $a \wedge a'
> \bot$, as in Proposition~\ref{quasi-atoms-basics}.\ref{sym-trans}.
\begin{lemma}
  \label{proto-respect-equivalence}
  Let $a, a'$ be equivalent quasi-atoms, and $b$ be an arbitrary
  element.  Then $\{a,b\}$ is independent if and only if $\{a',b\}$ is
  independent.
\end{lemma}
\begin{proof}
  By Proposition~\ref{quasi-atoms-basics}.\ref{sym-trans}, we may
  assume $a \le a'$.  Then
  \begin{equation*}
    a' \wedge b = \bot \implies a \wedge b = \bot
  \end{equation*}
  trivially.  Conversely, suppose $a \wedge b = \bot$ but $a' \wedge b
  > \bot$.  Note that
  \begin{equation*}
    \{a, a' \wedge b\} \subseteq \{x \in P ~|~ \bot < x \le a'\}.
  \end{equation*}
  As $a'$ is a quasi-atom, we have
  \begin{equation*}
    \bot = a \wedge b = (a \wedge a') \wedge b = a \wedge (a' \wedge
    b) \stackrel{!}{>} \bot. \qedhere
  \end{equation*}
\end{proof}
\begin{lemma} \label{respect-equivalence}
  Let $a_1, \ldots, a_n$ be an independent sequence.  Suppose that
  $a_k$ is a quasi-atom for some $k$.  Let $a'_k$ be an equivalent
  quasi-atom.  Then $a_1, \ldots, a_{k-1}, a'_k, a_{k+1}, \ldots, a_n$
  is an independent sequence.
\end{lemma}
\begin{proof}
  By Proposition~\ref{perm-invar} we may assume $k = n$.  Let $b = a_1
  \vee \cdots \vee a_{n-1}$.  Then $\{b, a_n\}$ is independent and it
  suffices to show that $\{b, a'_n\}$ is independent.  This is
  Lemma~\ref{proto-respect-equivalence}.
\end{proof}

In the next few lemmas, we adopt the following notation.  The set of
quasi-atoms will be denoted by $Q$.  If $x \in P$, then $V(x)$ will
denote the set of $a \in Q$ such that $a \wedge x > \bot$.  Thus $Q
\setminus V(x)$ is the set of quasi-atoms independent from $x$.
\begin{lemma} \label{2-ary-v}
  $V(x) \cap V(y) = V(x \wedge y)$ for any $x, y \in P$.
\end{lemma}
\begin{proof}
  Let $a$ be a quasi-atom.  Note that
  \begin{equation*}
    (a \wedge x) \wedge (a \wedge y) > \bot \iff \left( (a \wedge x
    > \bot) \text{ and } (a \wedge y > \bot)\right);
  \end{equation*}
  the left-to-right implication holds generally and the
  right-to-left implication holds because $a$ is a quasi-atom.  The
  left hand side is equivalent to $a \in V(x \wedge y)$ and the
  right hand side is equivalent to $a \in V(x) \cap V(y)$.
\end{proof}
\begin{lemma}
  \label{main-lemma-for-pregeometry}
  Let $S$ be a finite subset of $Q$.
  \begin{enumerate}
  \item There is a closure operation on $S$ whose closed sets are
    exactly the sets of the form $V(x) \cap S$.
  \item \label{tfae} Let $a_1, \ldots, a_n$ be a sequence of elements
    of $S$, independent in the sense of Definition~\ref{def-of-ind}.
    The following are equivalent for $b \in S$:
    \begin{enumerate}
    \item \label{option-a} $b$ lies in the closure of
      $\{a_1,\ldots,a_n\}$.
    \item \label{option-b} $b \in V(a_1 \vee \cdots \vee a_n)$.
    \item \label{option-c} The sequence $a_1, \ldots, a_n, b$ is
      \emph{not} independent.
    \end{enumerate}
  \item If $a_1, \ldots, a_n$ is a sequence in $S$, there is an
    subsequence $b_1, \ldots, b_m$ having the same closure, and
    independent in the sense of Definition~\ref{def-of-ind}.
  \item The closure operation on $S$ satisfies exchange, i.e., it is a
    pregeometry.
  \item A sequence $a_1, \ldots, a_n$ in $S$ is independent with
    respect to the pregeometry if and only if it is independent in the
    sense of Definition~\ref{def-of-ind}.
  \end{enumerate}
\end{lemma}
\begin{proof}
  \begin{enumerate}
  \item Let $\mathcal{C}$ be the collection of sets of the form $V(x)
    \cap S$ for $x \in P$.  Then $\mathcal{C}$ is closed under 2-ary
    intersections by Lemma~\ref{2-ary-v}.  It is closed under 0-ary
    intersections because if $\{a_1,\ldots,a_n\}$ is an enumeration of
    $S$, then every $a_i$ lies in $V(a_1 \vee \cdots \vee a_n)$.
    Therefore, $\mathcal{C}$ is the set of closed sets with respect to
    some closure operation on $S$.
  \item The set $V(a_1 \vee \cdots \vee a_n) \cap S$ is closed and
    contains each $a_i$, so it contains the closure of
    $\{a_1,\ldots,a_n\}$.  This shows that
    (\ref{option-a})$\implies$(\ref{option-b}).  Since $a_1, \ldots,
    a_n$ is independent, the sequence $a_1, \ldots, a_n, b$ is
    non-independent if and only if $b \wedge (a_1 \vee \cdots \vee
    a_n) > \bot$, if and only if $b \in V(a_1 \vee \cdots \vee a_n)$.
    Thus (\ref{option-b})$\iff$(\ref{option-c}).  Finally, suppose
    (\ref{option-c}) holds.  Let $x$ be an element of $P$ such that
    $V(x) \cap S$ equals the closure of $\{a_1,\ldots,a_n\}$.  Let
    $a_i' = a_i \wedge x$.  By definition of $V(x)$ these elements are
    non-zero, and $a_i'$ is a quasi-atom equivalent to $a_i$ by
    Proposition~\ref{quasi-atoms-basics}.  By
    Lemma~\ref{respect-equivalence}, the sequence $a_1', a_2', \ldots,
    a_n', b$ is \emph{not} independent, but $a_1', a_2', \ldots, a_n'$
    \emph{is} independent.  Therefore,
    \begin{equation*}
      \bot < b \wedge (a_1' \vee \cdots \vee a_n') \le b \wedge x
    \end{equation*}
    because $a_i' \le x$.  Then $b \in V(x) \cap S$, so
    (\ref{option-a}) holds.
  \item Let $b_1, \ldots, b_m$ be the subsequence consisting of those
    $a_i$ which do not lie in the closure of the preceding $a_i$'s.
    This subsequence has the same closure as the original sequence.
    The sequence $\{b_1,\ldots,b_j\}$ is independent for $j \le m$, by
    induction on $j$, using the equivalence
    (\ref{option-a})$\iff$(\ref{option-c}) in the previous point.
  \item Let $a_1, \ldots, a_n, b, c$ be elements of $S$.  Suppose that
    $b$ is \emph{not} in the closure of $\{a_1, \ldots, a_n\}$, and
    $c$ is \emph{not} in the closure of $\{a_1,\ldots,a_n,b\}$.  We
    must show that $b$ is \emph{not} in the closure of
    $\{a_1,\ldots,a_n,c\}$.  By the previous point, we may assume that
    the $a_i$'s are independent.  In this case, the equivalence
    (\ref{option-a})$\iff$(\ref{option-c}) in (\ref{tfae}) immediately
    implies that the sequence $\{a_1,\ldots,a_n,b\}$ is independent,
    and then that $\{a_1,\ldots,a_n,b,c\}$ is independent.  By
    permutation invariance, it follows that $\{a_1,\ldots,a_n,c\}$ and
    $\{a_1,\ldots,a_n,c,b\}$ are independent.  By the equivalence,
    this means that $b$ is not in the closure of
    $\{a_1,\ldots,a_n,c\}$.
  \item We proceed by induction on $n$.  For $n = 1$, note that
    $V(\bot) \cap S = \emptyset$ and so every sequence of length 1 is
    independent with respect to the pregeometry (and vacuously
    independent with respect to the lattice).  Suppose $n > 1$.  Let
    $a_1, \ldots, a_{n-1}, a_n$ be a sequence.  We may assume that
    $a_1, \ldots, a_{n-1}, a_n$ is independent with respect to the
    pregeometry or the lattice.  Either way, $a_1, \ldots, a_{n-1}$ is
    independent with respect to the pregeometry or the lattice, hence
    with respect to the lattice or the pregeometry (by induction).
    Then $a_1, \ldots, a_n$ is independent with respect to the
    pregeometry if and only if $a_n$ is not in the closure of $a_1,
    \ldots, a_{n-1}$.  By (\ref{tfae})'s equivalence, this is the same
    as $a_1, \ldots, a_n$ being independent with respect to the
    lattice. \qedhere
  \end{enumerate}
\end{proof}
\begin{corollary}
  \label{the-pregeometry-exists}
  Let $Q$ be the set of quasi-atoms in $(P,\le)$.
  \begin{enumerate}
  \item There is a finitary pregeometry on the set $Q$ of quasi-atoms,
    characterized by the fact that a finite set $\{a_1,\ldots,a_n\}$
    of quasi-atoms is independent with respect to the pregeometry if
    and only if it is independent in the lattice-theoretic sense of
    Definition~\ref{def-of-ind}.
  \item If $x$ is any element of $P$, then the set $V(x) = \{a \in Q ~|~
    a \wedge x > \bot\}$ is a closed subset of $Q$.
  \item Two quasi-atoms $a, b$ are parallel (i.e., have the same
    closure) if and only if they are equivalent in the sense of
    Proposition~\ref{quasi-atoms-basics}.  Consequently, there is an
    induced \emph{geometry} on the equivalence classes of quasi-atoms.
  \end{enumerate}
\end{corollary}
\begin{proof}
  To specify a finitary closure operation on an infinite set $Q$, it
  suffices to give a closure operation on each finite subset $S
  \subseteq Q$, subject to the compatibility requirements that when
  $S' \subseteq S$, the closed sets on $S'$ are exactly the sets of
  the form $V \cap S'$ for $V$ a closed set on $S$.  We clearly have
  this compatibility.  Furthermore, the induced finitary closure
  operation on $Q$ satisfies exchange if and only if the closure
  operation on each finite $S \subseteq Q$ satisfies exchange.  When
  this holds, a finite set $I$ is independent with respect to the
  pregeometry on $Q$ if and only if it is independent with respect to
  the pregeometry on $S$, for any/every finite $S$ containing $I$.
  Moreover, a set $V \subseteq Q$ is closed if and only if $V \cap S$
  is a closed subset of $S$ for every finite $S \subseteq Q$.
  Therefore, the sets $V(x)$ are certainly closed.  For the final
  point, it follows on general pregeometry grounds that two
  non-degenerate elements $a, b$ are parallel if and only if $\{a,b\}$
  is independent.  Because independence in the pregeometry agrees with
  lattice-theoretic independence, we see that $a, b$ are parallel if
  and only if they are equivalent in the sense of
  Proposition~\ref{quasi-atoms-basics}.
\end{proof}
\begin{lemma}
  \label{3-fold-lemma}
  Let $x, a, b$ be three elements of $P$, such that the sets $\{x,
  a\}, \{x, b\}, \{a,b\}$ are independent, but $\{x, a, b\}$ is not
  independent.  Suppose that $a$ is a quasi-atom.  Then there is a
  quasi-atom $w \le x$ such that $\{w, a, b\}$ is not independent (but
  $\{w, a\}, \{w, b\}, \{a, b\}$ are independent).
\end{lemma}
\begin{proof}
  Let $w = (a \vee b) \wedge x$.  As $\{a, b\}$ is independent but
  $\{a, b, x\}$ is not, we must have $w > \bot$.  As $w \le x$ it
  follows that $\{w, b\}$ and $\{w, a\}$ are independent.  From the
  independence of $\{b, w\}$, there is an isomorphism of posets
  \begin{equation*}
    (\bot,w] \cong (b, b \vee w].
  \end{equation*}
  Now $b \vee w \le a \vee b$ because $w \le a \vee b$ (by definition)
  and $b \le a \vee b$.  Therefore $(b, b \vee w] \subseteq [b, a \vee
      b]$.  By the independence of $a$ and $b$, there is a poset
    isomorphism
  \begin{align*}
    [b, a \vee b] & \stackrel{\sim}{\to} [\bot, a] \\
    x & \mapsto x \wedge a.
  \end{align*}
  This isomorphism induces an isomorphism of the subposets
  \begin{equation*}
    (b, b \vee w] \cong (\bot, a \wedge (b \vee w)].
  \end{equation*}
  In particular, there is a chain of poset isomorphisms
  \begin{equation*}
    (\bot,w] \cong (b, b \vee w] \cong (\bot, a \wedge (b \vee w)].
  \end{equation*}
  The left hand side is non-empty, because $w > \bot$.  Therefore, the
  right hand side is non-empty, and so $a' := a \wedge (b \vee w)$ is
  greater than $\bot$.  On the other hand, $a'$ is less than or equal
  to the quasi-atom $a$, so $a'$ is a quasi-atom.  Therefore, the
  right hand side is a downwards directed poset.  The same must hold
  for the left hand side, implying that $w$ is a quasi-atom.  Note
  that $w \wedge (a \vee b) = w > \bot$, so the sequence $\{a, b, w\}$
  is not independent.
\end{proof}

\begin{proposition} \label{prop:is-modular}
  The pregeometry of Corollary~\ref{the-pregeometry-exists} is
  modular.
\end{proposition}
\begin{proof}
  \begin{claim}
    \label{the-previous-claim}
    Let $a, b, c_1, \ldots, c_n$ be elements of $Q$, with $n > 0$.
    Suppose that $a \in \cl\{b, c_1, \ldots, c_n\}$.  Then there is
    $c' \in \cl\{c_1, \ldots, c_n\}$ such that $a \in \cl\{b, c'\}$.
  \end{claim}
  \begin{claimproof}
    We may assume that the $c_i$ are independent, passing to an
    independent subsequence otherwise\footnote{This preserves $n > 0$
      because no quasi-atom is in the closure of the empty set, as
      sequences of length 1 are always independent}.  We may assume
    that $a \notin \cl\{b\}$, or else take $c' = c_1$.  We may assume
    that $a \notin \cl\{c_1,\ldots,c_n\}$, or else take $c' = a$.
    Then $b \notin \cl\{c_1,\ldots,c_n\}$.  It follows that the
    sequences
    \begin{itemize}
    \item $\{a, b\}$
    \item $\{c_1,\ldots,c_n,a\}$
    \item $\{c_1,\ldots,c_n,b\}$
    \end{itemize}
    are independent, but the sequence $\{c_1, \ldots, c_n, a, b\}$ is
    \emph{not} independent.  Let $x = c_1 \vee \cdots \vee c_n$.  Then
    \begin{align*}
      a \wedge b &= \bot \\
      a \wedge x &= \bot \\
      b \wedge x &= \bot
    \end{align*}
    but $\{a, b, x\}$ is \emph{not} independent.  By
    Lemma~\ref{3-fold-lemma}, there is a quasi-atom $w \le x$ such
    that $\{a, w\}, \{b, w\}, \{a, b\}$ are independent but $\{a, b,
    w\}$ is not. Now $w \le x = c_1 \vee \cdots \vee c_n$, so the
    sequence $\{c_1, \ldots, c_n, w\}$ is not independent.  Therefore
    $w \in \cl\{c_1,\ldots,c_n\}$.  The fact that $\{b, w\}$ is
    independent but $\{a, b, w\}$ is not implies that $a \in \cl\{b,
    w\}$.  Take $c' = w$.
  \end{claimproof}
  Given the claim, modularity follows on abstract grounds.  First, one
  upgrades the claim to the following:
  \begin{claim}
    If $X, Y$ are two non-empty closed subsets of $Q$ and $a \in \cl(X
    \cup Y)$, then there is $x \in X$ and $y \in Y$ such that $a \in
    \cl\{x,y\}$.
  \end{claim}
  \begin{claimproof}
    Because closure is finitary, there exist $b_1, \ldots, b_n \in X$
    and $c_1, \ldots, c_m \in Y$ such that $a \in
    \cl\{b_1,\ldots,b_n,c_1,\ldots,c_m\}$.  By repeated applications
    of the previous claim one can find
    \begin{itemize}
    \item $p_1 \in \cl\{b_2, \ldots, b_n, c_1, \ldots, c_m\}$ such
      that $a \in \cl\{b_1,p_1\}$.
    \item $p_2 \in \cl\{b_3, \ldots, b_n, c_1, \ldots, c_m\}$ such
      that $p_1 \in \cl\{b_2,p_2\}$.
    \item \ldots
    \item $p_k \in \cl\{b_{k+1},\ldots,b_n,c_1,\ldots,c_m\}$ such that
      $p_{k-1} \in \cl\{b_k,p_k\}$.
    \item \ldots
    \item $p_n \in \cl\{c_1,\ldots,c_m\}$ such that $p_{n-1} \in
      \cl\{b_n,p_n\}$.
    \end{itemize}
    Then $\cl\{b_1, b_2, \ldots, b_n, p_n\}$ contains $p_{n-1}$, hence
    also $p_{n-2}, p_{n-3}, \ldots, p_2, p_1,$ and $a$.  Thus $a \in
    \cl\{b_1,b_2,\ldots,b_n,p_n\}$.  By one final application of
    Claim~\ref{the-previous-claim}, there is $q \in
    \cl\{b_1,b_2,\ldots,b_n\}$ such that $a \in \cl\{p_n,q\}$.  Take
    $x = q$ and $y = p_n$.
  \end{claimproof}
  Finally, let $V_1, V_2, V_3$ be closed subsets of $Q$ with $V_1
  \subseteq V_2$.  Write $+$ for the lattice join on the lattice of
  closed sets.  We must show one direction of the modular equation:
  \begin{equation*}
    (V_1 + V_3) \cap V_2 \subseteq V_1 + (V_3 \cap V_2)
  \end{equation*}
  as the $\supseteq$ direction holds in any lattice.  If $V_1$ or
  $V_3$ is empty, the equality is clear, so we may assume both are
  non-empty.  Take $x \in (V_1 + V_3) \cap V_2$.  By the previous
  claim, there are $(y,z) \in V_1 \times V_3$ such that $x \in
  \cl\{y,z\}$.  If $x \in \cl\{y\}$, then $x \in V_1$ so certainly $x
  \in V_1 + (V_3 \cap V_2)$.  Otherwise, $x \in \cl\{y,z\} \setminus
  \cl\{y\}$, so $z \in \cl\{x,y\}$.  Now $x \in V_2$ by choice of $x$,
  and $y \in V_1 \subseteq V_2$.  Therefore, $z \in \cl\{x,y\}
  \subseteq V_2$.  The element $z$ is also in $V_3$, so $z \in (V_3
  \cap V_2)$.  Thus $x \in \cl\{y,z\}$ where $y \in V_1$ and $z \in
  V_3 \cap V_2$.  It follows that $x \in V_1 + (V_3 \cap V_2)$,
  completing the proof of modularity.
\end{proof}
In a later paper (\cite{prdf3}), we will see that equivalence classes
of quasi-atoms in $M$ correspond to atoms in the category $\Pro M$
obtained by formally adding filtered infima to $M$.  The category
$\Pro M$ is itself a modular lattice, and the modular pregeometry
constructed in Corollary~\ref{the-pregeometry-exists} and
Proposition~\ref{prop:is-modular} comes from the usual modular
geometry on atoms.  See \S7.1-7.2 of \cite{prdf3} for details.

\begin{speculation}
  Let $P$ be the modular lattice of type-definable subgroups of
  $(K,+)$ in a dp-finite field $K$.  Or let $P$ be an interval inside
  that lattice.  In these cases, the modular geometry on quasi-atoms
  in $P$ should consist of finitely many pieces, each of which is
  canonically a projective space over a division ring.  In particular,
  \begin{itemize}
  \item Exotic non-Desarguesian projective planes shouldn't appear.
  \item Even if a component has rank 1, it should still carry the
    structure of $\Pp^1(D)$ for some division ring $D$.
  \end{itemize}
  These nice properties should hold because of the fact that we can
  embed $P$ into a larger lattice $P_n$ of type-definable subgroups of
  $K^n$.  The existence of the family of $P_n$'s for $n \ge 1$ should
  ensure that all the connected components of the modular geometry can
  be embedded into connected modular geometries of arbitrarily high
  rank, forcing Desargues' theorem to hold.  Similar statements should
  hold in the abstract setting of
  Remark~\ref{abelian-category-motivation}.\footnote{All these claims
    have now been verified in \cite{prdf3} \S2.2 and \S 7.4.}
\end{speculation}

\subsection{Miscellaneous facts}

\begin{proposition} \label{covering-lemma-home}
  Let $(P,\le)$ be a modular lattice with top and bottom elements
  $\top$ and $\bot$.  Suppose that $P$ admits a subadditive rank.
  \begin{enumerate}
  \item The rank of the modular pregeometry on quasi-atoms is at most
    the reduced rank $\redrk(\top/\bot)$.  In particular, it is
    finite.
  \item Every closed set is of the form $V(x)$ for some $x$.
  \item Let $x$ be any element of $P$, and $y_1, \ldots, y_n$ be a
    basis for $V(x)$.  Let $z_1, \ldots, z_m$ be a sequence of
    quasi-atoms.  Then $\{x, z_1, \ldots, z_m\}$ is independent if and
    only if \[\{y_1, \ldots, y_n, z_1, \ldots, z_m\}\] is independent.
  \item \label{covering-lemma} If $x$ is any element of $P$, there is a pregeometry basis of
    the form $y_1, \ldots, y_n, z_1, \ldots, z_m$ where
    \begin{itemize}
    \item Each $y_i \le x$.
    \item The set $\{y_1,\ldots,y_n\}$ is a basis for $V(x)$.
    \item The set $\{x, z_1, \ldots, z_m\}$ is independent.
    \end{itemize}
  \end{enumerate}
\end{proposition}
\begin{proof}
  \begin{enumerate}
  \item If the pregeometry rank is at least $n$, then there is a
    sequence of independent quasi-atoms $a_1, \ldots, a_n$.  Each
    $a_i$ is greater than $\bot$, so this determines a strict $n$-cube
    by Proposition~\ref{cubes-and-independence}.
  \item Let $V$ be a closed set.  Because the pregeometry rank is
    finite, we can find a basis $a_1, \ldots, a_n$ for $V$.  Note that
    $\{a_1,\ldots,a_n\}$ is an independent set.  Let $x = a_1 \vee
    \cdots \vee a_n$.  If $b$ is any element of $Q$, then the
    following statements are equivalent
    \begin{itemize}
    \item $b$ is in $V$.
    \item $b$ is in the closure of $\{a_1,\ldots,a_n\}$.
    \item The set $\{a_1,\ldots,a_n,b\}$ is \emph{not} independent.
    \item $b \wedge x > \bot$.
    \item $b \in V(x)$.
    \end{itemize}
    Therefore $V = V(x)$.
  \item Let $y'_i = y_i \wedge x$.  By definition of $V(x)$, each
    $y'_i > \bot$, and so $y'_i$ is a quasi-atom equivalent to $y_i$.
    By Lemma~\ref{respect-equivalence} we may replace $y_i$ with
    $y'_i$ and assume that each $y_i \le x$.  Let $x' = y_1 \vee
    \cdots \vee y_n \le x$.  If $\{x, z_1, \ldots, z_m\}$ is
    independent, then so is $\{x', z_1, \ldots, z_m\}$.  As the $y_i$
    are independent, it follows by Proposition~\ref{1-collapse} that
    $\{y_1, \ldots, y_n, z_1, \ldots, z_m\}$ is independent.
    Conversely, suppose that $\{y_1,\ldots,y_n,z_1,\ldots,z_m\}$ is
    independent, but $\{x,z_1,\ldots,z_m\}$ is not.  Choose $m$
    minimal such that this occurs.  Let $b = z_1 \vee \cdots \vee
    z_{m-1}$.  Then $(x \vee b) \wedge z_m > \bot$, but
    \begin{align*}
      x \wedge b &= \bot \\
      b \wedge z_m &= \bot \\
      x \wedge z_m &= \bot
    \end{align*}
    respectively: by choice of $m$; because $\vec{z}$ is independent;
    and because $z_m$ is independent from $y_1,\ldots,y_n$ hence not
    in $\cl\{y_1,\ldots,y_n\} = V(x)$.  By Lemma~\ref{3-fold-lemma}
    there is a quasi-atom $w \le x$ such that $\{w, b, z_m\}$ is not
    independent.  But $w \in V(x)$, so $w \in \cl\{y_1,\ldots,y_n\}$.
    The fact that $\{y_1,\ldots,y_n,z_1,\ldots,z_{m-1},z_m\}$ is
    independent then implies on pregeometry-theoretic grounds that
    $\{w,z_1,\ldots,z_{m-1},z_m\}$ is independent, which implies on
    lattice-theoretic grounds that $\{w, b, z_m\}$ is independent, a
    contradiction.
  \item Take $\{y_1, \ldots, y_n\}$ a basis for $V(x)$ and extend it
    to a basis $\{y_1,\ldots,y_n,z_1,\ldots,z_m\}$ for the entire
    pregeometry.  As in the proof of the previous point, we may
    replace $y_i$ with $y_i \wedge x$, and thus assume that $y_i \le
    x$.  By the previous point, independence of
    $\{y_1,\ldots,y_n,z_1,\ldots,z_m\}$ implies independence of
    $\{x,z_1,\ldots,z_m\}$. \qedhere
  \end{enumerate}
\end{proof}

\begin{remark} \label{bound-on-independence}
  Let $(P,\le)$ be a modular lattice with bottom element $\bot$ and
  reduced rank $n < \infty$.  The pregeometry on quasi-atoms has rank
  at most $n$.  If $a_1, \ldots, a_m$ is a basis, then the following
  facts hold:
  \begin{enumerate}
  \item Any independent sequence $\{b_1, \ldots, b_\ell\}$ of
    non-$\bot$ elements has size $\ell \le m$.
  \item If $\ell = m$ then each $b_i$ is a quasi-atom.
  \end{enumerate}
\end{remark}
\begin{proof}
  \begin{enumerate}
  \item For each $b_i$, we may find a quasi-atom $b'_i \le b_i$, by Proposition~\ref{quasi-atoms-basics}.\ref{enough-quasiats}.  The
    sequence $b'_1, \ldots, b'_\ell$ is pregeometry-independent, so
    $\ell \le m$.
  \item Suppose $\ell = m$.  If, say, $b_\ell$ is not a quasi-atom, we
    can find $x, y \in (\bot,b_\ell]$ such that $x \wedge y = \bot$.
      Replacing $b_\ell$ with $x \vee y \le b_\ell$, and then applying
      Proposition~\ref{1-collapse}, we see that $b_1, \ldots, b_{\ell
        - 1}, x, y$ is an independent sequence contradicting part 1. \qedhere
  \end{enumerate}
\end{proof}

\begin{definition}
  If $(P,\le)$ is a modular lattice and $a \ge b$ are elements, we
  define $\botrk(a/b)$ to be the supremum of $n$ such that there is a
  strict $n$-cube in $[b,a]$ with bottom $b$.
\end{definition}
Equivalently (by Proposition~\ref{cubes-and-independence}),
$\botrk(a/b)$ is the supremum over $n$ such that there exist $c_1,
\ldots, c_n \in (b,a]$ relatively independent over $b$.  The rank
  $\botrk(-)$ is not a subadditive rank in the sense of
  Definition~\ref{def-of-subadd-rank}.
\begin{remark} \label{botrks}
  If $a \ge b$, then
  \begin{enumerate}
  \item \label{botrk-vs-redrk} $\botrk(a/b) \le \redrk(a/b)$
  \item If $\redrk(a/b) < \infty$, there exists $b' \in [b,a]$ such
    that $\botrk(a/b') = \redrk(a/b')$.  Indeed, take a strict
    $n$-cube in $[b,a]$ for maximal $n$ and let $b'$ be the bottom of
    the cube.
  \item If $\redrk(a/b) < \infty$, then $\botrk(a/b)$ is the rank of
    the pregeometry of quasi-atoms in $[b,a]$, by
    Remark~\ref{bound-on-independence}.
  \end{enumerate}
\end{remark}
\begin{lemma} \label{botrk-superadd}
  For any $a, b$
  \begin{equation*}
    \botrk(a \vee b / a \wedge b) \ge \botrk(a/a \wedge b) + \botrk(b/a \wedge b).
  \end{equation*}
  Also, $\botrk(a \vee b/a) = \botrk(b/a \wedge b)$.
\end{lemma}
\begin{proof}
  If $c_1, \ldots, c_n$ is a strict independent sequence in $[a \wedge
    b, a]$ and $d_1, \ldots, d_m$ is a strict independent sequence in
  $[a \wedge b, b]$, then $c_1, \ldots, c_n, d_1, \ldots, d_m$ is a
  strict independent sequence in $[a \wedge b, a \vee b]$ by a couple
  applications of Proposition~\ref{1-collapse}.  The ``also'' clause
  is unrelated and follows immediately from the isomorphism of
  lattices between $[a, a \vee b]$ and $[a \wedge b, b]$.
\end{proof}

\section{Invariant valuation rings} \label{sec:valuations}
Let $(\Mm,+,\cdot,\ldots)$ be a monster-model finite dp-rank expansion
of a field.  Assume that $\Mm$ is not of finite Morley rank.  Fix a
small model $M_0$ large enough for Corollary~\ref{big-enough-model} to
apply.  Thus, for any type-definable $M_0$-linear subspace $J \le
\Mm^n$, we have $J = J^{00}$.

Let $\mathcal{P}_n$ be the poset of type-definable $M_0$-linear
subspaces of $\Mm^n$, let $\mathcal{P} = \mathcal{P}_1$, and let
$\mathcal{P}^+$ be the poset of non-zero elements of $\mathcal{P}$.

We collect the basic facts about these posets in the following
proposition:
\begin{proposition}
  \label{those-posets}
  ~
  \begin{enumerate}
  \item For each $n$, $\mathcal{P}_n$ is a bounded lattice.
  \item \label{p-nontriviality} For any small model $M \supseteq M_0$,
    the group $I_M$ is an element of $\mathcal{P}$.  In particular,
    $\mathcal{P}$ contains an element other than $\bot = 0$ and $\top
    = \Mm$.
  \item If $J \in \mathcal{P}$ is non-zero, every definable set $D$
    containing $J$ is heavy.
  \item \label{infinitesimals-below} If $J \in \mathcal{P}^+$ is
    $M$-definable for some small model $M \supseteq M_0$, then $J
    \supseteq I_M$.
  \item If $J \in \mathcal{P}_n$, then $J = J^{00}$.
  \item \label{nonzero-intersect} $\mathcal{P}^+$ is a sublattice of
    $\mathcal{P}$, i.e., it is closed under intersection.  Thus,
    $\mathcal{P}^+$ is a bounded-above lattice.
  \item \label{p-plus-vs-p} $\mathcal{P}$ has reduced rank $r$ for
    some $0 < r \le \dpr(\Mm)$.  The reduced rank of $\mathcal{P}^+$
    is also $r$, and the reduced rank of $\mathcal{P}^n$ is $rn$.
  \end{enumerate}
\end{proposition}
\begin{proof}
  \begin{enumerate}
  \item Clear---the lattice operations are given by
    \begin{align*}
      G \vee H &= G + H \\
      G \wedge H &= G \cap H \\
      \bot &= 0 \\
      \top &= \Mm^n.
    \end{align*}
  \item By Remark~\ref{basic-infs} and Theorem~\ref{inf-add}, $I_M$ is
    a non-zero $M_0$-linear subspace of $\Mm$, distinct from $0$ and
    $\Mm$.
  \item Replacing $J$ with $a \cdot J$ for some $a \in \Mm^\times$, we
    may assume $1 \in J$.  Then $M_0 = M_0 \cdot 1 \subseteq J$ by
    $M_0$-linearity.  Let $D$ be any definable set containing $J$.  By
    Lemma~\ref{something-similar-2} with $Z = \Mm$ and $W = D$, the
    set $D$ is heavy.
  \item Corollary~\ref{minimal-heavy-subgroup}.
  \item By choice of $M_0$.
  \item Let $J_1, J_2$ be two non-zero elements of $\mathcal{P}$.  Let
    $M$ be a small model containing $M_0$, over which both $J_1$ and
    $J_2$ are type-definable.  Then $J_1 \cap J_2 \ge I_M > \bot$.
    Therefore $\mathcal{P}^+$ is closed under intersection.
  \item Proposition~\ref{reduced-rank-vs-dp-rk} gives the bound $r \le
    n$.  Then
    \begin{equation*}
      0 < \redrk(\mathcal{P}^+) \le \redrk(\mathcal{P}) = r \le n,
    \end{equation*}
    where the left inequality is sharp because $\mathcal{P}$ has at
    least three elements by part \ref{p-nontriviality}.  If $r >
    \redrk(\mathcal{P}^+)$, there is a strict $r$-cube in
    $\mathcal{P}$ which does not lie in $\mathcal{P}^+$.  The bottom
    of this cube must be $\bot$, the only element of $\mathcal{P}
    \setminus \mathcal{P}^+$.  Then $r \le \botrk(\mathcal{P})$.
    However, part \ref{nonzero-intersect} says $\botrk(\mathcal{P})
    \le 1$, so $r \le 1 \le \redrk(\mathcal{P}^+)$, a contradiction.
    Therefore $\redrk(\mathcal{P}^+) = r = \redrk(\mathcal{P})$.
    Finally, in $\mathcal{P}_n$, if we let $J_i = 0^{\oplus (i - 1)}
    \oplus \Mm \oplus 0^{\oplus (n - i)}$ for $i = 1, \ldots, n$, then
    the sequence $J_1, \ldots, J_n$ is independent and $J_1 \vee
    \cdots \vee J_n = \Mm^n$.  Thus
    \begin{equation*}
      \redrk(\Mm^n/0) = \sum_{i = 1}^n \redrk(J_i/0)
    \end{equation*}
    by Remark~\ref{subadditive-rk-in-cubes}.  However, $\redrk(J_i/0)
    = r$ because of the isomorphism of lattices
    \begin{align*}
      \mathcal{P} &\to [0,J_i] \\
      X &\mapsto 0^{\oplus(i-1)} \oplus X \oplus 0^{\oplus(n-i)}. \qedhere
    \end{align*}
  \end{enumerate}
\end{proof}

In what follows, we will let $r$ be $\redrk(\Mm/0)$.
\begin{remark}
  If $r = 1$, then $\mathcal{P}$ is totally ordered and we can reuse
  the arguments for dp-minimal fields to immediately see that $I_M$ is
  a valuation ideal.  Usually we are not so lucky.
\end{remark}

\subsection{Special groups} \label{sec:special}
\begin{definition}
  An element $J \in \mathcal{P}_n$ is \emph{special} if
  $\botrk(\Mm^n/J) = \redrk(\Mm^n/J) = rn$.
\end{definition}
Equivalently, $J \in \mathcal{P}_n$ is special if $\botrk(\Mm^n/J) \ge
rn$.  This follows by Proposition~\ref{those-posets}.\ref{p-plus-vs-p} and Remark~\ref{botrks}.\ref{botrk-vs-redrk}.
\begin{proposition}
  \label{special-proposition}
  ~
  \begin{enumerate}
  \item There is at least one non-zero special $J \in \mathcal{P} =
    \mathcal{P}_1$.
  \item \label{guards} Let $J \in \mathcal{P}$ be special.  Let $A_1,
    \ldots, A_r$ be a basis of quasi-atoms over $J$.  Let $G \in
    \mathcal{P}$ be arbitrary.  If $G \cap A_i \not\subseteq J$ for
    each $i$, then $G \supseteq J$.
  \item \label{guard-application} If $J \in \mathcal{P}$ is special
    and nonzero and type-definable over a small model $M \supseteq
    M_0$, then
    \begin{equation*}
      I_M \cdot J \subseteq I_M \subseteq J
    \end{equation*}
  \item \label{oplus} If $I \in \mathcal{P}_n$ and $J \in
    \mathcal{P}_m$ are special, then $I \oplus J \in
    \mathcal{P}_{n+m}$ is special.
  \item \label{special-scale} If $I \in \mathcal{P}_n$ is special and
    $\alpha \in \Mm^\times$, then $\alpha \cdot I$ is special.
  \end{enumerate}
\end{proposition}
\begin{proof}
  \begin{enumerate}
  \item By Proposition~\ref{those-posets}.\ref{p-plus-vs-p} the
    reduced rank of $\mathcal{P}^+$ is exactly $r$, so we can find a
    strict $r$-cube in $\mathcal{P}^+$.  The base of such a cube is a
    non-zero special element of $\mathcal{P}$.
  \item For $i = 1, \ldots, r$ take $a_i \in (G \cap A_i) \setminus
    J$.  Then for each $i$, we have
    \begin{align*}
      & a_i \in G \cap A_i \subseteq (G + J) \cap A_i \supseteq J \cap A_i = J \\
      & a_i \notin J \\
      (\implies) & (G + J) \cap A_i \supsetneq J.
    \end{align*}
    Consequently,
    \begin{equation*}
      A_1 \cap (G + J), A_2 \cap (G + J), \ldots, A_r \cap (G + J)
    \end{equation*}
    is a strict independent sequence over $J$.  It follows that
    \begin{equation*}
      \redrk(G/(G \cap J)) = \redrk((G+J)/J) \ge \botrk((G+J)/J) \ge r.
    \end{equation*}
    On the other hand
    \begin{equation*}
      \redrk(G/(G \cap J)) + \redrk(J/(G \cap J)) = \redrk((G+J)/(G \cap J)) \le r,
    \end{equation*}
    and so $\redrk(J/(G \cap J)) = 0$.  This forces $J = G \cap J$, so
    $J \subseteq G$.
  \item The inclusion $I_M \subseteq J$ is
    Proposition~\ref{those-posets}.\ref{infinitesimals-below}. Let
    $A_1, \ldots, A_r$ be a basis of quasi-atoms over $J$ as in part
    \ref{guards}.  For each $A_i$ choose an element $a_i \in
    A_i\setminus J$.  Let $M'$ be a small model containing $M$ and the
    $a_i$'s.  We first claim that $I_{M'} \cdot J \subseteq I_{M'}$.
    Let $\varepsilon$ be a non-zero element of $I_{M'}$.  As $I_{M'}$ is
    closed under multiplication by $(M')^\times$, we have $M'
    \subseteq \varepsilon^{-1} \cdot I_{M'}$.  In particular, $a_i \in
    \varepsilon^{-1} \cdot I_{M'}$ for each $i$.  Then
    \begin{equation*}
      (\varepsilon^{-1} \cdot I_{M'}) \cap A_i \not \subseteq J
    \end{equation*}
    for each $i$, so by part \ref{guards} we have
    \begin{equation*}
      \varepsilon^{-1} \cdot I_{M'} \supseteq J.
    \end{equation*}
    In other words, $\varepsilon \cdot J \subseteq I_{M'}$.  As
    $\varepsilon$ was an arbitrary non-zero element of $I_{M'}$, it
    follows that $I_{M'} \cdot J \subseteq I_{M'}$.  Now suppose that
    $D$ is an $M$-definable basic neighborhood.  Then $D$ is an
    $M'$-definable basic neighborhood.  By the above and compactness,
    there is an $M'$-definable basic neighborhood $X \mininf X$ and a
    definable set $D_2 \supseteq J$ such that $(X \mininf X) \cdot D_2
    \subseteq D$.  Furthermore, $D_2$ can be taken to be
    $M$-definable, because $J$ is a directed intersection of
    $M$-definable sets.  Having done this, we can then pull the
    parameters defining $X$ into $M$, and assume that $X$ is
    $M$-definable.  (This uses the fact that heaviness is definable in
    families).  Then we have an $M$-definable basic neighborhood $X
    \mininf X$ and an $M$-definable set $D_2 \supseteq J$ such that
    $(X \mininf X) \cdot D_2 \subseteq D$.  As $D$ was arbitrary, it
    follows that
    \begin{equation*}
      I_M \cdot J \subseteq I_M.
    \end{equation*}
  \item The interval $[I,\Mm^n]$ in $\mathcal{P}_n$ is isomorphic to
    $[I \oplus J, \Mm^n \oplus J]$ in $\mathcal{P}_{n+m}$, so
    \begin{align*}
      rn = \botrk(\Mm^n/I) &= \botrk((\Mm^n \oplus J)/(I \oplus J)) \\
      rm = \botrk(\Mm^m/J) &= \botrk((I \oplus \Mm^m)/(I \oplus J)),
    \end{align*}
    where the second line is true for similar reasons.  By
    Lemma~\ref{botrk-superadd},
    \begin{equation*}
      \botrk((\Mm^n \oplus \Mm^m)/(I \oplus J)) \ge rn + rm.
    \end{equation*}
    On the other hand
    \begin{equation*}
      \botrk((\Mm^n \oplus \Mm^m)/(I \oplus J)) \le \redrk((\Mm^n
      \oplus \Mm^m)/(I \oplus J)) \le r(n+m)
    \end{equation*}
    so equality holds and $I \oplus J$ is special.
  \item For any $\alpha \in \Mm^\times$, the map $X \mapsto \alpha
    \cdot X$ is an automorphism of $\mathcal{P}_n$. \qedhere
  \end{enumerate}
\end{proof}
\begin{corollary} \label{ring-topology}
  For any model $M$, $I_M \cdot I_M \subseteq I_M$.  The topology in
  Remark~\ref{exists-topology} is a ring topology, not just a group
  topology.
\end{corollary}
\begin{proof}
  First suppose $M \supseteq M_0$.  Let $J$ be a non-zero special
  element of $\mathcal{P}_1$.  Then $I_M \subseteq J$, and so
  \begin{equation*}
    I_M \cdot I_M \subseteq I_M \cdot J \subseteq I_M
  \end{equation*}
  by Proposition~\ref{special-proposition}.\ref{guard-application}.
  Then we can shrink $M$ using the technique of the proof of
  Proposition~\ref{special-proposition}.\ref{guard-application}.
\end{proof}
With much more work, one can show that the canonical topology is a
field topology.  See Corollary~5.15 in \cite{prdf2}.
\begin{speculation}
  Say that $J \in \mathcal{P}_1$ is \emph{bounded} if $J \le J'$ for
  some special $J'$.  Based on the argument in
  Proposition~\ref{special-proposition}.\ref{guards}-\ref{guard-application},
  it seems that $J$ is bounded if and only if $\alpha \cdot J
  \subseteq I_M$ for some $\alpha \in \Mm^\times$ and some small model
  $M$.  Bounded elements should form a sublattice of
  $\mathcal{P}_1$.\footnote{These ideas have been developed in \S8 of
    \cite{prdf2}.}
\end{speculation}

\begin{definition}
  Let $I \in \mathcal{P}_n$ be special and $D \in \mathcal{P}_n$ be
  arbitrary.  Then $D$ \emph{dominates} $I$ if $D \ge I$ and
  $\botrk(D/I) = nr$.
\end{definition}
\begin{lemma}
  \label{characterization-of-domination}
  Let $J \in \mathcal{P}_n$ be special, let $A_1, \ldots, A_{nr}$ be a
  basis of quasi-atoms in $[J,\Mm^n]$, and $D$ be arbitrary.  Then $D$
  dominates $J$ if and only if $D \cap A_i \supsetneq J$ for each $i$.
  In particular, this condition doesn't depend on the choice of the
  basis $\{A_1,\ldots,A_{nr}\}$.
\end{lemma}
\begin{proof}
  Suppose $D$ dominates $J$.  Let $B_1, \ldots, B_{nr}$ be a basis of
  quasi-atoms in $[J,D]$.  The $B_i$ are independent quasi-atoms in
  the larger interval $[J,\Mm^n]$, so $\{B_1,\ldots,B_{nr}\}$ is
  another basis of quasi-atoms in $[J,\Mm^n]$.  Therefore, for every
  $i$ the sequence $\{B_1,\ldots,B_{nr},A_i\}$ is \emph{not}
  independent over $J$.  Consequently
  \begin{equation*}
    D \cap A_i \supseteq (B_1 + \cdots + B_{nr}) \cap A_i \supsetneq
    J.
  \end{equation*}
  Conversely, suppose $D \cap A_i \supsetneq J$ for each $i$.  Then
  certainly $D \supseteq J$, and it remains to show $\botrk(D/J) \ge
  nr$.  Let $A'_i := D \cap A_i$.  Then each $A'_i$ is a quasi-atom
  over $J$, equivalent to $A_i$, and so $\{A'_1,\ldots,A'_{nr}\}$ is
  another basis of quasi-atoms over $J$.  As each $A'_i$ lies in
  $[J,D]$, it follows that $\botrk(D/J) \ge nr$.
\end{proof}

\begin{lemma}
  \label{special-lemma-1}
  Let $I \in \mathcal{P}_n$ be special, and $V$ be a $k$-dimensional
  $\Mm$-linear subspace of $\Mm^n$.  Then
  \begin{align*}
    \botrk((V + I)/I) &= \redrk((V + I)/I) = kr \\
    \botrk(V/(V \cap I)) &= \redrk(V/(V \cap I)) = kr \\
    \botrk(\Mm^n/(V + I)) &= \redrk(\Mm^n/(V + I)) = (n-k)r.
  \end{align*}
  Moreover, there exist $A_1, \ldots, A_{kr}, B_1, \ldots, B_{(n-k)r}
  \in \mathcal{P}_n$ such that the following conditions hold:
  \begin{enumerate} 
  \item The set $\{A_1, \ldots, A_{kr}, B_1, \ldots, B_{(n-k)r}\}$ is
    a basis of quasi-atoms in $[I, \Mm^n]$.
  \item Let $\tilde{A}_i = A_i \cap V$.  Then $\{\tilde{A}_1, \ldots,
    \tilde{A}_{kr}\}$ is a basis of quasi-atoms in $[V \cap I, V]$.
  \item Let $\tilde{B}_i = B_i + V$.  Then $\{\tilde{B}_1, \ldots,
    \tilde{B}_{(n-k)r}\}$ is a basis of quasi-atoms in $[V + I,
      \Mm^n]$.
  \end{enumerate}
  Given a $D \in \mathcal{P}_n$ dominating $I$, we may choose the
  $A_i$ and $B_i$ to lie in $[I,D]$.
\end{lemma}
\begin{proof}
  Let $W$ be a complementary $(n-k)$-dimensional $\Mm$-linear
  subspace, so that $V + W = \Mm^n$.  Let $V' = V + I$ and $W' = W+I$.
  Then
  \begin{align*}
    nr = \redrk(\Mm^n/I) & \le \redrk(\Mm^n/V') + \redrk(V'/I) \\
    & = \redrk((V'+W')/V') + \redrk(V'/I) \\
    & = \redrk(W'/(W' \cap V')) + \redrk(V'/I) \\
    & \le \redrk(W'/I) + \redrk(V'/I) \\
    & = \redrk(W/(W \cap I)) + \redrk(V/(V \cap I)) \\
    & \le \redrk(W/0) + \redrk(V/0).
  \end{align*}
  Now any $\Mm$-linear isomorphism $\phi : \Mm^k \stackrel{\sim}{\to}
  V$ induces an isomorphism of posets from $\mathcal{P}_k$ to $[0,V]
  \subseteq \mathcal{P}_n$, so
  \begin{align*}
    \redrk(V/0) &= \redrk(\mathcal{P}_k) = kr \\
    \redrk(W/0) &= \redrk(\mathcal{P}_{n-k}) = (n-k)r,
  \end{align*}
  where the second line follows similarly.  Therefore the inequalities
  above are all equalities, and
  \begin{align*}
    \redrk((V+I)/I) &= \redrk(V/(V \cap I)) = kr \\
    \redrk(\Mm^n/(V+I)) &= \redrk(\Mm^n/V') = (n-k)r.
  \end{align*}
  By Proposition~\ref{covering-lemma-home}.\ref{covering-lemma}, there
  is a basis $\{A_1, \ldots, A_m, B_1, \ldots, B_{nr-m}\}$ of
  independent quasi-atoms in $[I, \Mm^n]$ such that
  \begin{itemize}
  \item Each $A_i \subseteq V+I$.
  \item The sequence $(V+I), B_1, \ldots, B_{nr-m}$ is independent
    over $I$.
  \end{itemize}
  Given a bound $D$ dominating $I$, we may replace each $A_i$ with
  $A_i \cap D$ and $B_i$ with $B_i \cap D$, and assume henceforth that
  $A_i, B_i \subseteq D$.  By Remark~\ref{subadditive-rk-in-cubes},
  \begin{align*}
    nr = \redrk(\Mm^n/I) & \ge \redrk((V+I)/I) + \redrk(B_1/I) + \cdots
    + \redrk(B_{nr-m}/I) \\ & = kr + \redrk(B_1/I) + \cdots + \redrk(B_{nr-m}/I).
  \end{align*}
  Each $B_i$ is strictly greater than $I$, so
  \begin{equation*}
    nr \ge kr + nr - m,
  \end{equation*}
  and thus $m \ge kr$.  On the other hand, the set $\{A_1, \ldots,
  A_m\}$ is a set of independent quasi-atoms in $[I, V + I]$, so
  \begin{equation*}
    m \le \botrk((V+I)/I) \le \redrk((V+I)/I) = kr.
  \end{equation*}
  Thus equality holds, $m = kr$, and the set $\{A_1, \ldots, A_m\}$ is
  a basis of quasi-atoms in $[I, V + I]$.  Applying the isomorphism
  \begin{align*}
    [I, V + I] & \stackrel{\sim}{\to} [V \cap I, V] \\
    X & \mapsto X \cap V
  \end{align*}
  shows that the $\tilde{A}_i$ form a basis of quasi-atoms in $[V \cap
    I, V]$.  Next, let $Q = B_1 \vee \cdots \vee B_{(n-k)r}$.  (Note
  that $nr - m = (n-k)r$.)  The fact that $(V+I), B_1, \ldots,
  B_{(n-k)r}$ is independent over $I$ implies that $(V+I) \cap Q = I$.
  Therefore, there is an isomorphism
  \begin{align*}
    [I,Q] &\stackrel{\sim}{\to} [V+I,V+I+Q] \\
    X & \mapsto X + (V + I) = X + V.
  \end{align*}
  The elements $\{B_1, \ldots, B_{(n-k)r}\}$ are independent
  quasi-atoms in $[I,Q]$, and therefore the $\tilde{B}_i$ are a set of
  independent quasi-atoms in $[V+I,V+I+Q]$ or even in $[V+I,\Mm^n]$.
  It follows that
  \begin{equation*}
    (n-k)r \le \botrk(\Mm^n/(V+I)) \le \redrk(\Mm^n/(V+I)) = (n-k)r,
  \end{equation*}
  so equality holds and the $\tilde{B}_i$ are a basis of quasi-atoms
  in $[V+I,\Mm^n]$.
\end{proof}

\begin{speculation} \label{r-fold-specs}
  Fix a special element $I$ of $\mathcal{P}_1$.  For every $n$, $I^n$
  is a special element of $\mathcal{P}_n$.  Let $\mathcal{G}_n$ be the
  modular geometry associated to the pregeometry of quasi-atoms over
  $I^n$ in $\mathcal{P}_n$.  By utilizing the embeddings between the
  $\mathcal{G}_n$, one should be able to prove the following: there is
  an object $T$ in a semi-simple abelian category $\mathcal{C}$ such
  that $T$ has length $r$, and for every $n$ the geometry
  $\mathcal{G}_n$ is the geometry of atoms in the lattice of
  subobjects of $T^n$.  The category $\mathcal{C}$ should be enriched
  over $M_0$-vector spaces.  For every $n$ there is a map $f_n$ from
  the lattice of $\Mm$-linear subspaces of $\Mm^n$ to the lattice of
  subobjects of $T^n$, defined in such a way that $f_n(V)$ corresponds
  to the closed set cut out by $V + I^n$ in the pregeometry of
  quasi-atoms in $[I^n,\Mm^n]$.  These maps $f_n$ should have the
  following properties:
  \begin{itemize}
  \item $V \le W$ should imply $f_n(V) \le f_n(W)$.
  \item The length of $f_n(V)$ should be $r$ times the dimension of
    $V$, essentially by Lemma~\ref{special-lemma-1}.  In particular,
    $f_1(0) = 0$ and $f_1(\Mm) = T$.
  \item For each $n$, the map $f_n$ should be $GL_n(M_0)$-equivariant.
  \item $f_n(V) \oplus f_m(W)$ should equal $f_{n+m}(V \oplus W)$.
  \end{itemize}
  When $r = 1$, this should secretly amount to a valuation on $\Mm$.
  Indeed, $T$ is then a simple object so we may assume that
  $\mathcal{C}$ is the category of finite-dimensional $D$-modules for
  some skew field $D$ extending $M_0$, and $T$ is just $D^1$.  Then
  the above configuration should necessarily come from a map $\Mm \to
  k \to D$ where $\Mm \to k$ is a valuation/place, and $k \to D$ is an
  inclusion.  The hope was that similar configurations for $r > 1$
  should come from multi-valued fields, but several things go wrong.
  \begin{enumerate}
  \item The idea is to recover a valuation on $\Mm$ where the residue
    of $a$ is $\alpha \in End(T)$ if
    \begin{equation*}
      f_2(\{(x,ax) ~|~ x \in \Mm\}) = \{(x, \alpha x) ~|~ x \in T\}.
    \end{equation*}
    This line of thinking is what underlies the ring $R_J$ and ideal
    $I_J$ in Proposition~\ref{rings-and-ideals} below--$R_J$ is the
    set of $a$ such that the left hand side is the graph of any
    endomorphism of $T$ and $I_J$ is the set of $a$ such that the left
    hand side is the graph of $0$.
  \item Ideally, we would get $R_J$ equal to a Bezout domain with
    finitely many maximal ideals.  This amounts to the requirement
    that for any $\alpha \in \Mm$ and any distinct $q_0, \ldots, q_n
    \in M_0$, at least one of $1/(\alpha - q_i)$ lies in $R_J$.
  \item However, this fails when the system of maps looks like
    \begin{equation*}
      \left(V \le (\Qq(\sqrt{2}))^n\right) \mapsto (V + V^\dag, V \cap V^\dag),
    \end{equation*}
    where the right hand side should be thought of as an element of
    $FinVec(\Qq) \times FinVec(\Qq)$, and $V^\dag$ is the image of $V$
    under the non-trivial element of $\Aut(\Qq(\sqrt{2})/\Qq)$.
  \item There is a certain way to mutate $f$ that improves the
    situation, namely
    \begin{equation*}
      f'_n(V) := f_{2n}(\{(\vec{x},a \cdot \vec{x}) ~|~ \vec{x} \in V\})
    \end{equation*}
    for well-chosen $a$, e.g., $a = \sqrt{2}$.  This mutation is what
    underlies Lemma~\ref{twists} below.
  \item Originally, I had hoped that finitely many mutations would
    reduce to the ``good'' Bezout case, but there are genuine
    type-definable $J$ in dp-finite fields for which this fails.  For
    example, if $\Gamma$ is the group $\Qq(1/3)$ and $K$ is
    $\Cc((\Gamma))$, we can make $K$ have dp-rank 2 by adding a
    predicate for the subfield $\Cc((2 \cdot \Gamma))$.  Any Hahn
    series $x \in K$ can be decomposed as $x_{odd} + x_{even}$ where
    $x_{odd}$ has the terms with odd exponents and $x_{even}$ has the
    terms with even exponents.  The decomposition is definable.
    Taking $J$ to be the set of $x$ such that $\val(x_{even}) > 0$ and
    $\val(x_{odd}) \ge 1$ (or something similar), one arrives at a
    situation where $R_J$ and $I_J$ cannot be fixed via any finite
    sequence of mutations.
  \item Nevertheless, one can still recover a good Bezout domain in
    the limit, which is Theorem~\ref{bezout-theorem}.  Unfortunately,
    the resulting Bezout domain is not closely tied to the
    infinitesimals---not enough to run the henselianity machine from
    the dp-minimal case.
  \end{enumerate}
  For details on the above construction, see \cite{prdf3}.
\end{speculation}
We now return to actual proofs.
\begin{lemma}
  \label{special-lemma-2}
  Let $I, J \in \mathcal{P}_n$ be special.  Then $I + J$ and $I \cap
  J$ are special.  Furthermore, there exists
  \begin{itemize}
  \item a basis of quasi-atoms $\hat{A}_1, \ldots, \hat{A}_n$ in $[I
    \cap J, \Mm^n]$,
  \item a basis of quasi-atoms $\hat{B}_1, \ldots, \hat{B}_n$ in $[I
    + J, \Mm^n]$, and
  \item a basis of quasi-atoms $A_1, \ldots, A_n, B_1, \ldots, B_n$
    in $[I \oplus J, \Mm^{2n}]$
  \end{itemize}
  related as follows:
  \begin{align*}
    \hat{A}_i & = \{ \vec{x} \in \Mm^n ~|~ (\vec{x},\vec{x}) \in A_i\}
    \\
    \hat{B}_i & = \{ \vec{x} - \vec{y} ~|~ (\vec{x},\vec{y}) \in B_i\}.
  \end{align*}
  Given $D \in \mathcal{P}_{2n}$ dominating $I \oplus J$, we may
  choose the $A_i$ and $B_i$ to lie in $[I \oplus J, D]$.
\end{lemma}
\begin{proof}
  For any $J \in \mathcal{P}_n$, define
  \begin{align*}
    \Delta(C) &= \{(\vec{x},\vec{x}) ~|~ \vec{x} \in C\} \in
    \mathcal{P}_{2n} \\
    \nabla(C) &= \{(\vec{x}, \vec{x} + \vec{y}) ~|~ \vec{x} \in \Mm^n, ~
    \vec{y} \in C\} \in \mathcal{P}_{2n}.
  \end{align*}
  Let $V = \Delta(\Mm^n) = \nabla(0)$.  The maps $\Delta(-),
  \nabla(-)$ yield isomorphisms
  \begin{align*}
    \Delta : \mathcal{P}_n &\stackrel{\sim}{\to} [0,V] \subseteq
    \mathcal{P}_{2n} \\
    \nabla : \mathcal{P}_n &\stackrel{\sim}{\to} [V,\Mm^{2n}]
    \subseteq \mathcal{P}_{2n}.
  \end{align*}
  Indeed, the inverses are given by
  \begin{align*}
    \Delta^{-1} : [0, V] & \stackrel{\sim}{\to} [0, \Mm^n] \\
    C & \mapsto \{\vec{x} ~|~ (\vec{x},\vec{x}) \in C\} \\
    \nabla^{-1} : [V, \Mm^{2n}] & \stackrel{\sim}{\to} [0, \Mm^n] \\
    C & \mapsto \{\vec{x} - \vec{y} ~|~ (\vec{x},\vec{y}) \in C\}.
  \end{align*}
  Note that $\Delta^{-1}(V \cap (I \oplus J)) = I \cap J$ and
  $\nabla^{-1}(V + (I \oplus J)) = I + J$.  Therefore, $\Delta^{-1}$
  and $\nabla^{-1}$ restrict to isomorphisms
  \begin{align*}
    \Delta^{-1} : [V \cap (I \oplus J), V] & \stackrel{\sim}{\to} [I
      \cap J, \Mm^n] \\
    \nabla^{-1} : [V + (I \oplus J), \Mm^{2n}] & \stackrel{\sim}{\to}
          [I + J, \Mm^n].
  \end{align*}
  It follows that
  \begin{align*}
    \botrk(\Mm^n/(I \cap J)) &= \botrk(V/(V \cap (I \oplus J))) =
    \botrk((V + (I \oplus J))/V) \\
    \botrk(\Mm^n/(I + J)) &= \botrk(\Mm^{2n}/(V + (I \oplus J))/V).
  \end{align*}
  Now $I \oplus J$ is special in $\mathcal{P}_{2n}$ by
  Proposition~\ref{special-proposition}.\ref{oplus}, and $V$ is an
  $n$-dimensional $\Mm$-linear subspace of $\Mm^{2n}$, so by
  Lemma~\ref{special-lemma-1},
  \begin{align*}
    \botrk(\Mm^n/(I \cap J)) &= \botrk((V + (I \oplus J))/V) = rn
    \\
    \botrk(\Mm^n/(I + J)) &= \botrk(\Mm^{2n}/(V + (I \oplus J))) = rn.
  \end{align*}
  Therefore $I \cap J$ and $I + J$ are special.  Furthermore, by
  Lemma~\ref{special-lemma-1} there exists a basis
  $\{A_1,\ldots,A_{rn},B_1,\ldots,B_{rn}\}$ of quasi-atoms over $I
  \oplus J$ such that
  \begin{itemize}
  \item The elements $\tilde{A}_i := A_i \cap V$ form a basis of
    quasi-atoms in $[V \cap (I \oplus J), V]$.
  \item The elements $\tilde{B}_i := B_i + V$ form a basis of
    quasi-atoms in $[V + (I \oplus J), \Mm^{2n}]$.
  \end{itemize}
  (Additionally, the $A_i$ and $B_i$ can be chosen below any given $D$
  dominating $I \oplus J$.)  Applying $\Delta^{-1}$ and $\nabla^{-1}$
  we see that the elements
  \begin{align*}
    \hat{A}_i = \Delta^{-1}(A_i \cap V) & = \{\vec{x} ~|~
    (\vec{x},\vec{x}) \in A_i \cap V\} \\
    & = \{\vec{x} ~|~ (\vec{x},\vec{x}) \in A_i\} \\
    \hat{B}_i = \nabla^{-1}(A_i + V) & = \{\vec{x} - \vec{y} ~|~
    (\vec{x},\vec{y}) \in B_i + V\} \\
    & = \{\vec{x} - \vec{y} ~|~ (\vec{x},\vec{y}) \in B_i\}
  \end{align*}
  form bases of quasi-atoms for $[I \cap J, \Mm^n]$ and $[I + J,
    \Mm^n]$, respectively.
\end{proof}
\begin{question}
  By Lemma~\ref{special-lemma-2} special elements of $\mathcal{P}_n$
  form a sublattice.  Can this be proven directly (lattice
  theoretically) within $\mathcal{P}_n$ without using the larger
  lattice $\mathcal{P}_{2n}$?
\end{question}

\subsection{The associated rings and ideals}
\begin{definition}
  Let $J \in \mathcal{P}_n$ be special, and $a \in \Mm^\times$.  Say
  that $a$ \emph{contracts} $J$ if $a = 0$ or $J$ dominates $a \cdot
  J$ (i.e., $\botrk(J/a \cdot J) = nr$).
\end{definition}
Note that when $a \ne 0$, $J$ dominates $a \cdot J$ if and only if
$a^{-1} \cdot J$ dominates $J$.
\begin{lemma}
  \label{contraction}
  ~
  \begin{enumerate}
  \item \label{c-1} Let $A_1, \ldots, A_{nr}$ be a basis of
    quasi-atoms over $J \in \mathcal{P}_n$ and $a$ be an element of
    $\Mm$.  If $a \cdot A_i \subseteq J$ for all $i$, then $a$
    contracts $J$.  Conversely, suppose $a$ contracts $J$.  Then there
    exists $A'_i \in (J,A_i]$ such that $A'_1, \ldots, A'_{nr}$ is a
      basis of quasi-atoms over $J$ and $a \cdot A'_i \subseteq J$ for
      each $i$.
  \item \label{c-2} If $a$ contracts $J$ and $b \in \Mm^\times$, then
    $a$ contracts $b \cdot J$.
  \item \label{c-3} If $a, b$ contract $J$ then $a + b$ contracts $J$.
  \item \label{c-4} If $a$ contracts $J$ and $b \cdot J \subseteq J$,
    then $a \cdot b$ contracts $J$.
  \item If $a$ contracts both $I \in \mathcal{P}_n$ and $J \in
    \mathcal{P}_m$, then $a$ contracts $I \oplus J \in
    \mathcal{P}_{n+m}$.
  \item \label{c-6} If $a$ contracts $I, J \in \mathcal{P}_n$, then
    $a$ contracts $I \cap J$ and $I + J$.
  \end{enumerate}
\end{lemma}
\begin{proof}
  \begin{enumerate}
  \item First suppose $a \cdot A_i \subseteq J$.  If $a = 0$ then $a$
    contracts $J$ by definition, so suppose $a \ne 0$.  Then
    \begin{equation*}
      (a^{-1} \cdot J) \cap A_i = A_i \supsetneq J
    \end{equation*}
    for each $i$, so by Lemma~\ref{characterization-of-domination} the
    group $a^{-1} \cdot J$ dominates $J$, or equivalently, $J$
    dominates $a \cdot J$.  Thus $a$ contracts $J$.  Conversely,
    suppose that $a$ contracts $J$.  If $a = 0$ then $a \cdot A_i
    \subseteq J$ so we may take $A'_i = A_i$.  Otherwise, note that
    $a^{-1} \cdot J$ dominates $J$, so by
    Lemma~\ref{characterization-of-domination},
    \begin{equation*}
      A'_i := (a^{-1} \cdot J) \cap A_i \supsetneq J.
    \end{equation*}
    Then $A'_i$ is a quasi-atom over $J$, equivalent to $A_i$, so
    $\{A'_1, \ldots, A'_{nr}\}$ is another basis of quasi-atoms over
    $J$.  Furthermore $A'_i \subseteq a^{-1} \cdot J$, so $a \cdot
    A'_i \subseteq J$.
  \item Multiplication by $b$ induces an automorphism of
    $\mathcal{P}_n$ sending the interval $[a \cdot J, J]$ to $[a \cdot
    (b \cdot J), b \cdot J]$, so $\botrk(J/a \cdot J) = \botrk(b \cdot
    J / (ab) \cdot J)$.
  \item Take a basis of quasi-atoms $A_1, \ldots, A_{rn}$ over $J$.
    By part 1, we may shrink the $A_i$ and assume that $a \cdot A_i
    \subseteq J$.  Shrinking again, we may assume $b \cdot A_i
    \subseteq J$.  Then
    \begin{equation*}
      (a + b) \cdot A_i \subseteq a \cdot A_i + b \cdot A_i \subseteq J + J = J
    \end{equation*}
    so by part 1, $a+b$ contracts $J$.
  \item Suppose $a$ contracts $J$ and $b \cdot J \subseteq J$.  Then
    \begin{equation*}
      \botrk(J/a \cdot b \cdot J) \ge \botrk(b \cdot J/ a \cdot b
      \cdot J) = \botrk(J/a \cdot J) = nr.
    \end{equation*}
  \item Let $A_1, \ldots, A_{rn}$ be a basis of quasi-atoms in
    $[I,\Mm^n]$, and $B_1, \ldots, B_{rm}$ be a basis of quasi-atoms
    in $[J,\Mm^m]$.  Shrinking the $A_i$ and $B_i$, we may assume $a
    \cdot A_i \subseteq I$ and $a \cdot B_i \subseteq J$.  Note that
    the sequence
    \begin{equation*}
      A_1 \oplus J, A_2 \oplus J, \ldots, A_{rn} \oplus J, I \oplus
      B_1, I \oplus B_2, \ldots, I \oplus B_{rm}
    \end{equation*}
    is a basis of quasi-atoms in $[I \oplus J, \Mm^{n+m}]$.
    Multiplication by $a$ collapses each of these quasi-atoms into $I
    \oplus J$ (using the fact that $a \cdot I \subseteq I$ and $a
    \cdot J \subseteq J$).  Therefore $a$ contracts $I \oplus J$.
  \item We may assume $a \ne 0$.  By the previous point, $a^{-1} \cdot
    (I \oplus J)$ dominates $I \oplus J$.  By
    Lemma~\ref{special-lemma-2}, $I + J$ and $I \cap J$ are special.
    Moreover, there is a basis of quasi-atoms $A_1, \ldots, A_n, B_1,
    \ldots, B_n$ in $[I \oplus J, \Mm^{2n}]$ such that for
    \begin{align*}
      \hat{A}_i &= \{ \vec{x} \in \Mm^n ~|~ (\vec{x},\vec{x}) \in A_i\} \\
      \hat{B}_i &= \{ \vec{x} - \vec{y} ~|~ (\vec{x},\vec{y}) \in B_i\}
    \end{align*}
    the set $\{\hat{A}_1, \ldots, \hat{A}_n\}$ is a basis of
    quasi-atoms over $I \cap J$ and the set $\{\hat{B}_1, \ldots,
    \hat{B}_n\}$ is a basis of quasi-atoms over $I + J$.  Furthermore
    Lemma~\ref{special-lemma-2} ensures that the $A_i$ and $B_i$ can
    be chosen in $[I \oplus J, a^{-1} \cdot (I \oplus J)]$.  Thus $a
    \cdot A_i \subseteq I \oplus J$ and $a \cdot B_i \subseteq I
    \oplus J$.  Then
    \begin{equation*}
      \vec{x} \in \hat{A}_i \iff (\vec{x},\vec{x}) \in A_i \implies (a
      \cdot \vec{x}, a \cdot \vec{x}) \in I \oplus J \iff a \cdot
      \vec{x} \in I \cap J,
    \end{equation*}
    so $a \cdot \hat{A}_i \subseteq I \cap J$.  As the $\hat{A}_i$
    form a basis of quasi-atoms over $I \cap J$, it follows that $a$
    contracts $I \cap J$.  Similarly,
    \begin{equation*}
      (\vec{x},\vec{y}) \in B_i \implies (a \cdot \vec{x}, a \cdot
      \vec{y}) \in I \oplus J \implies a \cdot (\vec{x} - \vec{y}) \in
      I + J
    \end{equation*}
    so $a \cdot \hat{B}_i \subseteq I + J$.  Thus $a$ contracts $I +
    J$. \qedhere
  \end{enumerate}
\end{proof}
\begin{proposition}
  \label{rings-and-ideals}
  For any special $J \in \mathcal{P} = \mathcal{P}_1$, let $R_J$ be
  the set of $a \in \Mm$ such that $a \cdot J \subseteq J$, and let
  $I_J$ be the set of $a \in \Mm$ that contract $J$.
  \begin{enumerate}
  \item $R_J$ is a subring of $\Mm$, containing $M_0$.
  \item $I_J$ is an ideal in $R_J$.
  \item \label{r-scale} If $b \in \Mm^\times$ then $R_J = R_{b \cdot
    J}$ and $I_J = I_{b \cdot J}$.
  \item If $J$ is type-definable over $M \supseteq M_0$, then $R_J$
    and $I_J$ are $M$-invariant.
  \item \label{ij-im} If $J$ is non-zero and type-definable over $M
    \supseteq M_0$ then $I_M \subseteq I_J$.
  \item \label{cap-not-vee} If $J_1$ and $J_2$ are special, then
    \begin{align*}
      R_{J_1} \cap R_{J_2} & \subseteq R_{J_1 \cap J_2} \\
      I_{J_1} \cap I_{J_2} & \subseteq I_{J_1 \cap J_2}
    \end{align*}
  \item \label{jacobson} $(1 + I_J) \subseteq R_J^\times$.
    Consequently, $I_J$ lies inside the Jacobson radical of $R_J$.
  \end{enumerate}
\end{proposition}
\begin{proof}
  \begin{enumerate}
  \item Straightforward.
  \item The set $I_J$ is a subset of $R_J$.  The fact that $I_J \lhd
    R_J$ is exactly Lemma~\ref{contraction}.\ref{c-3}-\ref{c-4}.
  \item For $I_J$ this is Lemma~\ref{contraction}.\ref{c-2}.  For
    $R_J$ this is clear:
    \begin{equation*}
      a \cdot J \subseteq J \implies (ab) \cdot J \subseteq b \cdot J.
    \end{equation*}
  \item The definitions are $\Aut(\Mm/M)$-invariant.
  \item Let $A_1, \ldots, A_r$ be a basis of quasi-atoms over $J$.
    For each $i$ let $a_i$ be an element of $A_i\setminus J$.  Let
    $M'$ be a small model containing $M$ and the $a_i$'s.
    \begin{claim}
      Any $\varepsilon \in I_{M'}$ contracts $J$.
    \end{claim}
    \begin{claimproof}
      We may assume $\varepsilon \ne 0$.  Let $D = \varepsilon^{-1}
      \cdot I_{M'}$.  By
      Remark~\ref{basic-infs}.\ref{infinitesimal-times-standard}, $a_i
      \cdot \varepsilon \in I_{M'}$, or equivalently $a_i \in D$.
      Thus
      \begin{equation*}
        D \cap A_i \not \subseteq J
      \end{equation*}
      By Proposition~\ref{special-proposition}.\ref{guards}, $D
      \supseteq J$.  Then
      \begin{equation*}
        D \cap A_i \supsetneq J
      \end{equation*}
      so by Lemma~\ref{characterization-of-domination}, $D$ dominates
      $J$.  By
      Proposition~\ref{those-posets}.\ref{infinitesimals-below},
      $\varepsilon^{-1} \cdot J \supseteq \varepsilon^{-1} \cdot I_{M'} =
      D$.  Thus $\varepsilon^{-1} \cdot J$ dominates $J$.
    \end{claimproof}
    Let $\varepsilon$ be a realization of the partial type over $M'$
    asserting that $\varepsilon \in I_{M'}$ and $\varepsilon \notin X$ for
    any light $M'$-definable set $X$.  This type is consistent because
    $M'$-definable basic neighborhoods are heavy
    (Proposition~\ref{basic-nbhds}.\ref{heavy-nbhd}) and no heavy set
    is contained in a finite union of light sets
    (Theorem~\ref{heavy-light}).  Then $\varepsilon \in I_{M'} \subseteq
    I_J$.  As $I_J$ is $M$-invariant, every realization of
    $\tp(\varepsilon/M)$ is in $I_J$.  Let $Y$ be the $M$-definable set
    of realizations of $\tp(\varepsilon/M)$.  For any $M$-definable $X
    \supseteq Y$ we have
    \begin{equation*}
      I_M \subseteq X \mininf X \subseteq X - X.
    \end{equation*}
    Therefore $I_M \subseteq Y- Y$.  But $Y - Y \subseteq I_J - I_J = I_J$.
  \item If $a \in R_{J_1}$ and $a \in R_{J_2}$, then
    \begin{equation*}
      a \cdot (J_1 \cap J_2) = (a \cdot J_1) \cap (a \cdot J_2)
      \subseteq J_1 \cap J_2
    \end{equation*}
    so $a \in R_{J_1 \cap J_2}$.  The inclusion $I_{J_1} \cap I_{J_2}
    \subseteq I_{J_1 \cap J_2}$ is Lemma~\ref{contraction}.\ref{c-6}.
  \item First note that $1$ does not contract $J$.  Indeed,
    $\botrk(J/J) = 0 \ne r$.  Thus $1 \notin I_J$.  As $I_J$ is an
    ideal, it follows that $-1 \notin I_J$.
    \begin{claim}
      If $\varepsilon \in I_J$ then $\varepsilon/(1 + \varepsilon) \in I_J$.
    \end{claim}
    \begin{claimproof}
      We may assume $\varepsilon \ne 0$.  Using
      Lemma~\ref{contraction}.\ref{c-1} choose a basis
      $\{A_1,\ldots,A_r\}$ of quasi-atoms over $J$ such that $\varepsilon
      \cdot A_i \subseteq J$.  For each $i$ choose $a_i \in A_i
      \setminus J$.  Then $\varepsilon \cdot a_i \in J$, so $(1 +
      \varepsilon) \cdot a_i \in A_i \setminus J$.  Let $\beta = (1 +
      \varepsilon)/\varepsilon$.  Then
      \begin{align*}
        \beta \cdot (\varepsilon \cdot a_i) & \in \beta \cdot J \\
         (\beta \cdot \varepsilon) \cdot a_i = (1 + \varepsilon) \cdot a_i &
        \in A_i \setminus J.
      \end{align*}
      In particular
      \begin{equation*}
        (\beta \cdot J) \cap A_i \not \subseteq J,
      \end{equation*}
      for every $i$, so $\beta \cdot J \supseteq J$ by
      Proposition~\ref{special-proposition}.\ref{guards}.  Then
      $(\beta \cdot J) \cap A_i \supsetneq J$ for every $i$, so $\beta
      \cdot J$ dominates $J$ by
      Lemma~\ref{characterization-of-domination}.  This means that
      $\beta^{-1} = \varepsilon/(1 + \varepsilon)$ lies in $I_J$.
    \end{claimproof}
    Now if $\varepsilon \in I_J$, then
    \begin{equation*}
      \frac{1}{1 + \varepsilon} = 1 - \frac{\varepsilon}{1 + \varepsilon} \in 1
      + I_J \subseteq R_J. \qedhere
    \end{equation*}
  \end{enumerate}
\end{proof}
\begin{remark}
  Proposition~\ref{rings-and-ideals}.\ref{cap-not-vee} also holds for
  $R_{J_1 + J_2}$ and $I_{J_1 + J_2}$.
\end{remark}
\begin{speculation}
  In Proposition~\ref{rings-and-ideals}.\ref{ij-im}, not only is $I_M$
  a subset of $I_J$, it is a sub\emph{ideal} in the ring $R_J$.  One
  can probably prove this by first increasing $M$ to contain a
  non-zero element $j_0$ of $J$.  Then for any $\varepsilon \in I_M$
  and $a \in R_J$, we have
  \begin{equation*}
    \varepsilon \cdot a \cdot j_0 \in I_M \cdot R_J \cdot J \subseteq I_M \cdot J \subseteq I_M,
  \end{equation*}
  so $\varepsilon \cdot a \in j_0^{-1} I_M = I_M$.  Thus $I_M \cdot R_J
  \subseteq I_M$.  Then one can probably shrink $M$ back to the
  original model by the usual methods.

  Now, the important thing is that if $R_J$ were already a good Bezout
  domain, then there would be some valuations $\val_i : \Mm \to
  \Gamma_i$ and cuts $\Xi_i$ in $\Gamma_i$ such that
  \begin{equation*}
    I_M = \{x \in \Mm ~|~ \val_1(x) > \Xi_1 \wedge \val_2(x) > \Xi_2 \wedge \cdots\}.
  \end{equation*}
  By replacing $\Xi_i$ with $2 \Xi_i$ or $\Xi_i/2$ we could obtain
  alternate $M$-invariant subgroups smaller than $I_M$, contradicting
  the analogue of Corollary~\ref{minimal-heavy-subgroup} for invariant
  (not type-definable) subgroups.  Unless, of course, the $\Xi_i$ lie
  on the edges of convex subgroups.  Then, by coarsening, one could
  essentially arrive at a situation where $I_M$ is the Jacobson
  radical of an $M$-invariant good Bezout domain.  This would allow
  some of the henselianity arguments used in the dp-minimal case to
  directly generalize, though there are some additional complications
  in characteristic 0.\footnote{The details are worked out in
    \cite{prdf2}.}
\end{speculation}

\subsection{Mutation and the limiting ring} \label{sec:mutate}
The next two lemmas provide a way to ``mutate'' a special group $J$
and obtain a better special group $J'$ for which $R_{J'}$ is closer
than $R_J$ to being a good Bezout domain.
\begin{lemma}
  \label{twists}
  Let $J \in \mathcal{P}$ be special and non-zero.  Let $a_1, \ldots,
  a_n$ be elements of $\Mm^\times$.  Let $J' = J \cap a_1 \cdot J \cap
  a_2 \cdot J \cap \cdots \cap a_n \cdot J$.  Then $J'$ is special and
  non-zero, $R_J \subseteq R_{J'}$, and $I_J \subseteq I_{J'}$.
\end{lemma}
\begin{proof}
  By Proposition~\ref{special-proposition}.\ref{special-scale}, each
  $a_i \cdot J$ is special, so the intersection $J'$ is special by
  Lemma~\ref{special-lemma-2}.  It is nonzero by
  Proposition~\ref{those-posets}.\ref{nonzero-intersect}.  By
  Proposition~\ref{rings-and-ideals}.\ref{r-scale} we have $R_J =
  R_{a_i \cdot J}$ and $I_J = I_{a_i \cdot J}$ for each $i$.  Then the
  inclusions $R_J \subseteq R_{J'}$ and $I_J \subseteq I_{J'}$ follow
  by an interated application of
  Proposition~\ref{rings-and-ideals}.\ref{cap-not-vee}.
\end{proof}

Recall that $r$ is the reduced rank of $\mathcal{P}$.
\begin{lemma}
  \label{vandermonde}
  Let $J \in \mathcal{P}$ be special and non-zero.  Let $\alpha \in
  \Mm^\times$ be arbitrary.  Let $J' = J \cap (\alpha \cdot J) \cap
  \cdots \cap (\alpha^{r-1} \cdot J)$.  Let $q_0, q_1, \ldots, q_r$ be
  $r+1$ distinct elements of $M_0$.  Then there is at least one $i$
  such that $\alpha \ne q_i$ and
  \begin{equation*}
    \frac{1}{\alpha - q_i} \in R_{J'}.
  \end{equation*}
\end{lemma}
\begin{proof}
  For each $0 \le i \le r$ let
  \begin{align*}
    \alpha_i & := \alpha - q_i \\
    G_i & := \{x \in \Mm ~|~ \alpha_ix \in J \wedge \alpha_i^2 x \in
    J\wedge \cdots \wedge \alpha_i^rx \in J\} \\
    H_i & := J \cap G_i = \{x \in \Mm ~|~ x \in J \wedge \alpha_ix
    \in J \wedge \cdots \wedge \alpha_i^rx \in J\}.
  \end{align*}
  Also let
  \begin{align*}
    H = \{x \in \Mm ~|~ x \in J \wedge \alpha x \in J \wedge \cdots
    \wedge \alpha^r x \in J\}.
  \end{align*}
  \begin{claim} \label{claim-hi}
    $H_i = H$ for any $i$.
  \end{claim}
  \begin{claimproof}
    Note $\alpha = \alpha_i + q_i$.  If $x \in H_i$ then
    \begin{equation*}
      \alpha^n x = (\alpha_i + q_i)^n x = \sum_{k = 0}^n \binom{n}{k}
      \alpha_i^k q_i^{n-k} x \in J
    \end{equation*}
    for $0 \le n \le r$, because $\alpha_i^k x \in J$, $q_i^{n-k} \in
    M_0$, and $J$ is an $M_0$-vector space.  Thus $H_i \subseteq H$;
    the reverse inclusion follows by symmetry.
  \end{claimproof}
  Because the $q_i$ are distinct, the $(r+1) \times (r+1)$ Vandermonde
  matrix built from the $q_i$ is invertible.  Let $f : \Mm^{r+1} \to
  \Mm^{r+1}$ be the $\Mm$-linear map sending $(1,q_i,\ldots,q_i^r)$ to
  the $i$th basis vector.  Let $g : \Mm \to \Mm^{r+1}$ be the map
  \begin{equation*}
    g(x) = (x, \alpha x, \ldots, \alpha^r x).
  \end{equation*}
  \begin{claim}
    \label{vandermonde-key}
    The composition
    \begin{equation*}
      \Mm \stackrel{g}{\to} \Mm^{r+1} \stackrel{f}{\to} \Mm^{r+1}
      \twoheadrightarrow (\Mm/J)^{r+1}
    \end{equation*}
    has kernel $H$, and maps $G_i$ into $0^i \oplus (\Mm/J) \oplus
    0^{r-i}$.
  \end{claim}
  \begin{claimproof}
    The invertible matrix defining $f$ has coefficients in $M_0$, and
    $J$ is closed under multiplication by $M_0$, so $f$ maps $J^{r+1}$
    isomorphically to $J^{r+1}$.  Therefore,
    \begin{equation*}
      f(g(x)) \in J^{r+1} \iff g(x) \in J^{r+1} \iff x \in H,
    \end{equation*}
    where the second $\iff$ is the definition of $H$.  Now suppose $x
    \in G_i$.  Then $g(x) - (x, q_ix, \ldots, q_i^r x) \in J^{r+1}$.
    Indeed, for any $0 \le n \le r$ we have
    \begin{equation*}
      \alpha^n x = (\alpha_i + q_i)^n x = q_i^n x + \sum_{k = 1}^n \binom{n}{k} q_i^{n-k} (\alpha_i^k x),
    \end{equation*}
    and the sum is an element of $J$ by definition of $G_i$.  As $f$
    preserves $J^{r+1}$, it follows that
    \begin{equation*}
      f(g(x)) \equiv f(x, q_ix, \ldots, q_i^r x) = x \cdot e_i \pmod{
      J^{r+1}},
    \end{equation*}
    where $e_i$ is the $i$th basis vector.
  \end{claimproof}
  \begin{claim}
    \label{ult}
    If $(x_0, x_1, \ldots, x_r) \in G_0 \times \cdots \times G_r$ has
    $x_0 + \cdots + x_r \in H$, then each $x_i \in H$.
  \end{claim}
  \begin{claimproof}
    For $0 \le i \le r$ let $p_i : \Mm^{r+1} \to \Mm/J$ be the
    composition of the $i$th projection and the quotient map $\Mm \to
    \Mm/J$.  Claim~\ref{vandermonde-key} implies that
    \begin{align*}
      x \in H & \implies p_i(f(g(x))) = 0 \\
      x \in G_j &\implies p_i(f(g(x))) = 0 \qquad \text{if } i \ne j.
    \end{align*}
    Thus
    \begin{equation*}
      0 = p_i(f(g(x_0 + \cdots + x_r))) = p_i(f(g(x_i))).
    \end{equation*}
    As $p_j(f(g(x_i))) = 0$ for $j \ne i$, it follows that
    $p_j(f(g(x_i))) = 0$ for all $j$.  In other words, $f(g(x_i)) \in
    J^{r+1}$.  By Claim~\ref{vandermonde-key}, $x_i \in H$.
  \end{claimproof}
  Now Claim~\ref{ult} implies that the map
  \begin{align*}
    (G_0/H) \times \cdots \times (G_r/H) & \to \Mm/H \\
    (x_0,\ldots,x_r) & \mapsto x_0 + \cdots + x_r
  \end{align*}
  is injective.  The image is $D/H$ for some type-definable $D \in
  \mathcal{P}$, namely $D = G_0 + \cdots + G_r$.  Then the interval
  $[H^{r+1},G_0 \oplus \cdots \oplus G_r]$ in $\mathcal{P}_{r+1}$ is
  isomorphic to the interval $[H,D]$ in $\mathcal{P}_1$.  Thus
  \begin{equation*}
    r \ge \redrk(D/H) = \redrk(G_0/H) + \cdots + \redrk(G_r/H).
  \end{equation*}
  Therefore $G_i = H = H_i$ for at least one $i$.  By definition of
  $G_i$ and $H_i$, this means that
  \begin{equation}
    \alpha_i x \in J \wedge \cdots \wedge \alpha_i^r x \in J \implies
    x \in J \label{callback}
  \end{equation}
  for any $x \in \Mm$.  As $J \ne 0$, this implies $\alpha_i \ne 0$.
  Then (\ref{callback}) can be rephrased as
  \begin{equation}
    \alpha_i^{-1} \cdot J \cap \cdots \cap \alpha_i^{-r} \cdot J
    \subseteq J. \label{callback2}
  \end{equation}
  Define
  \begin{align*}
    J'' & := J \cap \alpha_i^{-1} J \cap \cdots \cap \alpha_i^{-(r-1)} J \\
    & = J \cap \alpha^{-1} J \cap \cdots \cap \alpha^{-(r-1)} J,
  \end{align*}
  where the second equality follows by the proof of
  Claim~\ref{claim-hi}.  By (\ref{callback2}),
  \begin{equation*}
    \alpha_i^{-1} \cdot J'' = \alpha_i^{-1} J \cap \cdots \cap
    \alpha_i^{-r} \stackrel{!}{\subseteq} J \cap \alpha_i^{-1} J \cap \cdots \cap \alpha_i^{-(r-1)} J =
    J''.
  \end{equation*}
  Therefore $\alpha_i^{-1} \in R_{J''}$.  But
  \begin{equation*}
    J' = J \cap \cdots \cap \alpha^{r-1} J = \alpha^{r-1} \cdot (J
    \cap \cdots \cap \alpha^{-(r-1)} J) = \alpha^{r-1} J''.
  \end{equation*}
  Thus, by Proposition~\ref{rings-and-ideals}.\ref{r-scale}
  \begin{equation*}
    \alpha_i^{-1} \in R_{J''} = R_{J'}. \qedhere
  \end{equation*}
\end{proof}

\begin{theorem}
  \label{bezout-theorem}
  Let $J \in \mathcal{P}_1$ be special, non-zero, and type-definable
  over $M \supseteq M_0$.  Then there is an $M$-invariant ring
  $R^\infty_J$ and ideal $I^\infty_J \lhd R^\infty_J$ satisfying the
  following properties:
  \begin{itemize}
  \item $R^\infty_J$ and $I^\infty_J$ are $M$-invariant.
  \item $(1 + I^\infty_J) \subseteq (R^\infty_J)^\times$, so
    $I^\infty_J$ is a subideal of the Jacobson radical of
    $R^\infty_J$.
  \item The $M$-infinitesimals $I_M$ are a subgroup of $I^\infty_J$
    (and therefore of the Jacobson radical).
  \item $M_0 \subseteq R^\infty_J$.
  \item $R^\infty_J$ is a Bezout domain with at most $r$ maximal
    ideals.
  \item The field of fractions of $R^\infty_J$ is $\Mm$.
  \end{itemize}
\end{theorem}
\begin{proof}
  Let $P$ be the set of finite $S \subseteq \Mm^\times$ such that $1
  \in S$.  Then $P$ is a commutative monoid with respect to the
  product $S \cdot S' = \{x \cdot y ~|~ x \in S, y \in S'\}$.  For any
  $S \in P$ and $G \in \mathcal{P}_1$, define
  \begin{equation*}
    G^S := \bigcap_{s \in S} s \cdot G.
  \end{equation*}
  Note that $(G^S)^{S'} = G^{S \cdot S'}$.  If $G$ is special and
  non-zero then by Lemma~\ref{twists} $G^S$ is special and non-zero,
  and there are inclusions $R_G \subseteq R_{G^S}$ and $I_G \subseteq
  I_{G^S}$.  Define sets
  \begin{align*}
    R_J^\infty & := \bigcup_{S \in P} R_{J^S} \\
    I_J^\infty & := \bigcup_{S \in P} I_{J^S}.
  \end{align*}
  These sets are clearly $M$-invariant.  Moreover, the unions are
  directed: given any $S$ and $S'$ we have
  \begin{align*}
    R_{J^S} \cup R_{J^{S'}} &\subseteq R_{J^{S \cdot S'}} \\
    I_{J^S} \cup I_{J^{S'}} &\subseteq I_{J^{S \cdot S'}}.
  \end{align*}
  Therefore $R^\infty_J$ is a ring and $I^\infty_J$ is an ideal.  The
  fact that $(1 + I^\infty_J) \subseteq (R^\infty_J)^\times$ also
  follows (using Proposition~\ref{rings-and-ideals}.\ref{jacobson}).
  Taking $S = \{1\}$, we see that $I_J \subseteq I^\infty_J$.
  Proposition~\ref{rings-and-ideals}.\ref{ij-im} says $I_M \subseteq
  I_J$, so $I_M \subseteq I^\infty_J$ as desired.  Similarly, $M_0
  \subseteq R_J \subseteq R^\infty_J$.
  \begin{claim} \label{key-to-bezout}
    If $q_0, q_1, \ldots, q_r$ are distinct elements of $M_0$ and
    $\alpha \in \Mm^\times$, then at least one of $1/(\alpha - q_i)$
    is in $R^\infty_J$.
  \end{claim}
  \begin{claimproof}
    By Lemma~\ref{vandermonde}, at least one of $1/(\alpha - q_i)$
    lies in $R_{J^S}$ for $S = \{1, \alpha, \ldots, \alpha^{r-1}\}$.
  \end{claimproof}
  It follows formally that $R^\infty_J$ is a Bezout domain with no
  more than $r$ maximal ideals.  Let $a, b$ be two elements of
  $R^\infty_J$.  We claim that the ideal $(a,b)$ is principal.  This
  is clear if $a = 0$ or $b = 0$.  Otherwise, let $\alpha = a/b$.  As
  $M_0$ is infinite, Claim~\ref{key-to-bezout} implies that
  \begin{equation*}
    \frac{b}{a - q b} = \frac{1}{\frac{a}{b} - q} \in R^\infty_J
  \end{equation*}
  for some $q \in M_0$.  Then the principal ideal $(a - qb) \lhd
  R^\infty_J$ contains $b$, hence $qb$ and thus $a$.  Therefore $(a -
  qb) = (a, b)$.

  Next, we show that $R^\infty_J$ has at most $r$ maximal ideals.
  Suppose for the sake of contradiction that there were distinct
  maximal ideals $\mathfrak{m}_0, \ldots, \mathfrak{m}_r$ in
  $R^\infty_J$.  As $R^\infty_J$ is an $M_0$-algebra, each quotient
  $R^\infty_J/\mathfrak{m}_i$ is a field extending $M_0$.  Take
  distinct $q_0, \ldots, q_r \in M_0$, and find an element $x \in
  R^\infty_J$ such that $x \equiv q_i \pmod{\mathfrak{m}_i}$ for each
  $i$, by the Chinese remainder theorem.  Then $x - q_i \in
  \mathfrak{m}_i \subseteq R^\infty_J \setminus (R^\infty_J)^\times$
  for each $i$.  So $1/(x - q_i)$ does not lie in $R^\infty_J$ for any
  $0 \le i \le r$, contrary to Claim~\ref{key-to-bezout}.

  Lastly, note that if $x$ is any element of $\Mm^\times$, then $1/(x
  - q) \in R^\infty_J$ for some $q \in M_0$, $q \ne x$.  As $q \in M_0
  \subseteq R^\infty_J$, the field of fractions of $R^\infty_J$
  contains $x$.  So the field of fractions must be all of $\Mm$.
\end{proof}

\subsection{From Bezout domains to valuation rings}
We double-check some basic facts about Bezout domains.
\begin{remark}
  \label{bezout-basics}
  Let $R$ be a Bezout domain.
  \begin{enumerate}
  \item For each maximal ideal $\mathfrak{m}$, the localization
    $R_{\mathfrak{m}}$ is a valuation ring on the field of fractions
    of $R$.
  \item $R$ is the intersection of the valuation rings
    $R_{\mathfrak{m}}$.
  \end{enumerate}
\end{remark}
\begin{proof}
  \begin{enumerate}
  \item Given non-zero $a, b \in R$, we must show that $a/b$ or $b/a$
    lies in $R_{\mathfrak{m}}$.  If $c$ is such that $(c) = (a,b)$,
    then we may replace $a$ and $b$ with $a/c$ and $b/c$, and assume
    that $(a,b) = (1) = R$.  Both of $a$ and $b$ cannot lie in
    $\mathfrak{m}$, and so either $a/b$ or $b/a$ is in
    $R_{\mathfrak{m}}$.
  \item Given $a \in \bigcap_{\mathfrak{m}} R_{\mathfrak{m}}$, we will
    show $a \in R$.  Let $I = \{x \in R ~|~ ax \in R\}$.  This is an
    ideal in $R$.  If $I$ is proper, then $I \le \mathfrak{m}$ for
    some $\mathfrak{m}$.  As $a \in R_{\mathfrak{m}}$ we can write $a
    = b/s$ where $b \in R$ and $s \in R \setminus \mathfrak{m}$.  Then
    $s \in R$ and $as = b \in R$, so $s \in I \le \mathfrak{m}$, a
    contradiction.  Therefore $I$ is an improper ideal, so $1 \in I$,
    and $a = a \cdot 1 \in R$. \qedhere
  \end{enumerate}
\end{proof}

\begin{theorem}
  \label{any-valuation-at-all}
  Let $\Mm$ be a sufficiently saturated dp-finite field, possibly with
  extra structure.  Suppose $\Mm$ is not of finite Morley rank.  Then
  there is a small set $A \subseteq \Mm$ and a non-trivial
  $A$-invariant valuation ring.
\end{theorem}
\begin{proof}
  Take $M_0$ as usual in this section.  By
  Proposition~\ref{special-proposition} there is a non-zero special $J
  \in \mathcal{P}_1$.  The group $J$ is type-definable over some small
  $M \supseteq M_0$.  Let $R$ be the $R^\infty_J$ of
  Theorem~\ref{bezout-theorem}.  Then $R$ is an $M$-invariant Bezout
  domain with at most $r$ maximal ideals, the Jacobson radical of $R$
  is non-zero (because it contains $I_M$), and $\Frac(R) = \Mm$.  Let
  $\mathfrak{m}_1, \ldots, \mathfrak{m}_k$ enumerate the maximal
  ideals of $R$.  Let $\mathcal{O}_i$ be the localization
  $R_{\mathfrak{m}_i}$.  By Remark~\ref{bezout-basics}, each
  $\mathcal{O}_i$ is a valuation ring on $\Mm$, and
  \begin{equation*}
    R = \mathcal{O}_1 \cap \cdots \cap \mathcal{O}_k.
  \end{equation*}
  At least one $\mathcal{O}_i$ is non-trivial; otherwise $R = \Mm$ and
  has Jacobson radical 0.\footnote{Tracing through the proof, here is
    what explicitly happens. If $\varepsilon \in I_M$ then
    $-1/\varepsilon$ cannot be in $R^\infty_J$, or else $\varepsilon
    \in I_M \subseteq I_J \subseteq I_{J^S} \lhd R_{J^S}$ and
    $-1/\varepsilon \in R_{J^S}$ for large enough $S$, so that $-1 \in
    I_{J^S}$, contradicting
    Proposition~\ref{rings-and-ideals}.\ref{jacobson}.}  Without loss
  of generality $\mathcal{O}_1$ is non-trivial.  By the Chinese
  remainder theorem, choose $a \in R$ such that $a \equiv 1 \pmod{
  \mathfrak{m}_1}$ and $a \equiv 0 \pmod{\mathfrak{m}_i}$ for $i \ne
  1$.  We claim that $\mathcal{O}_1$ is $\Aut(\Mm/aM)$-invariant.  If
  $\sigma \in \Aut(\Mm/aM)$, then $\sigma \in \Aut(\Mm/M)$ so $\sigma$
  preserves $R$ setwise.  It therefore permutes the finite set of
  maximal ideals.  As $\mathfrak{m}_1$ is the unique maximal ideal not
  containing $a$, it must be preserved (setwise).  Therefore $\sigma$
  preserves the localization $\mathcal{O}_1$ setwise.
\end{proof}
\begin{remark}
  Stable fields do not admit non-trivial invariant valuation rings
  (\cite{prdf2}, Lemma~2.1).  Consequently,
  Theorem~\ref{any-valuation-at-all} can be used to give an extremely
  roundabout proof of Halevi and Palac\'in's theorem that stable
  dp-finite fields have finite Morley rank (\cite{Palacin},
  Proposition~7.2).
\end{remark}

\section{Shelah conjecture and classification} \label{sec:the-end}
\begin{proposition} \label{proto-shelah}
  Let $K$ be a sufficiently saturated dp-finite field of positive
  characteristic.  Then one of the following holds:
  \begin{itemize}
  \item $K$ has finite Morley rank (and is therefore finite or
    algebraically closed).
  \item $K$ admits a non-trivial henselian valuation.
  \end{itemize}
\end{proposition}
\begin{proof}
  This is Theorem~\ref{invariant-valuations-are-enough} and
  \ref{any-valuation-at-all}.
\end{proof}

\begin{lemma} \label{o-infty}
  Let $K$ be a sufficiently saturated dp-finite field of positive
  characteristic.  Assume $K$ is infinite.  Let $\mathcal{O}_\infty$
  be the intersection of all $K$-definable valuation rings on $K$.
  Then $\mathcal{O}_\infty$ is a henselian valuation ring on $K$ whose
  residue field is algebraically closed.
\end{lemma}
\begin{proof}
  The proof for dp-minimal fields (Theorem 9.5.7 in \cite{myself})
  goes through without changes, using Lemma~\ref{hensel-key},
  Theorem~\ref{henselianity-conjecture}, and
  Proposition~\ref{proto-shelah}.  Additionally, we must rule out the
  possibility that the residue field is real closed or finite.  The
  first cannot happen because we are in positive characteristic.  The
  second cannot happen because $K$ is Artin-Schreier closed, a
  property which transfers to the residue field.
\end{proof}

\begin{corollary} \label{cor-infty}
  Let $K$ be a sufficiently saturated infinite dp-finite field of
  positive characteristic.  If every definable valuation on $K$ is
  trivial, then $K$ is algebraically closed.
\end{corollary}

\begin{corollary} \label{shelah-conjecture}
  Let $K$ be a dp-finite field of positive characteristic.  Then one
  of the following holds:
  \begin{itemize}
  \item $K$ is finite.
  \item $K$ is algebraically closed.
  \item $K$ admits a non-trivial definable henselian valuation.
  \end{itemize}
\end{corollary}
\begin{proof}
  Suppose $K$ is neither finite nor algebraically closed.  Let $K'
  \succeq K$ be a sufficiently saturated elementary extension.  Then
  $K'$ is neither finite nor algebraically closed.  By
  Corollary~\ref{cor-infty} there is a non-trivial definable valuation
  $\mathcal{O} = \phi(K',a)$ on $K'$.  The statement that $\phi(x;a)$
  cuts out a valuation ring is expressed by a 0-definable condition on
  $a$, so we can take $a \in \dcl(K)$.  Then $\phi(K,a)$ is a
  non-trivial valuation ring on $K$, henselian by
  Theorem~\ref{henselianity-conjecture}.
\end{proof}
So the Shelah conjecture holds for dp-finite fields of positive
characteristic.

By Proposition~3.9, Remark~3.10, and Theorem~3.11 in
\cite{halevi-hasson-jahnke}, this implies the following classification
of dp-finite fields of positive characteristic: up to elementary
equivalence, they are exactly the Hahn series fields $\mathbb{F}_p((\Gamma))$
where $\Gamma$ is a dp-finite $p$-divisible group.  Dp-finite ordered
abelian groups have been algebraically characterized and are the same
thing as strongly dependent ordered abelian groups
(\cite{sd-groups-dolich-goodrick}, \cite{sd-groups-farre},
\cite{sd-groups-halevi-hasson}).

\begin{acknowledgment}
  The author would like to thank
\begin{itemize}
\item Meng Chen, for hosting the author at Fudan University, where
  this research was carried out.
\item Jan Dobrowolski, who provided some helpful references.
\item John Goodrick, for suggesting that the ideal of finite sets
  could be replaced with other ideals in the construction of
  infinitesimals.
\item Franziska Jahnke, for convincing the author that the proof in
  \S\ref{sec:henselianity} was valid.
\item Shichang Song, who invited the author to give a talk on
  dp-minimal fields.  The ideas for
  \S\ref{sec:broad-narrow}-\ref{sec:heavy-light} arose while preparing
  that talk.
\item The UCLA model theorists, who read parts of this paper and
  pointed out typos.
\end{itemize}
{\tiny This material is based upon work supported by the National Science
Foundation under Award No. DMS-1803120.  Any opinions, findings, and
conclusions or recommendations expressed in this material are those of
the author and do not necessarily reflect the views of the National
Science Foundation.}
\end{acknowledgment}

\bibliographystyle{plain} \bibliography{mybib}{}

\end{document}